\documentclass[11pt,UTF-8,reqno]{amsart}
\usepackage{enumerate, bbm}
\usepackage{amssymb,url,color, booktabs}
\usepackage{mathrsfs}
\usepackage{cite}
\usepackage[utf8]{inputenc}

\usepackage{color}
\usepackage[colorlinks=true]{hyperref}
\hypersetup{
    linkcolor=blue,          
    citecolor=red,        
    filecolor=blue,      
    urlcolor=cyan
}

\usepackage{color}
\definecolor{MyDarkBlue}{cmyk}{0.8,0.3,0.8,0.4}
\definecolor{yellow}{rgb}{0.99,0.99,0.70}
\definecolor{white}{rgb}{1.0,1.0,1.0}
\definecolor{black}{rgb}{0.00,0.00,0.00}

\numberwithin{equation}{section}

\newcommand{\be}{\begin{eqnarray}}
\newcommand{\ee}{\end{eqnarray}}
\newcommand{\ce}{\begin{eqnarray*}}
\newcommand{\de}{\end{eqnarray*}}
\newtheorem{theorem}{Theorem}[section]
\newtheorem{lemma}{Lemma}[section]

\newtheorem{definition}{Definition}[section]
\newtheorem{proposition}{Proposition}[section]
\newtheorem{example}{Example}[section]
\newtheorem{corollary}{Corollary}[section]

\newtheorem{rem}[theorem]{Remark}

\def\[{{\Big[}}
\def\]{{\Big]}}
\def\<{{\langle}}
\def\>{{\rangle}}
\def\({{\big(}}
\def\){{\big)}}

\def\min{{\mathord{{\rm min}}}}

\def\bb2{{\boldsymbol{2}}}

\def\={&\!\!=\!\!&}

\def\b1{{\mathbbm 1}}

\def\geq{\geqslant}
\def\leq{\leqslant}
\def\ge{\geqslant}
\def\le{\leqslant}

\def\[{{\Big[}}
\def\]{{\Big]}}
\def\<{{\langle}}
\def\>{{\rangle}}

\def\min{{\mathord{{\rm min}}}}

\def\={&\!\!=\!\!&}
\def\bt{\begin{theorem}}
\def\et{\end{theorem}}
\def\bl{\begin{lemma}}
\def\el{\end{lemma}}
\def\br{\begin{remark}}
\def\er{\end{remark}}
\def\bd{\begin{definition}}
\def\ed{\end{definition}}
\def\bc{\begin{corollary}}
\def\ec{\end{corollary}}

\def\geq{\geqslant}
\def\leq{\leqslant}
\def\ge{\geqslant}
\def\le{\leqslant}

\def\<{\langle} \def\>{\rangle}

\allowdisplaybreaks

\begin{document}

\title[Yosida Approximation and Multi-Valued SEI]
{Generalized Yosida approximation and Multi-Valued Stochastic Evolution Inclusions$^\dagger$}


\thanks{$\dagger$
This work is supported by National Key R\&D program of China (No. 2023YFA1010101). The research of W. Hong is also supported by  NSFC (No.~12401177) and  NSF of Jiangsu Province (No.~BK20241048).  The research of W. Liu is also supported by NSFC (No.~12171208, 12090011,12090010). }

\maketitle
\centerline{\scshape Wujing Fan$^a$, Wei Hong$^b$,   Wei Liu$^{b}$\footnote{~Corresponding author: weiliu@jsnu.edu.cn} }

\medskip

\vspace{1mm}
 {\footnotesize
\centerline{  $a.$ School of Statistics and Data Science, Nankai University, Tianjin 300071, China}}
	
\vspace{1mm}
{\footnotesize
\centerline{ $b.$ School of Mathematics and Statistics, Jiangsu Normal University, Xuzhou 221116, China}}

\begin{abstract}
In this paper, we generalize the classical Yosida approximation by utilizing a nonstandard duality mapping to establish the existence and uniqueness of both (probabilistically) weak and strong solutions and demonstrate the continuous dependence on initial values for a class of  multi-valued stochastic evolution inclusions within the variational framework. 

Furthermore,  leveraging this generalized Yosida approximation,  we derive the finite-time extinction of solutions with probability one and also provide an explicit upper bound of the moment of extinction time for multi-valued stochastic evolution inclusions perturbed by linear multiplicative noise.  The main results are applicable to various examples, including multi-valued stochastic porous media equations, stochastic $\Phi$-Laplace equations and  stochastic evolution inclusions involving subdifferentials.


\end{abstract}
\maketitle \rm

\tableofcontents

\section{Introduction}
The multi-valued stochastic evolution inclusion (SEI), which is given by the following form
 \begin{equation}\label{ge1}
	dX(t)\in \mathcal{A}(t,X(t))dt+\sigma(t,X(t))dW(t),
\end{equation}
is closely linked to a wide range of physical phenomena and models, such as  self-organized criticality,  fluid flow in porous media, the dynamics of non-Newtonian fluids and phase transition processes (cf.~e.g.~\cite{BDR091,BDR12,BR12,BRR15,G15,GT14} and references therein).

 As far as we know, the first result concerning  multi-valued SEIs was obtained by Kr{\'e}e  \cite{kree82}. Subsequently, C\'epa \cite{Cpa1998ProblmeDS} demonstrated that the Skorokhod problem is equivalent to a multi-valued maximal monotone operator equation and established an existence result for its generalized form. Petterson \cite{Pettersson95} employed the Yosida approximation technique to establish the existence of solutions to SEIs involving a maximal-monotone operator. Zhang \cite{Zhang2007SkorohodPA} extended C{\'e}pa's results to infinite-dimensional settings and  relaxed the Lipschitz continuity condition on the single-valued  operator to a global monotonicity assumption.  Ren et al.~\cite{Ren12,RWZ10} studied  the long-time behaviour and the regularity of invariant measures to  multi-valued stochastic differential equations (SDEs). Rascanu et al. \cite{BR97,R81} developed the framework of stochastic variational inequalites (SVIs), introducing the notion of weak (or SVI) solutions for multi-valued SEIs. This approach has  been widely employed in the analysis of multi-valued SEIs; see, for instance,  \cite{BDR09,GR15,BR13,GT16} and references therein. More recently, Gess et al.~\cite{GRW24} established the existence and uniqueness results for SVI solutions to stochastic nonlinear (possibly multi-valued) diffusion equations driven by multiplicative noise, where the drift operator $L$ is the generator of a transient Dirichlet form.



On the other hand, the theory of maximal-monotone operators plays a crucial role in the  analysis of nonlinear partial differential equations. In this paper, we focus on  SEIs in infinite-dimensional spaces, where the multi-valued mapping $\mathcal{A}$ in \eqref{ge1} is decomposed into a combination of a multi-valued maximal-monotone operator and a single-valued pseudo-monotone operator. To study this class of SEIs, we employ the variational framework which will be introduced below.


\subsection{Variational framework}

Let $H$ be a separable Hilbert space, with its inner product and norm denoted by $\left\langle\cdot,\cdot\right\rangle_H$ and $\|\cdot\|_H$, respectively. The dual space of $H$ is denoted by $H^*$, which is identified with $H$  through the Riesz isomorphism. Let $V$ and its dual space $V^*$ be uniformly convex Banach spaces such that the embedding $V\subset H$ is compact, continuous, and dense. Under these conditions, $V$ and $V^*$ are also strictly convex and reflexive (cf. \cite{barbu2010nonlinear}). The norms of $V$ and $V^*$ are donated by $\|\cdot\|_V$ and $\|\cdot\|_{V^*}$, respectively. Furthermore, the dual pairing between $v\in V^*$ and $u\in V$ is represented by $\left\langle v,u \right\rangle$. Through the Riesz isomorphism, we construct the following Gelfand triple
\begin{equation}\label{gelfand}
V\subset H\simeq H^*\subset V^*.
\end{equation}
Consequently, we have
$$\left\langle z,v\right\rangle=\left\langle z,v\right\rangle_H,~~\text{for~all~}z\in H,~v\in V. $$
Let $(\Omega,\mathcal{F},\mathbb{P})$ be a complete probability space equipped with a  normal filtration $\{\mathcal{F}_t\}_{ t\geq0}$. Let $(U,\left<\cdot,\cdot\right>_U)$ be a separable Hilbert space, and let $(L_2(U,H),\|\cdot\|_{\mathcal{L}_2})$ denote the space of all Hilbert-Schmidt operators from  $U$ to $H$.

\vspace{1mm}
Let $T>0$. In this paper we focus on the following type  multi-valued SEIs
\begin{align}\label{general}
		dX(t)+B(t,X(t))dt\in -\mathcal{A}(t,X(t))dt+\sigma(t,X(t))dW(t),~X(0)= x\in H,
\end{align}
where the multi-valued (maximal-monotone) mapping $\mathcal{A}$ and the single-valued mappings $B,\sigma$ are defined as follows
\begin{align*}
	&\mathcal{A}: [0,T]\times V \to 2^{V^*},\\
	&B: [0,T]\times V \to V^*,\\
	&\sigma: [0,T]\times V \to L_2(U,H).
\end{align*}
Here $W(t)$ denotes a cylindrical Wiener process on $U$ w.r.t.~the filtration $\{\mathcal{F}_t\}_{ t\geq0}$. Furthermore, the mappings $\mathcal{A},B,\sigma$ are assumed to be measurable.

The classical variational framework and the theory of monotone operators were first introduced by Minty \cite{M62}. This foundational work was further developed by Browder \cite{browder63,browder64}, Leray and Lions \cite{LL65} as well as Hartman and Stampacchia \cite{HS66}. Afterwards,  the variational approach for stochastic partial differential equations (SPDEs) was originally developed by Pardoux \cite{Pardoux72,Pardoux75}. This approach was subsequently explored  by Krylov and Rozovskii \cite{Krylov79} and Gyongy \cite{Gyongy}. Based on these developments, Liu and R\"ockner \cite{liu10} and Liu \cite{liu11} extended the variational framework to encompass  a class of locally monotone operators, broadening the range of nonlinear SPDEs, including e.g.~stochastic 2D Navier-Stokes equations, that can be analyzed. Very recently, R{\"o}ckner et al.~\cite{rockner2022well} further generalized the results of \cite{liu10,liu11} to SPDEs with multiplicative noises and fully local monotonicity conditions. For additional applications and advancements of the variational approach to SPDEs, we refer the reader to \cite{LR15,prevot2007concise,Ren07,rockner08,rockner09,zhang09}.

Notice that all the aforementioned works focus on the single-valued case, there are much less results in the literature that can be applied to multi-valued SEIs. As previously mentioned, Zhang \cite{Zhang2007SkorohodPA} investigated the multi-valued SEIs (\ref{general}) within the Gelfand triple (\ref{gelfand}). However,  the multi-valued operator $\mathcal{A}$  in their study is restricted to the space $H$ (i.e.,~$\mathcal{A}: H \to 2^{H}$), which excludes some model of  nonlinear SEIs, such as multi-valued stochastic porous media equations.
Liu and Stephan \cite{liu2014yosida} made a progress by establishing the existence and uniqueness of solutions to a class of multi-valued SEIs (\ref{general}) driven by L\'evy noises. In their work,  the multi-valued operator $\mathcal{A}$ is assumed to be coercive with an exponent $\alpha \in (1,2]$, and the single-valued operator $B$ is required to be  Lipschitz continuous. Gess and T\"olle \cite{GT14} studied multi-valued, monotone, weakly dissipative SEIs, which extend to cases where the coercivity exponent $\alpha\in [1,2]$, thereby covering certain stochastic singular diffusion equations.

However, the well-posedness result of  multi-valued SEIs within the variational framework for coercivity  exponent $\alpha>2$  remains an open problem. We would like to highlight the main challenges in studying multi-valued SEIs with general coercivity  exponent $\alpha>1$. Specifically, the classical  Yosida approximation (cf.~\cite{barbu2010nonlinear}), commonly employed to handle maximal-monotone operators, only works for the case of $\alpha\in (1,2]$ due to certain technical limitations, as pointed out in \cite{liu2014yosida}. In particular, the coercivity of the Yosida approximation operators fails to hold when $\alpha>2$  (cf.~Lemma 3.10 in \cite{liu2014yosida}). Consequently, to address the general case $\alpha>2$, several new techniques involving the generalized  Yosida approximation are required.

More precisely,  we  generalize the classical Yosida approximation by employing a novel  duality mapping (cf.~Definition \ref{defJ} and (\ref{Alamd})), which is also of independent interest in analysis.  To achieve this, we first establish the existence and uniqueness of solutions to an associated resolvent equation  (cf.~(\ref{res}) below).
Subsequently, we derive several key properties of the  generalized Yosida approximation, which are crucial for  the convergence of the approximating sequence.

To  develop a more general theory for  the existence of (probabilistically) weak solutions to multi-valued SEIs (\ref{general}), we assume the single-valued  operator $B$ to be pseudo-monotone.
While the monotonicity technique is effective for SPDEs with single-valued monotone operators (cf.~\cite{LR15}), they fail in addressing pseudo-monotone operators. To solve this problem, we combine stochastic compactness arguments with  pseudo-monotonicity techniques. However, due to the joint perturbation caused by multi-valued maximal-monotone operators,  several non-trivial difficulties arise in proving the convergence of the coefficient $B(t,X_{\lambda}(t))$ and the Yosida approximation operator $\mathcal{A}_{\lambda}(t,X_{\lambda}(t))$, as detailed in Lemmas \ref{lemmaB1} and \ref{lemc}. Furthermore, we establish the existence of a unique (probabilistically) strong solution and demonstrate the continuous dependence on initial values for (\ref{general}) under fully local monotonicity conditions on single-valued operators.
\subsection{Finite-time extinction}
Self-organized criticality (SOC) is a widely studied framework in physics  for modeling complex systems,  including phenomena such as earthquakes and neuronal activity. The SOC behavior observed in the Bak-Tang-Wiesenfeld (BTW) sandpile model corresponds to the finite-time extinction of solutions to multi-valued stochastic diffusion equations (cf.~\cite{BDR091}). Similarly, the Stefan problem, which describes phase transitions in porous media or heat conduction, can be formulated as a class of multi-valued stochastic porous media equations, see also \cite{BDR091}. The time evolution of these systems exhibits SOC, underscoring the fundamental role of finite-time extinction in the analysis of multi-valued SEIs.

Several previous works have explored the finite-time extinction of solutions to multi-valued SEIs perturbed by linear multiplicative noises. For example, Barbu et al.~\cite{BDR091} proved that, for multi-valued stochastic porous media equations in  one spatial dimension,  finite-time extinction  occurs with positive probability for small initial values. In their  subsequent work \cite{BDR092}, the restriction to one dimension was relaxed, extending the results to all dimensions.
In \cite{BR12}, the  asymptotic extinction of solutions to multi-valued stochastic porous media equations was established with probability one in dimensions $d=\{1,2,3\}$. R\"ockner and Wang \cite{RW13}  generalized the results of \cite{BDR091} to encompass a broader class of operators, including fractional Laplacians. They proved the exponential integrability of the extinction time for the Zhang model, provided the noise is small enough. Gess \cite{G15} showed that the finite-time extinction of the BTW model holds with probability one in all dimensions.

Concerning finite time extinction, we distinguish the following situations:
\begin{enumerate}[(i)]

\item
Extinction with positive probability for small initial conditions;

\vspace{1mm}
\item
Extinction with positive probability for all initial values;

\vspace{1mm}
\item
Finite time extinction with probability one for all initial values.
\end{enumerate}
While both Situations (i) and (ii) are mathematically intriguing,  the robustness of the relaxation into subcritical states in SOC is of fundamental importance in physics. Consequently,   Situation (iii) aligns most closely with the perspective of SOC.

In  this  work, we establish finite-time extinction with probability one for  all initial values $x\in H$ (i.e.,~Situation (iii)) for a class of multi-valued SEIs perturbed by linear multiplicative noise. As a consequence in the proof, we demonstrate the $L^2$-convergence of the Yosida approximation sequence by utilizing the powerful characterization of convergence in probability as given by Gy\"{o}ngy and Krylov \cite{GK96} (see Theorem \ref{finiconv} for details). Based on this,  we provide a quantitative characterization of the probability that the extinction time $\tau_e$ is less than any given time $T$, i.e.
\begin{equation*}
\mathbb{P}(\tau_e\leq T)\geq 1-\frac{c^*\|x\|_{H}^{2-\alpha}}{T},
\end{equation*}
and also derive an explicit upper bound of the moment of $\tau_e$, i.e.
$$\mathbb{E}\tau_e\leq c^*\|x\|_H^{2-\alpha},$$
  for multi-valued SEIs in the variational framework with any coercivity exponent $\alpha \in (1,2)$. Moreover, in comparison to \cite{BDR091, BDR092, BDR12,BRR15},  our  results extend the finite-time extinction from cases with positive probability and small initial conditions to  cases with probability one for all initial values, and we also remove the restriction on spatial dimension.

To illustrate the generality of the present framework, we apply the main results to several concrete examples. The first application concerns the multi-valued stochastic porous media equations
\begin{align*}
	\left\{
	\begin{aligned}
		dX(t) &\in \Delta\Psi (X(t))dt+\sigma(t, X(t))dW(t),\\
		X(0) &= x\in H^{-1}(\Lambda).
	\end{aligned}
	\right.
\end{align*}
A typical example of the multi-valued mapping $\Psi$ is given by
\begin{align*}
    \Psi(s)=\left\{
	\begin{aligned}
		&\rho+ \delta s^{p-1}, \ \text{if~}s>0,\\
		&[-\rho,\rho], \ \text{if~}s= 0,\\
        &-\rho+\delta (-s)^{p-1}, \ \text{if~}s<0,
	\end{aligned}
	\right.
\end{align*}
where  $\rho,\delta$ are some positive constants. This example was involved in \cite{liu2014yosida} for $p\in(1,2]$. In this work, we extend the well-posedness result  to all $p>1$ and further establish finite-time extinction of solutions with probability one  for all initial values when $p\in (1,2)$.

Next, we  investigate the multi-valued stochastic $\Phi$-Laplace equations
\begin{align}\label{eq0001}
	\left\{
	\begin{aligned}
		dX(t) &\in \text{div}\Phi(\nabla X(t))dt+\sigma(t, X(t))dW(t),\\
		X(0) &= x\in L^2(\Lambda),
	\end{aligned}
	\right.
\end{align}
which have been studied in e.g.~\cite{GT14,SST23}  for  specific potential functions $\Phi$ (see Subsection \ref{exlap} for details). We develop a general well-posedness result  and also for the first time establish the upper bound of the moment estimate for extinction time to  (\ref{eq0001}), which is of independent interest.

Finally, we apply our results to  SEIs with subdifferentials
\begin{align*}
	\left\{
	\begin{aligned}
		\partial_t u(t,x)&\in  \left[\Delta u(t,x)-g(t,x, u(t,x),\nabla
   u(t,x))-\partial\varphi(u(t,x))\right]dt+\sigma(t,u(t,x))dW(t),\\
		u(0,x)&=u_0(x)\in L^2(\Lambda).
	\end{aligned}
	\right.
\end{align*}
This type of SEI with single-valued monotone operators was previously studied in  \cite{Zhang2007SkorohodPA}. Within our framework, we extend the results of \cite{Zhang2007SkorohodPA} to more general operators, broadening their applicability to a wider class of nonlinear parabolic equations.

The rest of the paper is structured as follows: Section 2 gives some preliminaries, including the definition and properties of the maximal-monotone operator and the generalized Yosida approximation. Section 3 presents the weak and strong well-posedness results. Section 4 discusses the finite-time extinction of solutions to multi-valued SEIs. Section 5 illustrates several concrete applications of our general framework. Finally, we provide several useful results related to multi-valued maximal-monotone operators in the Appendix.

\section{Preliminaries}
In this section, we recall some definitions and basic results regarding maximal monotone and pseudo monotone operators on Banach spaces.     Then, we construct a nonlinear version of  Yosida approximation for maximal monotone operators on Banach spaces by solving a newly formulated resolvent equation.

We denote the power set of $V^*$ by $2^{V^*}$. A multi-valued mapping $\mathcal{A}:V\to 2^{V^*}$  maps $x\in V$ to a subset $\mathcal{A}(x)\subset V^*$. The domain of  $\mathcal{A}$ is defined as the set $\mathcal{D}(\mathcal{A}):=\left\{x\in V \mid \mathcal{A}(x)\neq\emptyset\right\}$. For a multi-valued map $\mathcal{A}$, its graph is given by $$\mathcal{G}(\mathcal{A}):=\left\{[x,y]\in V\times V^* \mid y\in \mathcal{A}(x)\right\}.$$
\subsection{Maximal-monotone and pseudo-monotone operators}
First, we recall the definition of maximal-monotone operators.
\begin{definition}
$(i)$ A multi-valued operator $\mathcal{A}: V\to 2^{V^*}$ is called monotone if
	\begin{align*}
		\left\langle v_1 - v_2, u_1 - u_2 \right\rangle &\ge 0, \quad\text{for~all}~[v_i, u_i] \in \mathcal{G}(\mathcal{A}), \ i=1,2.
	\end{align*}

$(ii)$ A monotone operator $\mathcal{A}: V\to 2^{V^*}$ is called maximal-monotone if it is not properly contained in any other monotone extension $\tilde{\mathcal{A}}$ such that $\mathcal{G}(\mathcal{A})\subsetneq\mathcal{G}(\tilde{\mathcal{A}}).$

\vspace{1mm}
$(iii)$ The minimal section $\mathcal{A}^0 :\mathcal{D}(\mathcal{A})\subset V \to 2^{V^*}$ of a maximal-monotone operator $\mathcal{A}$ is defined for $x \in \mathcal{D}(\mathcal{A})$ as
    \begin{align*}
    	\mathcal{A}^0(x):=\Big\{y\in \mathcal{A}(x) \mid \|y\|_{V^*} = \min_{z\in \mathcal{A}(x)} \|z\|_{V^*}\Big\}.
    \end{align*}

\end{definition}

\begin{rem}
We note that since $V^*$ is strictly convex, then $\mathcal{A}^0$ is single-valued $($cf.~\cite{barbu2010nonlinear}$)$.
\end{rem}

The following lemma concerns the convergence of maximal-monotone operators.
\begin{lemma}$($cf.~\cite[Lemma 2.3]{barbu2010nonlinear}$)$\label{propmaxi}	Let $\mathcal{A}: V\to 2^{V^*}$ be maximal-monotone. Let $[u_n, v_n]\in \mathcal{G}(\mathcal{A})$ be such that $u_n \rightharpoonup u$, $v_n \rightharpoonup v,$ and either
\begin{equation*}
	\limsup\limits_{n,m \to \infty}\left\langle u_n-u_m, v_n-v_m \right\rangle \le  0
\end{equation*}
or
\begin{equation*}
	\limsup\limits_{n \to \infty}\left\langle u_n-u, v_n-v \right\rangle \le 0.
\end{equation*}
Then, $[u,v]\in \mathcal{G}(\mathcal{A})$ and $\left\langle u_n, v_n \right\rangle \to \left\langle u, v \right\rangle$, as $n\to \infty.$
\end{lemma}

In the sequel, we introduce the definition of duality mapping, which differs from the classical one presented in the existing works (e.g.~\cite{barbu2010nonlinear})  and plays an important role in applications to examples of  multi-valued nonlinear stochastic evolution inclusions.
\begin{definition}\label{defJ}
	Let $\alpha>1$. The duality mapping $J: V \to 2^{V^*}$ on the space $V$ is defined by
\begin{align*}
	J(u):=\left\{v \in V^* \mid \left\langle v, u \right\rangle = \| u \|_V^{\alpha} = \| v \|_{V^*}^{\frac{\alpha}{\alpha-1}} \right\}, \quad\text{for~each}~u \in V.
\end{align*}
\end{definition}

\begin{proposition}
$(i)$ The duality mapping $J$ is monotone, locally bounded and coercive.

        \vspace{1mm}
$(ii)$  $J$ is single-valued, demicontinuous and odd.
\end{proposition}
\begin{proof}
\textbf{Proof of (i)}:
Let $u_1, u_2\in V$ and $v_1\in J(u_1), v_2\in J(u_2)$. In view of Definition \ref{defJ} and $\alpha>1$,  we have
\begin{align}\label{Jpro1}
    \left\langle v_1-v_2, u_1-u_2 \right\rangle&\ge \left\langle v_1, u_1 \right\rangle+\left\langle v_2, u_2 \right\rangle-\|v_1\|_{V^*}\|u_2\|_V-\|v_2\|_{V^*}\|u_1\|_V\nonumber\\
    &=\|u_1\|_V^{\alpha}+\|u_2\|_V^{\alpha}-\|u_1\|_V^{\alpha-1}\|u_2\|_V-\|u_2\|_V^{\alpha-1}\|u_1\|_V\nonumber\\
    &=(\|u_1\|_V^{\alpha-1}-\|u_2\|_V^{\alpha-1})(\|u_1\|_V-\|u_2\|_V)\nonumber\\
    &\ge 0.
\end{align}
Thus, $J$ is monotone.

Let $u\in V$ and $v\in J(u)$. Since $\alpha>1$, $\|v\|_{V^*}=\|u\|_V^{\alpha-1}$
and $\left\langle v,u \right\rangle = \|u\|_V^{\alpha}$ , we deduce that $J$ is locally bounded and coercive.

\vspace{1mm}
\textbf{Proof of (ii)}: The proof is postponed in Proposition \ref{app1} in the Appendix.
\end{proof}

\begin{rem}
Note that, applying the Hahn-Banach theorem, $J(u)$ is non-empty for every $u\in V$, which implies that $\mathcal{D}(J)=V$. Recall that $J$ is a single-valued mapping.
Since $J$ is monotone, demicontinuous and coercive, it follows from Theorem 26.A in \cite{ZachariasB} that $J$ is surjective. Consequently, the mapping $J: V\to V^*$ is a bijection.

Furthermore,  since $J$ is monotone and demicontinuous, we can conclude that $J$ is maximal-monotone by Proposition 32.7 in \cite{ZachariasB}. Additionally, since $V^*$ is uniformly convex,  it implies from Proposition 21.23 $(d)$ in \cite{ZachariasB} that $J$ is continuous.
\end{rem}

\begin{definition}
The inverse mapping $J^{-1}:V^* \to 2^V$ defined by
$$J^{-1}(v): =\left\{u\in V \mid v\in J(u)\right\}$$ satisfies
\begin{align*}
	J^{-1}(v) = \left\{ u \in V \mid \left\langle v, u \right\rangle = \| u \|_V^{\alpha} = \| v \|_{V^*}^{\frac{\alpha}{\alpha-1}} \right\}, \quad\text{for~each}~ v\in V^*.
\end{align*}
\end{definition}

\begin{rem}
Since $V$ is reflexive,  $J^{-1}$ is the duality mapping on $V^*$ and $\mathcal{D}(J^{-1})=V^*$. Moreover, due to the strict convexity of $V$, we know that $J^{-1}$ is also single-valued $($from $V^*$ to $V$$)$ and demicontinuous.
\end{rem}

In what follows, we also recall the definition of pseudo-monotone operators, which was first proposed by Br\'ezis in \cite{Brezis68}.
\begin{definition}
	The operator $B:V\to V^*$ is called pseudo-monotone if $u_n\rightharpoonup u$ in $V$ and
	$$
	\limsup_{n\to \infty}\langle B(u_n),u_n-u\rangle \le 0,
	$$
	then for any $v\in V$,
	$$
	\langle B(u),u-v\rangle \le \liminf_{n\to \infty}\langle B(u_n),u_n-v\rangle.
	$$
\end{definition}
\begin{rem}
	There is an alternative definition  of pseudo-monotonicity introduced by Browder in \cite{browder77}. Specifically, an operator $B:V\to V^*$ is pseudo-monotone iff $u_n\rightharpoonup u$ in $V$ and
	$$
	\limsup\limits_{n\to \infty}\langle B(u_n),u_n-u\rangle \le 0,
	$$
	then $B(u_n)\rightharpoonup B(u)$ in $V^*$ and
$$
	\lim\limits_{n\to\infty}\langle B(u_n), u_n\rangle = \langle B(u),u\rangle.
	$$
It has been proved that these two definitions are equivalent, which can be referred to Remark 5.2.12 in \cite{LR15}.
\end{rem}

\subsection{Generalized Yosida approximation}
To construct the generalized Yosida approximation of maximal-monotone operators on Banach spaces, we first consider the following resolvent equation
\begin{align}
	0 \in J(x_{\lambda}-x)+\lambda \mathcal{A}(x_{\lambda}),\label{res}
\end{align}
where the operator $\mathcal{A}$ is maximal-monotone, $\lambda >0$ and $x\in V$.

\begin{rem}
It should be pointed out that since the definition of the duality mapping $J$ differs from
the classical one $($see e.g.~\cite{barbu2010nonlinear}$)$, the well-posedness of the resolvent equation $(\ref{res})$ and the properties of the generalized Yosida approximation require to be established.
\end{rem}

The existence and uniqueness of solutions to Eq.~(\ref{res}) are presented as follows.
\begin{proposition}\label{propres} For all $x\in V$, there exists a unique solution $x_{\lambda}$ to \eqref{res}.
\end{proposition}
\begin{proof}
Recall that $V$ is a reflexive Banach space, $\mathcal{A}:V\to 2^{V^*}$ is maximal-monotone and $J:V\to V^*$ is demicontinuous and monotone. It follows from Corollary 2.6 in \cite{barbu2010nonlinear} that for any $\lambda>0$ and $x\in V$, the operator $$J(\cdot-x)+\lambda \mathcal{A}$$
is maximal-monotone. Furthermore, since $J$ is coercive and $\mathcal{A}$ is monotone, we deduce that $J(\cdot-x)+\lambda \mathcal{A} $ is also coercive for any $\lambda>0$ and $x\in V$. Thus, by Corollary 2.2 in \cite{barbu2010nonlinear},  there exists $x_\lambda\in V$ such that
$$0 \in J(x_{\lambda}-x)+\lambda \mathcal{A}(x_{\lambda}).$$

Now, we show the uniqueness of solutions to Eq.~(\ref{res}) by contradiction. We assume that the solutions of Eq.~\eqref{res} are not unique, i.e.,  there exist $x_{\lambda}^1,x_{\lambda}^2$ that satisfy Eq.~\eqref{res}. Then, there exist $y_{\lambda}^1\in \mathcal{A}(x_{\lambda}^1)$ and $y_{\lambda}^2\in \mathcal{A}(x_{\lambda}^2)$ such that
$$J(x_{\lambda}^1-x)+\lambda y_{\lambda}^1=J(x_{\lambda}^2-x)+\lambda y_{\lambda}^2.$$
It follows that
\begin{equation}\label{es0}
\langle J(x_{\lambda}^1-x)-J(x_{\lambda}^2-x),x_{\lambda}^1-x-(x_{\lambda}^2-x)\rangle=-\lambda\langle y_{\lambda}^1-y_{\lambda}^2, x_{\lambda}^1-x_{\lambda}^2\rangle\le0.
\end{equation}
On the other hand, by the definition of $J$ we have
\begin{eqnarray}\label{es2}
\!\!\!\!\!\!\!\!&&\langle J(x_{\lambda}^1-x)-J(x_{\lambda}^2-x),x_{\lambda}^1-x-(x_{\lambda}^2-x)\rangle
\nonumber \\
\!\!\!\!\!\!\!\!&&\ge (\|x_{\lambda}^1-x\|_V^{\alpha-1}-\|x_{\lambda}^2-x\|_V^{\alpha-1})(\|x_{\lambda}^1-x\|_V-\|x_{\lambda}^2-x\|_V)
\nonumber \\
\!\!\!\!\!\!\!\!&&\ge 0.
\end{eqnarray}
Combining $(\ref{es0})$ and $(\ref{es2})$, we can deduce that
$$
\langle J(x_{\lambda}^1-x), x_{\lambda}^2-x \rangle=\langle J(x_{\lambda}^2-x), x_{\lambda}^1-x\rangle=\|x_{\lambda}^1-x\|_V^{\alpha}=\|x_{\lambda}^2-x\|_V^{\alpha}.
$$
In view of the definition of $J$ and the fact that $J$ is single-valued, it follows that $$J(x_{\lambda}^1-x)=J(x_{\lambda}^2-x).$$
Finally, since $J^{-1}$ is single-valued, we conclude that
$$x_{\lambda}^1=x_{\lambda}^2.$$
The proof is complete.
\end{proof}

Now, let $\mathcal{R}_{\lambda}: V \to V$ be the resolvent operator of $\mathcal{A}$, which is defined by
$$\mathcal{R}_{\lambda}(x) := x_{\lambda},$$
where $x_{\lambda}$ is the unique solution of \eqref{res}.
Then, the generalized Yosida approximation $\mathcal{A}_{\lambda}: V\to V^*$ of $\mathcal{A}$ is defined by
	\begin{equation}\label{Alamd}
		\mathcal{A}_{\lambda}(x):= \frac{1}{\lambda}J(x-\mathcal{R}_{\lambda}(x)),
	\end{equation}
where $\lambda >0~\text{and}~x\in V$. By  Definition \ref{defJ} and  Proposition \ref{propres}, it is known that for each $x\in V$ and $\lambda > 0$, the mapping $\mathcal{A}_{\lambda}(\cdot)$ is single-valued.

\begin{rem}
    Since Definition \ref{defJ} extends the classical  duality mapping, which reduces to the special case of $\alpha=2$, the generalized Yosida approximation (\ref{Alamd}) offers a more natural extension compared to  the classical Yosida approximation.
    Specifically, for any $\alpha>1$, if the maximal monotone operator $\mathcal{A}$ is coercive, its generalized Yosida approximation operator $\mathcal{A}_{\lambda}$ retains coercivity. In contrast, the classical Yosida approximation is limited to $\alpha\in (1,2]$, as  demonstrated in Lemma 3.10 in \cite{liu2014yosida}. Furthermore, the current definition is pivotal in establishing the convergence of $\mathcal{A}_{\lambda}$ for the general case  $\alpha>1$, as illustrated in the proof of Lemma \ref{lemc}.
\end{rem}
\vspace{1mm}
The properties of the Yosida approximation (\ref{Alamd}) are collected in the following.
\begin{proposition}\label{propyosgj}\
The generalized Yosida approximation $\mathcal{A}_{\lambda}$ has the following properties:

\vspace{1mm}
$(i)$ $\mathcal{A}_{\lambda}$ is monotone, bounded on bounded sets, and demicontinuous from $V$ to $V^*$.

\vspace{1mm}
$(ii)$ $\|\mathcal{A}_{\lambda}(x)\|_{V^*}\le \|\mathcal{A}^0(x)\|_{V^*}$, for each $x\in \mathcal{D}(\mathcal{A})$ and $\lambda >0.$
	
\vspace{1mm}
$(iii)$  $\mathcal{A}_{\lambda}(x) \to \mathcal{A}^0(x)$ in $V^*$, as $\lambda\to0$, for each $x\in \mathcal{D}(\mathcal{A}).$
	
\vspace{1mm}
$(iv)$ For any $x\in V$,  $\mathcal{A}_{\lambda}(x) \in \mathcal{A}(\mathcal{R}_{\lambda}(x)).$

\end{proposition}
\begin{proof}
Since $J$ is odd, it follows from \eqref{res}, \eqref{Alamd}, and the definition of $\mathcal{R}_{\lambda}$ that $\mathcal{A}_{\lambda}(x)\in \mathcal{A}(\mathcal{R}_{\lambda}(x))$ for all $x\in V$, which yields claim (iv). We now proceed to prove claims (i), (ii), and (iii) in order.

\vspace{1mm}
\textbf{Proof of (i):} First, it clear that $\mathcal{A}_{\lambda}:V\to V^*$ is monotone and demicontinuous.
Let $[u,v]\in\mathcal{G}(\mathcal{A})$. By the definition of $\mathcal{A}_{\lambda}(x)$ and the monotonicity of $\mathcal{A}$, we can obtain
$$
\langle J(\mathcal{R}_{\lambda}(x)-x), \mathcal{R}_{\lambda}(x)-u\rangle\le\lambda\langle v,u-\mathcal{R}_{\lambda}(x)\rangle,
$$
which implies
\begin{eqnarray*}
\!\!\!\!\!\!\!\!&&\|\mathcal{R}_{\lambda}(x)-x\|_V^{\alpha}
\nonumber \\
\le \!\!\!\!\!\!\!\!&& \|x-u\|_V\|\mathcal{R}_{\lambda}(x)-x\|_V^{\alpha-1}+\lambda\|x-u\|_V\|v\|_{V^*}+\lambda\|v\|_{V^*}\|\mathcal{R}_{\lambda}(x)-x\|_V
\nonumber \\
\le \!\!\!\!\!\!\!\!&&\frac{1}{2}\|\mathcal{R}_{\lambda}(x)-x\|_V^{\alpha}+C\Big(\|x-u\|_V^{\alpha}+\|x-u\|_V\|v\|_{V^*}+\|v\|_{V^*}^{\frac{\alpha}{\alpha-1}}\Big).
\end{eqnarray*}
Then, it follows that the operators $\mathcal{R}_{\lambda}$ and $\mathcal{A}_{\lambda}$ are bounded on bounded sets.

\vspace{1mm}
\textbf{Proof of (ii):} Let $[x,y]\in \mathcal{G}(\mathcal{A})$. By the monotonicity of $\mathcal{A}$, we have
$$
0\le \langle y- \mathcal{A}_{\lambda}(x), x-\mathcal{R}_{\lambda}(x)\rangle \le \|y\|_{V^*}\|x-\mathcal{R}_{\lambda}(x)\|_{V}-\frac{1}{\lambda}\|x-\mathcal{R}_{\lambda}(x)\|_V^{\alpha}.
$$
It follows that
$$
\lambda\|\mathcal{A}_{\lambda}(x)\|_{V^*}=\|x-\mathcal{R}_{\lambda}(x)\|_V^{\alpha-1}\le \lambda \|y\|_{V^*},~\forall y\in \mathcal{A}(x).
$$
Thus, we deduce that for any $x\in \mathcal{D}(\mathcal{A})$ and $\lambda>0$,
$$
\|\mathcal{A}_{\lambda}(x)\|_{V^*}\le  \|\mathcal{A}^0(x)\|_{V^*}.
$$

\vspace{1mm}
\textbf{Proof of (iii):}
The proof of claim (iii) is similar to that of Proposition 2.2 (v) in \cite{barbu2010nonlinear}, we outline it here for  reader's convenience.
Let $x\in \mathcal{D}(\mathcal{A})$ and consider $\lambda\to 0$ such that
$$\mathcal{A}_{\lambda}(x)\rightharpoonup y~~\text{in}~~V^*.$$
As shown in the proof of Proposition 2.2 (iv) in \cite{barbu2010nonlinear}, we have $y\in \mathcal{A}(x)$. Furthermore, since $\|\mathcal{A}_{\lambda}(x)\|_{V^*}\le \|\mathcal{A}^0(x)\|_{V^*}$, it implies that $y=\mathcal{A}^0(x)$.  Since $V^*$ is uniformly convex, by Lemma 1.1 in \cite{barbu2010nonlinear} it follows that $\mathcal{A}_{\lambda}(x)\to\mathcal{A}^0(x)$ in $V^*$. This completes the proof.
\end{proof}

\section{Well-posedness}

\subsection{Main results}
In this subsection, we show the existence and uniqueness of weak and strong solutions and the continuous dependence on initial values to Eq.~\eqref{general}.
We first recall the definition of (probabilistically) weak solutions.
\begin{definition}\label{weaksolution}
 A couple $(X,\eta,W)$ is called a (probabilistically) weak solution to Eq.~\eqref{general}, if there exists a stochastic basis $(\Omega,\mathcal{F},\{\mathcal{F}_t\}_{t\in[0,T]},\mathbb{P})$ such that $X\in L^{\alpha}([0,T]\times\Omega,V)\cap L^2([0,T]\times\Omega,H)$, $\eta\in L^1([0,T]\times\Omega,V^*)$ and  $W$ is an $U$-valued cylindrical Wiener process on $(\Omega,\mathcal{F},\{\mathcal{F}_t\}_{t\in[0,T]},\mathbb{P})$, which  satisfy

\vspace{1mm}
$(i)$ $X\in C([0,T],H)$ $\mathbb{P}$-a.s.;
	
\vspace{1mm}
$(ii)$ $X$ and $\int_{0}^{\cdot}\eta(s)ds$ are $(\mathcal{F}_t)$-adapted;
	
\vspace{1mm}
$(iii)$  $\eta \in \mathcal{A}(\cdot,X(\cdot))$ $dt\otimes\mathbb{P}\text{-a.e.};$
	
\vspace{1mm}
$(iv)$ for all $t\in[0,T]$, the following equality holds in $V^*$
	\begin{equation*}
		X(t) = x - \int_{0}^{t}( \eta(s)+B(s,\bar{X}(s)))ds + \int_{0}^{t}\sigma(s,\bar{X}(s))dW(s)~\mathbb{P}\text{-a.s.},
	\end{equation*}
	where $\bar{X}$ is any $V$-valued progressively measurable $dt\otimes \mathbb{P}$-version of $X$.

\end{definition}

\vspace{1mm}
We assume that there exist $f\in L^1([0,T],[0,\infty))$, $C,\delta>0$, $\beta\geq 0$ and $\alpha>1$ such that the following conditions hold for a.e.~$t \in [0,T]$.

\vspace{1mm}
More precisely, for the multi-valued operator $\mathcal{A}$, we suppose that

\vspace{1mm}
\begin{enumerate}
	\item [(\textbf{H}$_{\mathcal{A}}^1$)] (Maximal monotonicity): For any $x,y \in V$,
	$$
	\left<v-w,x-y\right>\ge 0, \quad \text{for~any}~v\in \mathcal{A}(t,x), w\in \mathcal{A}(t,y),
	$$
	and $\mathcal{A}(t,\cdot)$ is maximal-monotone.
	\item [(\textbf{H}$_{\mathcal{A}}^2$)] (Coercivity): For any $x\in V$ and $v\in \mathcal{A}(t,x)$,
	$$
	\left<v,x\right>\ge \delta\|x\|_V^{\alpha}-f(t).
	$$
	\item[(\textbf{H}$_{\mathcal{A}}^3$)] (Growth): For  any $x\in V$,
	$$
	\|\mathcal{A}^0(t,x)\|_{{V}^*}^{\frac{\alpha}{\alpha-1}}\le (f(t)+C\|x\|_V^{\alpha})(1+\|x\|_H^{\beta}).
	$$
\end{enumerate}

\vspace{1mm}
For mappings $B$ and $\sigma$, we suppose that

\vspace{1mm}
\begin{enumerate}
	
        \item[(\textbf{H}$_B^1$)] (Weak coercivity): For any $x\in V$,
	$$
	2\left<B(t,x),x\right> \ge -f(t)(1+\|x\|_H^2).
	$$
	\item[(\textbf{H}$_B^2$)] (Growth): For  any $x\in V$,
	$$
	\|B(t,x)\|_{{V}^*}^{\frac{\alpha}{\alpha-1}}\le (f(t)+C\|x\|_V^{\alpha})(1+\|x\|_H^{\beta}).
	$$
	\item[(\textbf{H}$_{\sigma}$)] For any sequence $\{x_n\}_{n=1}^{\infty}$ and $x$ in $V$ satisfying $\|x_n-x\|_H \to 0,$
	\begin{equation*}
		\|\sigma(t,x_n)-\sigma(t,x)\|_{\mathcal{L}_2}  \to 0.	
	\end{equation*}
	Moreover,  for any $x\in V$,
	\begin{equation*}
		\|\sigma(t,x)\|_{\mathcal{L}_2}^2\le f(t)(1+\|x\|_H^2).
	\end{equation*}
\end{enumerate}

In the sequel, we state the first main result of this work concerning the existence of weak solutions to Eq.~\eqref{general}.
\begin{theorem}\label{theoremzy11}
Assume that the operator $B(t,\cdot)$ is pseudo-monotone for a.e.~$\ t\in [0,T]$, and that $(\mathbf{H}_{\mathcal{A}}^1)$-$(\mathbf{H}_{\mathcal{A}}^3)$, $(\mathbf{H}_B^1)$, $(\mathbf{H}_B^2)$ and $(\mathbf{H}_{\sigma})$ hold. Then for any $x\in H$, there exists a $($probabilistically$)$ weak solution to Eq.~\eqref{general}. In addition, for any $p\ge 2$, the following moment estimate holds
\begin{align}
	\mathbb{E}\Big[\sup\limits_{t\in [0,T]}\|X(t)\|_H^p\Big]+\mathbb{E}\left\{\left(\int_{0}^{T}\|X(s)\|_V^{\alpha}ds\right)^{\frac{p}{2}}\right\}<\infty.\label{genest1}
\end{align}

\end{theorem}
\begin{rem}
     $(i)$ In previous works such as \cite{Zhang2007SkorohodPA,liu2014yosida},  the analysis was restricted that the single-valued operator $B(t,u)$ satisfies the  Lipschitz continuity or the global monotonicity. Our result provides a more general existence theorem for weak solutions to SEIs  by incorporating multi-valued maximal-monotone operators and single-valued pseudo-monotone operators within the variational framework. This extension allows for a broader class of nonlinear SEIs, as illustrated in Section~\ref{sec5}.

     $(ii)$ Theorem \ref{corollary1} also generalizes the results in \cite{GT14} by permitting a general polynomial growth condition for the maximal-monotone operator $($i.e.,~condition $(\mathbf{H}_{\mathcal{A}}^3)$$)$, whereas  \cite{GT14} is restricted to the linear growth case. This generalization encompasses a wider class of multi-valued stochastic $\Phi$-Laplace equations, such as $p$-Laplace equations $($with $p>1$$)$ and Non-Newtonian fluid models, as detailed in Subsection \ref{exlap}.
\end{rem}

Next, we consider the existence and uniqueness of (probabilistically) strong solutions to Eq.~\eqref{general}. The definitions of the (probabilistically) strong solutions and the pathwise uniqueness are presented as follows.
\begin{definition}\label{strongsolution}
 A couple $(X,\eta)$ is called a (probabilistically) strong solution to Eq.~\eqref{general}, if for every probability space $(\Omega,\mathcal{F},\{\mathcal{F}_t\}_{t\in[0,T]},\mathbb{P})$ with an $U$-valued cylindrical Wiener process $W$,  we have $X\in L^{\alpha}([0,T]\times\Omega,V)\cap L^2([0,T]\times\Omega,H)$ and $\eta\in L^1([0,T]\times\Omega,V^*)$ satisfy $(i)$-$(iv)$ in Definition \ref{weaksolution}.
\end{definition}

\begin{definition}\label{pathwiseunique}
    We say that pathwise uniqueness holds for Eq.~\eqref{general}, if whenever the probability space $(\Omega,\mathcal{F},\{\mathcal{F}_t\}_{t\in[0,T]},\mathbb{P})$ and the $U$-valued cylindrical Wiener process $W$ are fixed, two solutions $X$ and $X'$ such that $X(0)=X'(0)$ $\mathbb{P}$-a.s., then $\mathbb{P}$-a.s.
    $$
    X(t)=X'(t),\ t\in [0,T].
    $$
\end{definition}

In order to guarantee the existence and uniqueness of (probabilistically) strong solutions to Eq.~\eqref{general}, we assume the hemicontinuity and local monotonicity conditions instead of the pseudo-monotonicity as follows.

\begin{enumerate}
    \item[$(\mathbf{H}_B^3)$] (Hemicontinuity): The map $s \mapsto \langle B(t,x+sy),v\rangle$ is continuous on $\mathbb{R}$ for any $x,y,v\in V$.

\vspace{1mm}
	 \item[$(\mathbf{H}_B^4)$] (Local monotonicity): For any $x,y\in V$,
	$$
	-2\left<B(t,x)-B(t,y),x-y\right>+\|\sigma(t,x)-\sigma(t,y)\|_{\mathcal{L}_2}^2\le (f(t)+\rho(x)+\zeta(y))\|x-y\|_H^2,
	$$
where $\rho$ and $\zeta$ are  measurable functions from $V$ to $\mathbb{R}_{+}$ satisfying
	$$
	\rho(x)+\zeta(x)\le C(1+\|x\|_V^{\alpha})(1+\|x\|_H^{\beta}).
	$$
\end{enumerate}

The following discusses the pseudo-monotonicity of $B$ under assumptions $(\mathbf{H}_B^3)$ and $(\mathbf{H}_B^4)$, whose proof can refer to \cite{liu11}.
\begin{lemma}\label{lemmapseudo1}
	Assume that $(\mathbf{H}_B^3)$ and $(\mathbf{H}_B^4)$ hold. Then $B(t,\cdot)$ is pseudo-monotone from $V$ to $V^*$ for a.e.~$t\in [0,T].$
\end{lemma}

Now, we state the existence and uniqueness of (probabilistically) strong solutions to Eq.~\eqref{general}.
\begin{theorem}\label{corollary1}
	Assume that  $(\mathbf{H}_{\mathcal{A}}^1)$-$(\mathbf{H}_{\mathcal{A}}^3)$, $(\mathbf{H}_B^1)$-$(\mathbf{H}_B^4)$ and $(\mathbf{H}_{\sigma})$ hold.
  Then for any $x\in H$, there exists a unique $($probabilistically$)$ strong solution to Eq.~\eqref{general} and the estimate \eqref{genest1} holds.
\end{theorem}

The continuous dependence on  initial data is also given as follows.
\begin{theorem}\label{theoremzy21}
 Suppose that the assumptions in Theorem \ref{corollary1} hold.
  Let $\{x_n\}_{n=1}^{\infty}$ and $x$ be a sequence with $\|x_n-x\|_H \to 0$ and $X(t,x)$ be the unique solution of Eq.~\eqref{general} with initial value $x$. Then, for any $p>0$,
\begin{align*}
	\lim\limits_{n\to \infty}\mathbb{E}\Big[\sup_{t\in [0,T]}\|X(t,x_n)-X(t,x)\|_H^p\Big]=0.
\end{align*}
\end{theorem}

\begin{rem}
   The existence and uniqueness of solutions to multi-valued SEIs  were established in \cite{liu2014yosida}. Theorems  \ref{corollary1} and \ref{theoremzy21} extend this result by allowing arbitrary coercivity exponents $\alpha>1$, whereas \cite{liu2014yosida} is limited to $\alpha\in (1,2]$. Moreover, we establish the well-posedness  of multi-valued SEIs with single-valued fully local monotone operators, relaxing the Lipschitz continuity assumption imposed in \cite{liu2014yosida}. These extensions enable the unified treatment for a class of multi-valued stochastic porous media equations, thereby generalizing the results in \cite{BDR091,BDR092} $($see Subsection~\ref{mspme}$)$.
\end{rem}

\begin{rem}
	Note that Theorem \ref{theoremzy21} implies the Feller property of the transition semigroup $P_t:C_b(H)\to C_b(H)$ associated to Eq.~\eqref{general} with the time-homogenous coefficients $($$\mathcal{A},B,\sigma$ are independent of $t$$)$. Building upon this result, in forthcoming work, we aim to investigate the existence and uniqueness of invariant probability measures for Eq.~\eqref{general} by employing the generalized Yosida approximation developed in this work.
\end{rem}

\subsection{Approximating sequences}
In order to prove the existence of (probabilistically) weak solutions to Eq.~(\ref{general}), we consider the following approximating equations
\begin{equation}
	\left\{	
	\begin{aligned}
		&dX_{\lambda}(t)+\left[\mathcal{A}_{\lambda}(t,X_{\lambda}(t))+B(t,X_{\lambda}(t))\right]dt= \sigma(t,X_{\lambda}(t))dW(t),\\
		&X_{\lambda}(0)=x\in H,
	\end{aligned}
    \right.\label{subequ1}
\end{equation}
where $\mathcal{A}_{\lambda}$ is the generalized Yosida approximation of $\mathcal{A}$ given by (\ref{Alamd}).

The following lemma gives the existence of (probabilistically) weak solutions to Eq.~(\ref{subequ1}).
\begin{lemma}\label{propyos1}
Suppose that the assumptions in Theorem \ref{theoremzy11} hold. For any $\lambda\in\left(0,\delta^{-1}\right)$ and initial value $x\in H$, there exists a weak solution $X_\lambda$ to Eq.~\eqref{subequ1}.
\end{lemma}

\begin{proof}
It suffices to justify  the assumptions in Corollary 2.7 in \cite{rockner2022well}. Let $\lambda\in\left(0,\delta^{-1}\right)$. Denote
$$A:=-\mathcal{A}_{\lambda}-B.$$
Since $\mathcal{A}_{\lambda}(t,\cdot)$ is monotone and demicontinuous and $B(t,\cdot)$ is pseudo-monotone for a.e.~$t \in [0,T]$, it follows from Proposition 27.7 (c) in \cite{ZachariasB} that $B(t,\cdot) + \mathcal{A}_{\lambda}(t,\cdot)$ is pseudo-monotone for a.e.~$t \in [0,T]$. According to Lemma \ref{lemmacoe1},  the Yosida approximation $\mathcal{A}_{\lambda}$  is coercive. From this, using $(\mathbf{H}_B^1)$ and $(\mathbf{H}_{\sigma})$, we deduce that for any $v\in V$ and a.e.~$t\in[0,T]$,
\begin{align}\label{coel}
2\left<A(t,v),v\right>+\|\sigma(t,v)\|_{\mathcal{L}_2}^2 \le& Cf(t)(1+\|v\|_H^2)-\delta 2^{-\alpha+1}\|v\|_V^{\alpha}.	
\end{align}
Thus, Hypothesis (H3) in \cite{rockner2022well} is satisfied. Since $\mathcal{D}(\mathcal{A})=V$ and $\|\mathcal{A}_{\lambda}(\cdot,x)\|_{{V}^*}\le \|\mathcal{A}^0(\cdot,x)\|_{{V}^*}$ for all $x\in \mathcal{D}(\mathcal{A})$, the conditions $(\mathbf{H}_{\mathcal{A}}^3)$ and $(\mathbf{H}_B^2)$ ensure that for any $v\in V$ and $\alpha>1$, there exist constants $\beta\geq 0$ and $C>0$ such that
\begin{align}\label{ABgrowth}
\|A(t,v)\|_{{V}^*}^{\frac{\alpha}{\alpha-1}}\le \big(f(t)+C\|v\|_V^{\alpha}\big)\big(1+\|v\|_H^{\beta}\big).
\end{align}
Therefore, Hypothesis (H4) in \cite{rockner2022well} is also satisfied. This completes the proof.
\end{proof}
\begin{rem}
	Since $B(t,\cdot)+\mathcal{A}_{\lambda}(t,\cdot)$ is pseudo-monotone for a.e.~$t\in [0,T]$ and locally bounded, we find that $B(t,\cdot)+\mathcal{A}_{\lambda}(t,\cdot)$ is demicontinuous $($hence, hemicontinuous$)$ for a.e.~$t\in [0,T]$ by Proposition 27.7 $($b$)$ in \cite{ZachariasB}, which implies  Hypothesis $(H1)$ in \cite{rockner2022well} directly.
\end{rem}

\begin{lemma}\label{propgj1}
 There exists a constant $C_T>0$ such that for any $0<\lambda<\delta^{-1}$ we have
\begin{align*}
\mathbb{E}\Big[\sup_{t\in[0, T]}\|X_{\lambda}(t)\|_H^p\Big]+\mathbb{E}\int_{0}^{T}\|X_{\lambda}(s)\|_V^{\alpha}\|X_{\lambda}(s)\|_H^{p-2}ds \le C_T(1+\|x\|_H^p).
\end{align*}
\end{lemma}

\begin{proof}
By It\^o's formula we have
\begin{align*}
	\|X_{\lambda}(t)\|_H^p&=\|x\|_H^p-\frac{p}{2}\int_{0}^{t}\|X_{\lambda}(s)\|_H^{p-2} \big[2\left<\mathcal{A}_{\lambda}(s,X_{\lambda}(s))+B(s,X_{\lambda}(s)),X_{\lambda}(s)\right>\nonumber\\
	&\quad+\|\sigma(s,X_{\lambda}(s))\|_{\mathcal{L}_2}^2\big]ds+\frac{p(p-2)}{2}\int_{0}^{t}\|X_{\lambda}(s)\|_H^{p-4}\|\sigma(s,X_{\lambda}(s))^*X_{\lambda}(s)\|_{U}^2ds\nonumber\\
	&\quad+p\int_{0}^{t}\|X_{\lambda}(s)\|_H^{p-2}\left<X_{\lambda}(s), \sigma(s,X_{\lambda}(s))dW(s)\right>_H.
\end{align*}
Then by (\ref{coel}) and let $c=\delta2^{-\alpha+1}>0$, it follows that
\begin{align}\label{es1}
	&\|X_{\lambda}(t)\|_H^p+\frac{pc}{2}\int_{0}^{t}\|X_{\lambda}(s)\|_V^{\alpha}\|X_{\lambda}(s)\|_H^{p-2}ds\nonumber\\
	\le& \|x\|_H^p+C\int_{0}^{t}f(s)(1+\|X_{\lambda}(s)\|_H^{p})ds\nonumber\\
	&+p\int_{0}^{t}\|X_{\lambda}(s)\|_H^{p-2}\left<X_{\lambda}(s), \sigma(s, X_{\lambda}(s))dW(s)\right>_H.
\end{align}
Set
$$\tau_{\lambda}^M:=\inf\big\{t\ge0: \|X_{\lambda}(t)\|_H >M\big\} \land T.$$
We have $\tau_{\lambda}^M \to T$ $\mathbb{P}\text{-a.s.}$, as $M \to \infty,$ for each $\lambda\in \left(0,\delta^{-1}\right)$. Taking the supremum over $t\in[0, T\land \tau_{\lambda}^M]$ and then taking expectation on both sides of (\ref{es1}), we get
\begin{align}
	&\mathbb{E}\Big[\sup\limits_{t\in[0, T\land \tau_{\lambda}^M]}\|X_{\lambda}(t)\|_H^p\Big]+\frac{pc}{2}\mathbb{E}\int_{0}^{T\land \tau_{\lambda}^M}\|X_{\lambda}(s)\|_V^{\alpha}\|X_{\lambda}(s)\|_H^{p-2}ds\nonumber \\
	\le&\|x\|_H^p+C\int_{0}^{T}f(s)ds+C\mathbb{E}\int_{0}^{T\land \tau_{\lambda}^M}f(s)\|X_{\lambda}(s)\|_H^pds\nonumber \\
	&+p\mathbb{E}\left[\sup\limits_{t\in[0, T\land \tau_{\lambda}^M]}\left|\int_{0}^{t}\|X_{\lambda}(s)\|_H^{p-2}\left<X_{\lambda}(s), \sigma(s,X_{\lambda}(s))dW(s)\right>_H\right|\right].\label{pes1}
\end{align}
By B-D-G's inequality, Young's inequality and $(\mathbf{H}_{\sigma})$,  we have
\begin{align}
	&p\mathbb{E}\left[\sup\limits_{t\in [0,T\land \tau_{\lambda}^M]}\left|\int_{0}^{t}\|X_{\lambda}(s)\|_H^{p-2}\left<X_{\lambda}(s), \sigma(s,X_{\lambda}(s))dW(s)\right>_H\right|\right]
\nonumber \\
	\le& C\mathbb{E}\left[\int_{0}^{T\land \tau_{\lambda}^M}\|X_{\lambda}(s)\|_H^{2p-2}\|\sigma(s,X_{\lambda}(s))\|_{\mathcal{L}_2}^2ds\right]^{\frac{1}{2}}
\nonumber\\
	\le& C\mathbb{E}\left[\sup\limits_{s\in[0,T\land \tau_{\lambda}^M]}\|X_{\lambda}(s)\|_H^{p}\cdot\int_{0}^{T\land \tau_{\lambda}^M}\|X_{\lambda}(s)\|_H^{p-2}\|\sigma(s,X_{\lambda}(s))\|_{\mathcal{L}_2}^2ds\right]^{\frac{1}{2}}
\nonumber\\
	\le&\frac{1}{2}\mathbb{E}\Big[\sup\limits_{s \in [0,T\land \tau_{\lambda}^M]}\|X_{\lambda}(s)\|_H^{p}\Big]+C\mathbb{E}\int_{0}^{T\land \tau_{\lambda}^M}\|X_{\lambda}(s)\|_H^{p-2}\|\sigma(s,X_{\lambda}(s))\|_{\mathcal{L}_2}^2ds
\nonumber \\
	\le&\frac{1}{2}\mathbb{E}\Big[\sup\limits_{s \in[0, T\land \tau_{\lambda}^M]}\|X_{\lambda}(s)\|_H^{p}\Big]+C\int_{0}^{T}f(s)ds+C\mathbb{E}\int_{0}^{T\land \tau_{\lambda}^M}f(s)\|X_{\lambda}(s)\|_H^pds.\label{martes1}
\end{align}
Combining inequalities \eqref{pes1}-\eqref{martes1} and applying Gronwall's inequality, we obtain
\begin{align*}
	&\mathbb{E}\Big[\sup\limits_{t\in [0,T\land \tau_{\lambda}^M]}\|X_{\lambda}(t)\|_H^p\Big]+C\mathbb{E}\int_{0}^{T \land \tau_{\lambda}^M}\|X_{\lambda}(s)\|_V^{\alpha}\|X_{\lambda}(s)\|_H^{p-2}ds \\
	\le& C\left(\|x\|_H^p+\int_{0}^{T}f(s)ds\right)\exp\Big\{C\int_{0}^{T}f(s)ds\Big\}.
\end{align*}
Letting $M\to \infty$ and applying Fatou's lemma, due to $f\in L^1 ([0,T],\mathbb{R}_+)$, it follows that Lemma \ref{propgj1} holds.
\end{proof}

\begin{lemma}\label{corogjx1}For any $\lambda\in \left(0,\delta^{-1}\right)$, the following inequality holds
\begin{align*}
\mathbb{E}\int_{0}^{T}\Big(\|\mathcal{A}_{\lambda}(s,X_{\lambda}(s))\|_{V^*}^{\frac{\alpha}{\alpha-1}}+\|B(s,X_{\lambda}(s))\|_{V^*}^{\frac{\alpha}{\alpha-1}}+\|\sigma(s,X_{\lambda}(s))\|_{\mathcal{L}_2}^2\Big)ds
\le C_T(1+\|x\|_H^{\beta+2}).
\end{align*}
\end{lemma}
\begin{proof}
Recall the fact that the operator $\mathcal{A}$ has the full domain (i.e.~$\mathcal{D}(\mathcal{A})=V$). By $(\mathbf{H}_{\mathcal{A}}^3)$, $(\mathbf{H}_B^2)$, $(\mathbf{H}_{\sigma})$ and  Proposition $\ref{propyosgj}$ (ii), we have
\begin{align}
&\mathbb{E}\int_{0}^{T}\Big(\|\mathcal{A}_{\lambda}(s,X_{\lambda}(s))\|_{V^*}^{\frac{\alpha}{\alpha-1}}+\|B(s,X_{\lambda}(s))\|_{V^*}^{\frac{\alpha}{\alpha-1}}+\|\sigma(s,X_{\lambda}(s))\|_{\mathcal{L}_2}^2\Big)ds \nonumber\\
\le& \mathbb{E}\int_{0}^{T}\Big((f(s)+C\|X_{\lambda}(s)\|_V^{\alpha})(1+\|X_{\lambda}(s)\|_H^{\beta})+f(s)(1+\|X_{\lambda}(s)\|_H^2)\Big)ds.\label{lambgj1}
\end{align}
Combining Lemma $\ref{propgj1}$ with \eqref{lambgj1}, we complete the proof.
\end{proof}

\subsection{Pre-compactness}
In this part, we take a subsequence $\{\lambda_k\}_{k=1}^{\infty}$ of $\{\lambda\}_{\lambda\in \left(0,\delta^{-1}\right)}$, where $\lambda_k\to 0$ as $k \to \infty$. Along this subsequence, we denote $X_{\lambda_k}$ by $X_k$. Define the following stopping times
\begin{align*}
	\tau_k^M := \inf\big\{t\ge 0: \|X_k(t)\|_H^2 > M\big\}\land \inf \Big\{t\ge 0: \int_{0}^{t}\|X_k(s)\|_V^{\alpha}ds >M\Big\}\land T,
\end{align*}
with the convention $\inf \emptyset = \infty $. By Chebyshev's inequality,
Lemma $\ref{propgj1}$ implies that
\begin{align}
 	\lim\limits_{M \to \infty} \sup\limits_{k\in\mathbb{N}}\mathbb{P}(\tau_k^M < T)=0.\label{Cheb1}
 \end{align}

We present the tightness of the laws of $\{X_k\}_{k=1}^{\infty}$.
\begin{lemma}\label{lemmatig1}
The sequence $\{X_k\}_{k=1}^{\infty}$ is tight in the space $C([0,T],V^*)\cap L^{\alpha}([0,T],H)$.
\end{lemma}
\begin{proof}
In order to get the lemma, we will prove that $\{X_k\}_{k=1}^{\infty}$ is tight in $C([0,T],V^*)$ and in $L^{\alpha}([0,T],H)$ separately. Then for any sequence $\{\mathcal{L}_{X_k}\}_{k=1}^{\infty}$ we can find a subsequence, still denoted by $\{\mathcal{L}_{X_k}\}_{k=1}^{\infty}$, such that $\{\mathcal{L}_{X_k}\}_{k=1}^{\infty}$ converges weakly in $C([0,T],V^*)\cap L^{\alpha}([0,T],H)$.

\vspace{1mm}
\textbf{Step 1.} In this step, we prove that $\{X_k\}_{k=1}^{\infty}$ is tight in $C([0,T],V^*)$. Notice that $H$ is compactly embedded into $V^*$,
it is sufficient to show that for every $e\in P_m H$, $m\in \mathbb{N}$, $\{\left<X_k, e\right>\}_{k=1}^{\infty}$ is tight in the space $C([0,T],\mathbb{R})$ by Theorem 3.1 in \cite{jakubowski1986skorokhod}. In terms of the Aldou's tightness criterion, the result will be proved if we can show that for any stopping time $0\le \zeta_k \le T$ and for any $\varepsilon >0$,
\begin{equation}
\lim\limits_{\delta\to 0}\sup\limits_{k\in \mathbb{N}}\mathbb{P}\left(|\langle X_k(\zeta_k+\delta)-X_k(\zeta_k), e\rangle| > \varepsilon\right) = 0,\label{p1}	
\end{equation}
where $\zeta_k+\delta := T\land (\zeta_k+\delta)\vee 0$. Set $X_k^M(t):=X_k(t\land \tau_k^M)$. By Chebyshev's inequality, it implies that
\begin{align}
	&\mathbb{P}\left(\left|\left<X_k(\zeta_k+\delta)-X_k(\zeta_k),e\right>\right| > \varepsilon\right)\nonumber \\
	\le& \mathbb{P}\left(\left|\left<X_k(\zeta_k+\delta)-X_k(\zeta_k),e\right>\right| > \varepsilon, \tau_k^M \ge T\right)+\mathbb{P}(\tau_k^M < T) \nonumber \\
	\le& \frac{1}{{\varepsilon}^{\alpha}}\mathbb{E}\left|\langle X_k^M(\zeta_k+\delta)-X_k^M(\zeta_k), e\rangle\right|^{\alpha}+\mathbb{P}(\tau_k^M< T).\label{cheb1}
\end{align}
Using \eqref{subequ1} and B-D-G's inequality, we have
\begin{align}
	&\mathbb{E}\left|\left<X_k^M(\zeta_k+\delta)-X_k^M(\zeta_k), e\right>\right|^{\alpha}\nonumber \\
	\le& C\mathbb{E}\left(\int_{\zeta_k\land \tau_k^M}^{(\zeta_k+\delta)\land \tau_k^M}\left|\left<\mathcal{A}_{\lambda_k}(s,X_k(s))+B(s, X_k(s)), e\right>\right|ds\right)^{\alpha}\nonumber \\
	&+C\mathbb{E}\left(\int_{\zeta_k\land \tau_k^M}^{(\zeta_k+\delta)\land \tau_k^M}\|e\|_H^2\|\sigma(s,X_k(s))\|_{\mathcal{L}_2}^2 ds\right)^{\frac{\alpha}{2}}\nonumber\\
	=&: I_k+ II_k.\label{i1i21}
\end{align}
According to H\"older's inequality, $(\mathbf{H}_{\mathcal{A}}^3)$, $(\mathbf{H}_B^2)$ and Proposition $\ref{propyosgj}$ (ii), we deduce that
\begin{align*}
	I_k &\le C|\delta| \mathbb{E}\left(\int_{\zeta_k\land \tau_k^M}^{(\zeta_k+\delta)\land \tau_k^M}|\left<\mathcal{A}_{\lambda_k}(s,X_k(s))+B(s, X_k(s)), e\right>|^{\frac{\alpha}{\alpha-1}}ds \right)^{\alpha-1}\nonumber \\
	&\le C|\delta|\mathbb{E}\left(\int_{0}^{T\land \tau_k^M}\|e\|_V^{\frac{\alpha}{\alpha-1}}(f(s)+C\|X_k(s)\|_V^{\alpha})(1+\|X_k(s)\|_H^{\beta})ds\right)^{\alpha-1}\nonumber \\
	&\le C_M|\delta|.
\end{align*}
Thus, we have
\begin{equation}
	\lim\limits_{\delta \to 0}\sup\limits_{k\in \mathbb{N}} I_k = 0.\label{i11}
\end{equation}
Similarly, by $(\mathbf{H}_{\sigma})$ it follows that
\begin{align*}
	II_k &\le C \mathbb{E}\left(\int_{\zeta_k\land \tau_k^M}^{(\zeta_k+\delta)\land \tau_k^M}\|e\|_H^2 f(s)(1+\|X_k(s)\|_H^2)ds\right)^{\frac{\alpha}{2}}\\
	&\le C_M\left(\int_{\zeta_k\land \tau_k^M}^{(\zeta_k+\delta)\land \tau_k^M}f(s)ds\right)^{\frac{\alpha}{2}}.
\end{align*}
Since $f\in L^1([0,T],\mathbb{R}_+)$, by the absolute continuity of the Lebesgue integral, we have
\begin{equation}
	\lim\limits_{\delta \to 0}\sup\limits_{k\in \mathbb{N}} II_k = 0.\label{i21}
\end{equation}
Combining \eqref{i1i21}-\eqref{i21} yields
\begin{equation}
	\lim\limits_{\delta \to 0} \sup\limits_{k\in \mathbb{N}} \mathbb{E}\left|\left<X_k^M(\zeta_k+\delta)-X_k^M(\zeta_k), e\right>\right|^{\alpha}=0.\label{iz1}
\end{equation}
According to \eqref{Cheb1} and \eqref{iz1}, letting $\delta \to
0$ and then letting $M \to \infty$ in \eqref{cheb1} yield \eqref{p1}.

\vspace{1mm}
\textbf{Step 2.} In this step, we  show that $\{X_k\}_{k=1}^{\infty}$ is tight in $L^{\alpha}([0,T],H)$. By Lemma 5.2 in \cite{rockner2022well} and
\begin{equation*}
	\sup\limits_{k\in \mathbb{N}}\mathbb{E}\int_{0}^{T}\|X_k(t)\|_V^{\alpha}dt < \infty,
\end{equation*}
it suffices to prove that for any $\varepsilon >0$
\begin{equation}
	\lim\limits_{\delta \to 0^{+}} \sup\limits_{k\in \mathbb{N}}\mathbb{P}\left(\int_{0}^{T-\delta}\|X_k(t+\delta)-X_k(t)\|_H^{\alpha}dt > \varepsilon\right) = 0.\label{App1}
\end{equation}
Let $X_k^M(t):= X_k(t\land \tau_k^M)$. Then we have
\begin{align}
	&\mathbb{P}\left(\int_{0}^{T-\delta}\|X_k(t+\delta)-X_k(t)\|_H^{\alpha}dt > \varepsilon\right)\nonumber \\
	\le& \mathbb{P}\left(\int_{0}^{T-\delta}\|X_k(t+\delta)-X_k(t)\|_H^{\alpha}dt > \varepsilon, \tau_k^M \ge T\right)+\mathbb{P}\left(\tau_k^M < T\right)\nonumber \\
	\le& \frac{1}{\varepsilon}\mathbb{E}\int_{0}^{T-\delta}\|X_k^M(t+\delta)-X_k^M(t)\|_H^{\alpha}dt + \mathbb{P}(\tau_k^M < T).\label{jyb1}
\end{align}
Once we can prove that for any fixed $M > 0$,
\begin{equation}
	\lim\limits_{\delta \to 0^{+}}\sup\limits_{k\in \mathbb{N}}\mathbb{E}\int_{0}^{T-\delta}\|X_k^M(t+\delta)-X_k^M(t)\|_H^{\alpha} dt =0,\label{sj1}
\end{equation}
 letting $\delta\to 0$  and $M\to \infty$ in \eqref{jyb1},  then we deduce that \eqref{App1} holds by \eqref{Cheb1}. Next, we  prove \eqref{sj1}.

Firstly, we consider the case $\alpha\in(1, 2]$. By It\^o's formula, we have
\begin{align*}
	&\mathbb{E}\|X_k^M(t+\delta)-X_k^M(t)\|_H^2 \\
	=&\mathbb{E}\int_{t\land \tau_k^M}^{(t+\delta)\land \tau_k^M}\left[-2\langle \mathcal{A}_{\lambda_k}(r,X_k(r))+B(r,X_k(r)), X_k(r)-X_k(t\land \tau _k^M)\rangle\right] dr \\
	&+\mathbb{E}\int_{t\land \tau_k^M}^{(t+\delta)\land \tau_k^M}\|\sigma(r,X_k(r))\|_{\mathcal{L}_2}^2dr.
\end{align*}
Thus, it implies that
\begin{align}
	&\mathbb{E}\int_{0}^{T-\delta}\|X_k^M(t+\delta)-X_k^M(t)\|_H^2dt \nonumber \\
	=&\mathbb{E}\int_{0}^{T-\delta}dt \int_{t\land \tau_k^M}^{(t+\delta)\land \tau_k^M}\Big[-2\left<\mathcal{A}_{\lambda_k}(r,X_k(r))+B(r,X_k(r)), X_k(r)\right>+\|\sigma(r,X_k(r))\|_{\mathcal{L}_2}^2\Big]dr \nonumber \\
	&+2\mathbb{E}\int_{0}^{T-\delta}dt \int_{t\land \tau_k^M}^{(t+\delta)\land \tau_k^M} \left<\mathcal{A}_{\lambda_k}(r,X_k(r))+B(r,X_k(r)), X_k(t\land \tau_k^M)\right>dr \nonumber \\
	=:&~ III_k+IV_k. \label{I121}
\end{align}
By Fubini's theorem and \eqref{coel}, it follows that
\begin{align}
	III_k &=\mathbb{E}\int_{0}^{T\land \tau_k^M}\Big[-2\left<\mathcal{A}_{\lambda_k}(r,X_k(r))+B(r,X_k(r)), X_k(r)\right>+\|\sigma(r,X_k(r))\|_{\mathcal{L}_2}^2\Big]dr
\nonumber\\
&\cdot\int_{0\lor (r-\delta)}^{r}\mathbf{1}_{\{\tau_k^M>t\}} dt \nonumber \\
	&\le C\delta\mathbb{E}\int_{0}^{T\land \tau_k^M}f(s)\left(1+\|X_k(s)\|_H^2\right)ds \nonumber \\
	&\le C_M\delta.\label{I11}
\end{align}
Similarly, by \eqref{ABgrowth} we can deduce that
\begin{align}
	|IV_k|&=2\left|\mathbb{E}\int_{0}^{T\land \tau_k^M}dr \int_{0\lor (r-\delta)}^{r} \mathbf{1}_{\{\tau_k^M>t\}}\left<\mathcal{A}_{\lambda_k}(r,X_k(r))+B(r,X_k(r)), X_k(t\land \tau_k^M)\right>dt\right|\nonumber \\
	&\le C \mathbb{E}\int_{0}^{T\land \tau_k^M}\|\mathcal{A}_{\lambda_k}(r,X_k(r))+B(r,X_k(r))\|_{V^*}dr\int_{0\lor (r-\delta)}^{r}\|X_k(t\land \tau_k^M)\|_Vdt \nonumber \\
	&\le C \delta^{\frac{\alpha-1}{\alpha}}\mathbb{E}\int_{0}^{T\land \tau_k^M}\|\mathcal{A}_{\lambda_k}(r,X_k(r))+B(r,X_k(r))\|_{V^*}dr\left(\int_{0}^{T\land \tau_k^M}\|X_k(t)\|_V^{\alpha}dt\right)^{\frac{1}{\alpha}} \nonumber\\
	&\le C_M \delta^{\frac{\alpha-1}{\alpha}}. \label{I21}
\end{align}
Combining \eqref{I121}-\eqref{I21} yields
\begin{align*}
	\sup\limits_{k\in \mathbb{N}}\mathbb{E}\int_{0}^{T-\delta}\|X_k^M(t+\delta)-X_k^M(t)\|_H^{2}dt \le C_M(\delta+\delta^{\frac{\alpha-1}{\alpha}}).
\end{align*}
Since $\alpha\in (1,2]$, we can apply  H\"older's inequality to obtain
\begin{align*}
	&\lim\limits_{\delta \to 0^{+}}\sup\limits_{k\in \mathbb{N}}\mathbb{E}\int_{0}^{T-\delta}\|X_k^M(t+\delta)-X_k^M(t)\|_H^{\alpha}ds \\
	\le &C\left(\lim\limits_{\delta \to 0^{+}}\sup\limits_{k\in \mathbb{N}}\mathbb{E}\int_{0}^{T-\delta}\|X_k^M(t+\delta)-X_k^M(t)\|_H^2 ds\right)^{\frac{\alpha}{2}}=0.
\end{align*}
We have completed the proof of \eqref{sj1} for $\alpha\in (1,2].$

\vspace{1mm}
Now we consider the case $\alpha>2$. Note that
\begin{align*}
	&\mathbb{E}\|X_k^M(t+\delta)-X_k^M(t)\|_H^{\alpha}\nonumber \\
	=&\frac{\alpha}{2}\mathbb{E}\int_{t \land \tau_k^M}^{(t+\delta)\land \tau_k^M}\|X_k(r)-X_k(t\land \tau_k^M)\|_H^{\alpha-2}\nonumber \\
	&\cdot \left[-2\left<\mathcal{A}_{\lambda_k}(r,X_k(r))+B(r,X_k(r)), X_k(r)-X_k(t\land \tau_k^M)\right>+\|\sigma(r,X_k(r))\|_{\mathcal{L}_2}^2\right]dr \nonumber \\
	&+\frac{\alpha(\alpha-2)}{2}\mathbb{E}\int_{t \land \tau_k^M}^{(t+\delta)\land \tau_k^M}\|X_k(r)-X_k(t\land \tau_k^M)\|_H^{\alpha-4}
\\
&\cdot\|(\sigma(r,X_k(r)))^*(X_k(r)-X_k(t\land \tau_k^M))\|_U^2 dr.
\end{align*}
Using Fubini's theorem, we have
\begin{align}
	&\mathbb{E}\int_{0}^{T-\delta}\|X_k^M(t+\delta)-X_k^M(t)\|_H^{\alpha}dt\nonumber \\
	=&\frac{\alpha}{2}\mathbb{E}\int_{0}^{T-\delta}dt\int_{t \land \tau_k^M}^{(t+\delta)\land \tau_k^M}\|X_k(r)-X_k(t\land \tau_k^M)\|_H^{\alpha-2} \nonumber \\
	&\cdot\left[-2\left<\mathcal{A}_{\lambda_k}(r,X_k(r))+B(r,X_k(r)), X_k(r)\right>+\|\sigma(r,X_k(r))\|_{\mathcal{L}_2}^2\right]dr \nonumber \\
	&+\alpha \mathbb{E}\int_{0}^{T-\delta}dt\int_{t \land \tau_k^M}^{(t+\delta)\land \tau_k^M}\|X_k(r)-X_k(t\land \tau_k^M)\|_H^{\alpha-2}
\nonumber \\
&\cdot\left<\mathcal{A}_{\lambda_k}(r,X_k(r))+B(r,X_k(r)), X_k(t\land \tau_k^M)\right>dr \nonumber \\
	&+\frac{\alpha(\alpha-2)}{2}\mathbb{E}\int_{0}^{T-\delta}dt \int_{t \land \tau_k^M}^{(t+\delta)\land \tau_k^M}\|X_k(r)-X_k(t\land \tau_k^M)\|_H^{\alpha-4}
\nonumber \\
& \cdot\|(\sigma(r,X_k(r)))^*(X_k(r)-X_k(t\land \tau_k^M))\|_U^2 dr\nonumber \\
	=&: I_{k}+II_{k}+III_{k}. \label{J123}
\end{align}
Applying the similar arguments as in \eqref{I11} and \eqref{I21}, it implies that
\begin{align}
I_{k} &\le C_M\delta, \label{J1}\\
|II_{k}|&\le C_M\delta^{\frac{\alpha-1}{\alpha}}\label{J2}.	
\end{align}
On the other hand, by $(\mathbf{H}_{\sigma})$ and Fubini's theorem it follows that
\begin{align}
	III_k &\le C_M\mathbb{E}\int_{0}^{T\land \tau_k^M}f(r)\|X_k(r)-X_k(t\land\tau_k^M)\|_H^{\alpha-2}\left(1+\|X_k(r)\|_H^{2}\right)dr
\nonumber \\
&~~~\cdot\int_{0\lor (r-\delta)}^{r} \mathbf{1}_{\{\tau_k^M>t\}}dt \nonumber \\
	&\le C_M\delta. \label{J3}
\end{align}
Combining \eqref{J123}-\eqref{J3}, we obtain
\begin{align*}
	\sup\limits_{k\in \mathbb{N}}\mathbb{E}\int_{0}^{T-\delta}\|X_k^M(t+\delta)-X_k^M(t)\|_H^{\alpha}dt \le  C_M(\delta+\delta^{\frac{\alpha-1}{\alpha}}).
\end{align*}
Therefore, \eqref{sj1} is proved.  We complete the proof of the lemma.
\end{proof}

\subsection{Proof of Theorem \ref{theoremzy11}}
Let
$$
\Upsilon:=[L^{\alpha}([0,T],H)\cap C([0,T],V^*)]\times C([0,T],U_1),
$$
where $U_1$ is a Hilbert space and the embedding $U \subset U_1$ is Hilbert-Schmidt. According to Lemma \ref{lemmatig1}, we note that the family of the laws $\mathcal{L}(X_k,W)$ is tight in $\Upsilon$. By  the  Skorohod representation theorem, there exists a new probability space $(\widetilde{\Omega},\widetilde{\mathcal{F}},\widetilde{\mathbb{P}})$, along with a sequence of $\Upsilon$-valued random vectors $\{(\widetilde{X}_k,\widetilde{W}_k)\}$ and $(\widetilde{X},\widetilde{W})$, such that
\begin{enumerate}
     \item[(i)]
     \begin{equation}\label{eqlaw}
     \mathcal{L}(\widetilde{X}_k,\widetilde{W}_k)=\mathcal{L}(X_k,W);
      \end{equation}
     \item[(ii)] $\widetilde{\mathbb{P}}\text{-a.s.}$
       \begin{equation}\label{eqcw}
   \|\widetilde{W}_k-\widetilde{W}\|_{C([0,T],U_1)} \to 0	;
      \end{equation}
     \item[(iii)]$\widetilde{\mathbb{P}}\text{-a.s.}$
     \begin{align}
     	\|\widetilde{X}_k-\widetilde{X}\|_{L^{\alpha}([0,T],H)}+\|\widetilde{X}_k-\widetilde{X}\|_{C([0,T],V^*)} \to 0.	\label{gj1}
     \end{align}
\end{enumerate}

Let $\widetilde{\mathcal{F}}_t$ be the usual filtration generated by
$\{\widetilde{X}_k(s),\widetilde{X}(s),\widetilde{W}(s):s\le t\}$.
It is clear that $\widetilde{W}$ is an $\{\widetilde{\mathcal{F}}_t\}$-cylindrical Wiener process on $U$, and
\begin{align}
	\widetilde{X}_k(t) =& x -\int_{0}^{t}\mathcal{A}_{\lambda_k}(s,\widetilde{X}_k(s))ds-\int_{0}^{t}B(s,\widetilde{X}_k(s))ds\nonumber\\
&
+\int_{0}^{t}\sigma(s,\widetilde{X}_k(s))d\widetilde{W}_k(s), \quad t\in [0,T].\label{eqk1}
\end{align}
In view of the Claim (i) above and Lemma \ref{propyos1}, we know that $\{\widetilde{X}_k\}$ satisfies the same moment estimates, i.e., for any $p\ge 2$,
\begin{align}\label{xgj1}
	\sup\limits_{k\in \mathbb{N}}\left\{\widetilde{\mathbb{E}}\Big[\sup\limits_{t\in [0,T]}\|\widetilde{X}_k(t)\|_H^p\Big]+\widetilde{\mathbb{E}}\left(\int_{0}^{T}\|\widetilde{X}_k(t)\|_V^{\alpha}dt\right)^{\frac{p}{2}}\right\}< \infty.
\end{align}
Due to the fact that $\|\cdot\|_H$ and $\|\cdot\|_V$ are lower semicontinuous in $V^*$, by \eqref{gj1} and Fatou's lemma, it follows that
\begin{align}
	\widetilde{\mathbb{E}}\Big[\sup\limits_{t\in [0,T]}\|\widetilde{X}(t)\|_H^p\Big] & \le \widetilde{\mathbb{E}}\Big[\sup\limits_{t\in [0,T]}\liminf\limits_{k\to\infty}\|\widetilde{X}_k(t)\|_H^p\Big] \nonumber \\
	&\le \widetilde{\mathbb{E}}\Big[\liminf\limits_{k\to\infty}\sup\limits_{t\in [0,T]}\|\widetilde{X}_k(t)\|_H^p\Big] \nonumber \\
	&\le \liminf\limits_{k\to\infty}\widetilde{\mathbb{E}}\Big[\sup\limits_{t\in [0,T]}\|\widetilde{X}_k(t)\|_H^p\Big] <\infty.\label{low11}
\end{align}
Similarly, we have
\begin{align}
	\widetilde{\mathbb{E}}\left(\int_{0}^{T}\|\widetilde{X}(s)\|_V^{\alpha}ds\right)^{\frac{p}{2}} < \infty.\label{low21}
\end{align}

Then, by Lemma \ref{corogjx1} we can establish the following estimates.
\begin{lemma}\label{lem1}
The following estimates hold
\begin{align}
&\sup\limits_{k\in \mathbb{N}}\widetilde{\mathbb{E}}\int_{0}^{T}\|\mathcal{A}_{\lambda_k}(t,\widetilde{X}_k(t))\|_{V^*}^{\frac{\alpha}{\alpha-1}}dt < \infty,\nonumber\\
&\sup\limits_{k\in \mathbb{N}}\widetilde{\mathbb{E}}\int_{0}^{T}\|B(t,\widetilde{X}_k(t))\|_{V^*}^{\frac{\alpha}{\alpha-1}}dt < \infty.\label{Bk1}
\end{align}
\end{lemma}

In the sequel, we aim to prove that $(\widetilde{X},\widetilde{W})$ is a solution to equation \eqref{general}. Since the spaces $L^{\alpha}(\widetilde{\Omega}\times[0,T],V), L^{\frac{\alpha}{\alpha-1}}(\widetilde{\Omega}\times[0,T],V^*)$  are reflexive, by the estimate (\ref{xgj1}), Lemma \ref{lem1} and the Banach-Alaoglu's theorem, it implies  that there exist $\widehat{X}\in L^{\alpha}(\widetilde{\Omega}\times[0,T],V)$, $\widetilde{\mathcal{\eta}}\in L^{\frac{\alpha}{\alpha-1}}(\widetilde{\Omega}\times[0,T],V^*)$ and $\widetilde{\mathcal{B}}\in L^{\frac{\alpha}{\alpha-1}}(\widetilde{\Omega}\times[0,T],V^*)$  such that
\begin{align}
	&\widetilde{X}_k \rightharpoonup \widehat{X}\quad \text{in~} L^{\alpha}(\widetilde{\Omega}\times[0,T],V);\label{conv11} 
	\\
	&\mathcal{A}_{\lambda_k}(\cdot,\widetilde{X}_k(\cdot)) \rightharpoonup \widetilde{\mathcal{\eta}} \quad\text{in } L^{\frac{\alpha}{\alpha-1}}(\widetilde{\Omega}\times[0,T],V^*); \label{conv31} \\
        &B(\cdot,\widetilde{X}_k(\cdot)) \rightharpoonup \widetilde{\mathcal{B}}\quad \text{in } L^{\frac{\alpha}{\alpha-1}}(\widetilde{\Omega}\times[0,T],V^*). \label{Bjx1}
\end{align}
Here and in the remaining proof, the limits are taken by choosing a subsequence of $\{k\}$ if necessary.

The following lemma characterizes the limit of the stochastic integral on the right hand side of (\ref{eqk1}).
\begin{lemma}\label{lemmaD1}
We have the following convergence of the stochastic integral
$$\lim_{k\to\infty}\widetilde{\mathbb{E}}\bigg[\sup_{t\in[0,T]}\Big\|\int_{0}^{t}\sigma(s,\widetilde{X}_k(s))d\widetilde{W}_k(s)- \int_{0}^{t}\sigma(s,\widetilde{X}(s))d\widetilde{W}(s)\Big\|_H^2\bigg]= 0.$$
\end{lemma}
\begin{proof}
 Since $\|\widetilde{X}_k-\widetilde{X}\|_{L^{\alpha}([0,T],H)} \to 0,$ $\widetilde{\mathbb{P}}\text{-a.s.}$, by \eqref{xgj1}, \eqref{low21} and H\"older's inequality, we can use the dominated convergence theorem to obtain
\begin{align*}
	\lim\limits_{k\to \infty} \widetilde{\mathbb{E}}\int_{0}^{T}\|\widetilde{X}_k(t)-\widetilde{X}(t)\|_H^{\kappa} = 0~~ \text{for~all}~\kappa\in [1,\alpha).
\end{align*}
Then there exists a subsequence still denoted by $\{\widetilde{X}_k\}$ such that for $dt\otimes\mathbb{P}$-a.e.~$(t,\omega)$
\begin{align}
	\lim\limits_{k\to \infty}\|\widetilde{X}_k(t,\omega)-\widetilde{X}(t,\omega)\|_H = 0.\label{H}
\end{align}
Thus by $(\mathbf{H}_{\sigma})$, \eqref{xgj1} and \eqref{low11}, we can obtain
\begin{align}
	\lim\limits_{k\to \infty} \widetilde{\mathbb{E}}\int_{0}^{T}\|\sigma(t,\widetilde{X}_k(t))-\sigma(t,\widetilde{X}(t))\|_{\mathcal{L}_2}^2 dt = 0.\label{lemma81}
\end{align}
According to Lemma 4.3 in \cite{BMX}, due to the convergence (\ref{eqcw}) and (\ref{lemma81}), we have
\begin{equation}
\sup_{t\in[0,T]}\Big\|\int_{0}^{t}\sigma(s,\widetilde{X}_k(s))d\widetilde{W}_k(s)- \int_{0}^{t}\sigma(s,\widetilde{X}(s))d\widetilde{W}(s)\Big\|_H^2\to 0~~\text{in probability}.
\end{equation}
Therefore, making use of the Vitali's convergence theorem and B-D-G's inequality, it follows from the growth condition in $(\mathbf{H}_{\sigma})$ and the uniform estimate (\ref{low11}) that
$$\lim_{k\to\infty}\widetilde{\mathbb{E}}\bigg[\sup_{t\in[0,T]}\Big\|\int_{0}^{t}\sigma(s,\widetilde{X}_k(s))d\widetilde{W}_k(s)- \int_{0}^{t}\sigma(s,\widetilde{X}(s))d\widetilde{W}(s)\Big\|_H^2\bigg]= 0.$$
We complete the proof.
\end{proof}

Set
\begin{align}
	\bar{X}(t):= x-\int_{0}^{t}\widetilde{\mathcal{\eta}}(s)ds-\int_{0}^{t}\widetilde{\mathcal{B}}(s)ds+\int_{0}^{t}\sigma(s,\widetilde{X}(s))d\widetilde{W}(s).\label{eq1}
\end{align}
Then, it is clear that
\begin{align}
	\widetilde{X}=\widehat{X}=\bar{X} \quad dt\otimes \widetilde{\mathbb{P}}\text{-a.e..}\label{e31}
\end{align}
Indeed, the left equality in \eqref{e31} follows from the uniqueness of the limits and  the equality on the right follows from e.g.~page 58 in \cite{stephan2012yosida}. Moreover, applying Theorem 4.2.5 in \cite{prevot2007concise}, it follows that $\bar{X}$ is an $H$-valued continuous process. According to (\ref{gj1}) and \eqref{low11}, $\widetilde{X}$ is $H$-valued and weakly continuous in $H$. Hence, $\widetilde{X}$ and $\bar{X}$ are indistinguishable.

Throughout the remaining part of this subsection, we work on the new filtered probability space $(\widetilde{\Omega},\widetilde{\mathcal{F}},\{\widetilde{\mathcal{F}}_t\}_{t\ge 0},\widetilde{\mathbb{P}})$, however,   we drop all the superscripts for sake of simplicity.

The following lemma characterizes the limit of $B(\cdot,X_k(\cdot))$. To this end, we recall that $X_k$ and $X$ satisfy the equation \eqref{eqk1}, \eqref{eq1} respectively.
\begin{lemma}\label{lemmaB1}
If we have
\begin{align}
&X_k \rightharpoonup X \ {\rm in} \ L^{\alpha}([0,T]\times\Omega, V);\nonumber\\
&\mathcal{A}_{\lambda_k}(\cdot, X_k(\cdot)) \rightharpoonup \eta \ {\rm in}\ L^{\frac{\alpha}{\alpha-1}}([0,T]\times\Omega, V^*);\label{tj3}\\
&B(\cdot,X_k(\cdot)) \rightharpoonup \mathcal{B} \ {\rm in}\ L^{\frac{\alpha}{\alpha-1}}([0,T]\times\Omega, V^*);\label{tj2}\\
\limsup_{k\to\infty}\mathbb{E}\int_{0}^{T}\langle \mathcal{A}_{\lambda_k}(t, &X_k(t))+B(t,X_k(t)), X_k(t)\rangle dt \le \mathbb{E}\int_{0}^{T}\langle \eta(t)+\mathcal{B}(t), X(t)\rangle dt,\label{tj41}
\end{align}
then $\mathcal{B}(\cdot) = B(\cdot,X(\cdot))$ $dt\otimes\mathbb{P}\text{-a.e.}$.
\end{lemma}

\begin{proof}
According to the proof of Lemma \ref{propyos1} and the Young's inequality, it follows that
\begin{align}
	&\left<\mathcal{A}_{\lambda_k}(t,X_k(t))+B(t,X_k(t)), X_k(t)-X(t) \right>\nonumber \\
	\ge&C\|X_k(t)\|_V^{\alpha}+Cf(t)(1+\|X_k(t)\|_H^2)-\|\mathcal{A}_{\lambda_k}(t,X_k(t))+B(t,X_k(t))\|_{V^*}\|X(t)\|_V\nonumber \\
	\ge&C\|X_k(t)\|_V^{\alpha}+Cf(t)(1+\|X_k(t)\|_H^2)\nonumber \\
&-\big[f(t)+C\|X_k(t)\|_V^{\alpha}\big]^{\frac{\alpha-1}{\alpha}}\big[1+\|X_k(t)\|_H^{\beta}\big]^{\frac{\alpha-1}{\alpha}}\|X(t)\|_V\nonumber \\	\ge&\frac{C}{2}\|X_k(t)\|_V^{\alpha}-Cf(t)(1+\|X_k(t)\|_H^2)-C\|X(t)\|_V^{\alpha}-C\|X_k(t)\|_H^{\beta(\alpha-1)}\|X(t)\|_V^{\alpha}.\label{A341}
\end{align}
We denote
\begin{align*}
	&g_k(t,\omega):=\left<\mathcal{A}_{\lambda_k}(t,X_k(t,\omega))+B(t,X_k(t,\omega)), X_k(t,\omega)-X(t,\omega)\right>, \\
	&F_k(t,\omega):=C\Big(f(t)(1+\|X_k(t,\omega)\|_H^2)+\|X(t,\omega)\|_V^{\alpha}+\|X_k(t,\omega)\|_H^{\beta(\alpha-1)}\|X(t,\omega)\|_V^{\alpha}\Big).
\end{align*}
Then, \eqref{A341}  reduces to
\begin{align}
	g_k(t,\omega)\ge \frac{C}{2}\|X_k(t,\omega)\|_V^{\alpha}-F_k(t,\omega).\label{gF21}
\end{align}
We divide the rest of our proof in three steps.

\vspace{3mm}
\textbf{Step 1.} In this part, we intend to prove that for a.e. ($t,\omega$),
\begin{align}
	\liminf\limits_{k \to \infty} g_k(t,\omega)\ge 0.\label{claim121}
\end{align}
Due to \eqref{H} and Lemma \ref{lemmapseudo1}, there exists a measurable subset $\Gamma$ of $[0,T]\times\Omega$ such that $([0,T]\times\Omega)\backslash\Gamma $ is an $dt\otimes\mathbb{P}$-null set,
\begin{align}
	\lim\limits_{k\to \infty} \|X_k(t,\omega)-X(t,\omega)\|_H = 0, \quad \forall (t,\omega)\in \Gamma,\label{qsl21}
\end{align}
and $B(t,\cdot)$ is pseudo-monotone for any $(t,\omega)\in \Gamma$. Taking any fixed $(t,\omega)\in \Gamma$ and assuming that
$$
\liminf_{k\to\infty}g_k(t,\omega)<0.
$$
Therefore, there exists a subsequence $\{k_i\}_{i\in\mathbb{N}}$  such that
\begin{align}\label{gsub}
	\lim_{i\to\infty}g_{k_i}(t,\omega)<0.
\end{align}
Combining \eqref{gF21} and \eqref{qsl21} yields  that
$\left\{\|X_{k_i}(t,\omega)\|_V^{\alpha}\right\}_{i\in\mathbb{N}}$ is bounded. Thus, we have an element $z\in V$ such that $X_{k_i}(t,\omega)$ converges weakly to $z$ in $V$. According to \eqref{qsl21}, it is clear that $z=X(t,\omega)$ and
$$
X_{k_i}(t,\omega)\rightharpoonup X(t,\omega)~\text{in}~V~~\text{as}~i\to\infty.
$$
By the monotonicity of $A_{\lambda_{k_i}}$ and \eqref{gsub}, it follows that
\begin{align}
	&\langle B(t,X_{k_i}(t,\omega), X_{k_i}(t,\omega)-X(t,\omega)\rangle\nonumber\\
	<& -\langle A_{\lambda_{k_i}}(t,X_{k_i}(t,\omega)), X_{k_i}(t,\omega)-X(t,\omega)\rangle \nonumber\\
	\le& -\langle A_{\lambda_{k_i}}(t,X(t,\omega)), X_{k_i}(t,\omega)-X(t,\omega)\rangle.\label{ABClaim11}
\end{align}
By Proposition \ref{propyosgj} (iii), we know that $A_{\lambda_{k_i}}(x)\to A^0(x)$ for $x\in \mathcal{D}(A)$. Since $\mathcal{D}(A)=V$ and $X_{k_i}(t,\omega) \rightharpoonup X(t,\omega)$ in $V$ for a.e.~$(t,\omega)$, in view of Proposition 21.23 (j) in \cite{ZachariasA}  we have
$$\langle A_{\lambda_{k_i}}(t,X(t,\omega)), X_{k_i}(t,\omega)-X(t,\omega)\rangle\to 0~\text{as}~i\to \infty.$$
Thus, we can get
\begin{align}
	\liminf\limits_{i \to \infty }\langle A_{\lambda_{k_i}}(t,X_{k_i}(t,\omega), X_{k_i}(t,\omega)-X(t,\omega)\rangle\ge 0.\label{Ainf1}
\end{align}
Then, by \eqref{ABClaim11}, we have
\begin{align*}
	\limsup\limits_{i \to \infty }\langle B(t,X_{k_i}(t,\omega), X_{k_i}(t,\omega)-X(t,\omega)\rangle \le 0.
\end{align*}
Applying the pseudo-monotonicity of $B$, it follows that
\begin{align}
	\liminf\limits_{i \to \infty }\langle B(t,X_{k_i}(t,\omega), X_{k_i}(t,\omega)-X(t,\omega)\rangle \ge 0.\label{Bsup1}
\end{align}
Combining \eqref{Ainf1} and \eqref{Bsup1} yields that
\begin{align*}
\liminf\limits_{i\to \infty}\langle A_{\lambda_{k_i}}(t,X_{k_i}(t,\omega))+B(t,X_{k_i}(t,\omega)), X_{k_i}(t,\omega)-X(t,\omega)
\ge 0,
\end{align*}
that is
\begin{align*}
	\liminf\limits_{i\to\infty} g_{k_i}(t,\omega)\ge 0,
\end{align*}
which contradicts to \eqref{gsub}. Hence, we complete the proof of (\ref{claim121}).

\vspace{3mm}
\textbf{Step 2.} In this step, we prove that there exists a subsequence $\{k_i\}$ such that
\begin{align}\label{es3}
	\lim\limits_{i \to \infty}g_{k_i}(t,\omega)=0, \quad \text{for~a.e.~}(t,\omega).
\end{align}
In view of \eqref{qsl21}, we note that $F_k(t,\omega)$ is convergent for a.e.~($t,\omega$). By \eqref{xgj1}, we also know that $F_k$ is uniformly integrable. Therefore, in view of the generalized Fatou's lemma (see e.g.~\cite[p10]{cairoli2011sequential}), \eqref{gF21} and Step 1, it follows that
\begin{align}
	\liminf\limits_{k \to \infty}\mathbb{E}\int_{0}^{T}g_k(t)dt \ge \mathbb{E}\int_{0}^{T}\liminf\limits_{k \to \infty}g_k(t)dt \ge 0.\label{sup21}
\end{align}
Using conditions \eqref{tj3}-\eqref{tj41} yields
\begin{align}
	\limsup\limits_{k \to \infty}\mathbb{E}\int_{0}^{T}g_k(t)dt \le 0.\label{inf21}
\end{align}
Hence, combining \eqref{sup21} and \eqref{inf21} reduces
\begin{align}\label{intgk}
	\lim\limits_{k \to \infty}\mathbb{E}\int_{0}^{T}g_k(t)dt=0.
\end{align}
Set $g_k^-(t,\omega):=\min\{g_k(t,\omega),0\}.$ From Step 1, we know that for a.e.~$(t,\omega)$,
\begin{align*}
	\lim\limits_{k \to \infty}g_k^-(t,\omega)=0.
\end{align*}
By \eqref{gF21} and the uniform integrability of $F_k$, it follows that
\begin{align*}
	\lim\limits_{k \to \infty}\mathbb{E}\int_{0}^{T}g_k^-(t)dt =0.
\end{align*}
Then, applying $|g_k|=g_k-2g_k^-$ and \eqref{intgk}, we obtain
\begin{align*}
	\lim\limits_{k \to \infty}\mathbb{E}\int_{0}^{T}|g_k(t)|dt=0.
\end{align*}
Thus, (\ref{es3}) follows.

\vspace{3mm}
\textbf{Step 3.} In this step, we prove that $\mathcal{B}(\cdot)=B(\cdot,X(\cdot))$  $dt\otimes\mathbb{P}\text{-a.e.}$.

\vspace{1mm}
From \eqref{gF21} and Step 2, it follows that
\begin{align}
	\sup\limits_{i \in \mathbb{N}}\|X_{k_i}(t,\omega)\|_V <\infty \quad \text{for~a.e.~} (t,\omega).\label{claim421}
\end{align}
Combining \eqref{qsl21} and \eqref{claim421}  yields for a.e.~$(t,\omega)$, as $i\to \infty$,
\begin{align*}
	X_{k_i}(t,\omega) \rightharpoonup X(t,\omega)~\text{in~} V.
\end{align*}
By the monotonicity of $A_{\lambda_{k_i}}$ and Step 2, we have a.e.~$(t,\omega)$,
\begin{align*}
	&\limsup\limits_{i \to \infty}\left\langle B(t,X_{k_i}(t,\omega)), X_{k_i}(t,\omega)-X(t,\omega)\right\rangle\nonumber\\
	=& -\liminf\limits_{i \to \infty}\langle A_{\lambda_{k_i}}(t,X_{k_i}(t,\omega)), X_{k_i}(t,\omega)-X(t,\omega)\rangle\nonumber\\
	\le& -\liminf\limits_{i \to \infty}\langle A_{\lambda_{k_i}}(t,X(t,\omega)),X_{k_i}(t,\omega)-X(t,\omega)\rangle
\end{align*}
Using the same argument as in the proof of Step 1, we can get that for a.e.~$(t,\omega)$, as $i\to \infty$,
\begin{equation*}
\langle A_{\lambda_{k_i}}(t,X(t,\omega)), X_{k_i}(t,\omega)-X(t,\omega)\rangle \to 0.	
\end{equation*}
Hence, it follows that
\begin{equation}
\limsup\limits_{i \to \infty}\left<B(t,X_{k_i}(t,\omega)), X_{k_i}(t,\omega)-X(t,\omega)\right> \le 0.	\label{infB1}
\end{equation}
The inequality \eqref{infB1} together with the pseudo-monotonicity of $B$ yields that for a.e.~($t,\omega$), as $i\to \infty$,
\begin{align*}
	B(t,X_{k_i}(t,\omega)) \rightharpoonup B(t,X(t,\omega))~\text{in~}V^*,
\end{align*}
and
\begin{align}
	\lim\limits_{i \to \infty}\left\langle B(t,X_{k_i}(t,\omega)), X_{k_i}(t,\omega)-X(t,\omega)\right\rangle =0.\label{BClaim41}
\end{align}
According to \eqref{Bjx1}  and the uniqueness of limits, we deduce that $\mathcal{B}(\cdot)=B(\cdot,X(\cdot))$, which completes the proof of Lemma \ref{lemmaB1}.
\end{proof}

By combining the above lemmas, we present the proof of Theorem \ref{theoremzy11} as follows.
\vspace{2mm}

\textbf{Proof of Theorem \ref{theoremzy11}.} We intend to prove that $X$  is a weak solution to \eqref{general} in the sense of Definition \ref{weaksolution}.   First, we show that
\begin{equation}\label{Bee}
\mathcal{B}(\cdot)=B(\cdot,X(\cdot))\quad dt\otimes\mathbb{P} \text{-a.e.}.
 \end{equation}
Note that if we can prove \eqref{tj41}, by Lemma \ref{lemmaB1} the result follows immediately.
Applying It\^o's formula to the equation \eqref{eqk1} and \eqref{eq1}, which satisfied by $X_k$ and $X$  respectively, and taking expectations,  we obtain that for any $t\in [0,T]$,
\begin{align}
\mathbb{E}\|X_k(t)\|_H^2=&\|x\|_H^2-2\mathbb{E}\int_{0}^{t}\left<\mathcal{A}_{\lambda_k}(s,X_k(s)),X_k(s)\right>ds \nonumber \\
&-2\mathbb{E}\int_{0}^{t}\left<B(s,X_k(s)),X_k(s)\right>ds+\mathbb{E}\int_{0}^{t}\|\sigma(s,X_k(s))\|_{\mathcal{L}_2}^2ds,\label{em11}\\
\mathbb{E}\|X(t)\|_H^2=&\|x\|_H^2-2\mathbb{E}\int_{0}^{t}\left<\eta(s),X(s)\right>ds \nonumber \\
&-2\mathbb{E}\int_{0}^{t}\left<\mathcal{B}(s),X(s)\right>ds+\mathbb{E}\int_{0}^{t}\|\sigma(s,X(s))\|_{\mathcal{L}_2}^2ds. \label{em21}
\end{align}
Since $\|X_k-X\|_{C([0,T],V^*)} \to 0$, by the lower semicontinuity of $\|\cdot\|_H$ in $V^*$ and Fatou's lemma, it follows that for any $t\in [0,T]$,
\begin{align}
	\mathbb{E}\|X(t)\|_H^2\le \mathbb{E}\Big[\liminf\limits_{k\to \infty}\|X_k(t)\|_H^2\Big] \le \liminf\limits_{k\to \infty}\mathbb{E}\|X_k(t)\|_H^2.\label{liminf1}
\end{align}
Collecting \eqref{em11}-\eqref{liminf1}, in view of \eqref{lemma81} it implies that for any $t\in [0,T]$,
\begin{align*}
&2\limsup\limits_{k\to\infty}\mathbb{E}\int_{0}^{t}\langle \mathcal{A}_{\lambda_k}(s,X_k(s))+B(s,X_k(s)),X_k(s)\rangle ds\\
\le& 2\mathbb{E}\int_{0}^{t}\langle\eta(s)+\mathcal{B}(s),X(s)\rangle ds+\liminf\limits_{k\to\infty}\mathbb{E}\int_{0}^{t}\left(\|\sigma(s,X_k(s))\|_{\mathcal{L}_2}^2-\|\sigma(s,X(s))\|_{\mathcal{L}_2}^2\right)ds\\
\le& 2\mathbb{E}\int_{0}^{t}\langle\eta(s)+\mathcal{B}(s),X(s)\rangle ds+\liminf\limits_{k\to\infty}\mathbb{E}\int_{0}^{t}\|\sigma(s,X_k(s))-\sigma(s,X(s))\|_{\mathcal{L}_2}^2ds\\
=&2\mathbb{E}\int_{0}^{t}\langle\eta(s)+\mathcal{B}(s),X(s)\rangle ds.
\end{align*}
Thus, we can conclude that
$$
\limsup\limits_{k\to\infty}\mathbb{E}\int_{0}^{T}\langle \mathcal{A}_{\lambda_k}(t,X_k(t))+B(t,X_k(t)),X_k(t)\rangle dt\le \mathbb{E}\int_{0}^{T}\langle\eta(t)+\mathcal{B}(t),X(t)\rangle dt,
$$
which implies that (\ref{Bee}) holds.

Finally, once we can show that $\eta\in \mathcal{A}(\cdot,X)$ $dt\otimes\mathbb{P}\text{-a.e.}$,
whose proof is provided in the following lemma, we can deduce that $X$ is a weak solution to Eq.~\eqref{general}. The moment estimates \eqref{genest1} for $X$ are derived from the estimates \eqref{low11} and \eqref{low21}. Thus, we complete the proof. $\hfill\square$

\vspace{2mm}
The following lemma characterizes the limit of $\mathcal{A}(\cdot,X_k(\cdot))$.
\begin{lemma}\label{lemc}
$\eta\in \mathcal{A}(\cdot,X)$ $dt\otimes\mathbb{P}\text{-a.e.}$.
\end{lemma}
\begin{proof}
\textbf{Claim 1.} the multi-valued operator
$$L^{\alpha}([0,T]\times\Omega,V)\ni  x \mapsto \Psi(x):=\mathcal{A}(\cdot,x(\cdot))\in 2^{L^{\frac{\alpha}{\alpha-1}}([0,T]\times\Omega,V^*)}$$
 is maximal-monotone.

\vspace{1mm}
Let $x_1,x_2 \in L^{\alpha}([0,T]\times\Omega,V)$ and $v_i\in \Psi(x_i), i=1,2.$ In view of the monotonicity of $\mathcal{A}$, it follows that
\begin{align*}
	_{L^{\frac{\alpha}{\alpha-1}}([0,T]\times\Omega,V^*)}\langle v_1-v_2, x_1-x_2 \rangle_{L^{\alpha}([0,T]\times\Omega,V)}=\mathbb{E}\int_{0}^{T}\langle v_1(t)-v_2(t), x_1(t)-x_2(t)\rangle dt\ge 0.
\end{align*}
Hence, $\Psi$ is monotone.

Since for a.e. $t\in [0,T]$ the operator $\mathcal{A}(t,\cdot)$ is maximal-monotone, by Proposition 3.14 in \cite{liu2014yosida}, we have that for any $Y\in L^{\frac{\alpha}{\alpha-1}}([0,T]\times\Omega,V^*)$ there exists a progressively measurable process $X(t)$ such that for a.e. $t\in [0,T]$ and any $\lambda>0$,
\begin{align*}
	Y(t,\omega)\in \mathcal{A}(t, X(t,\omega))+\lambda J(X(t,\omega)), \  \omega\in\Omega.
\end{align*}
Let
$Z(t,\omega)\in \mathcal{A}(t, X(t,\omega))$ such that
$$Y(t,\omega)= Z(t,\omega)+\lambda J(X(t,\omega))~\text{for any}~\omega\in\Omega~\text{and a.e.}~t\in [0,T].$$
Taking the dualization product with $X(t,\omega)$, due to $(\mathbf{H}_{\mathcal{A}}^2)$ we can get for a.e.~$(t,\omega)\in [0,T]\times\Omega$
\begin{align*}
	\langle Y(t,\omega), X(t,\omega)\rangle &= \langle Z(t,\omega), X(t,\omega)\rangle+\lambda\langle J(X(t,\omega)), X(t,\omega)\rangle\nonumber\\
	&\ge (\delta+\lambda)\|X(t,\omega)\|_V^{\alpha}-f(t).
\end{align*}
Since $Y\in L^{\frac{\alpha}{\alpha-1}}([0,T]\times\Omega,V^*)$, according to Young's inequality, we have $X\in L^{\alpha}([0,T]\times\Omega,V)$. By Theorem \ref{maxiifon1} in Appendix below, we deduce that the mapping $\Psi$ is maximal-monotone. The claim follows.

\vspace{2mm}
\textbf{Claim 2.} $\mathcal{R}_{\lambda_k}(X_k)$ converges weakly to $X$ in $L^{\alpha}([0,T]\times\Omega,V)$, as $k\to \infty$.

\vspace{1mm}
By the definition of the generalized Yosida approximation (\ref{Alamd}), it follows that for any $(t,\omega)\in[0,T]\times \Omega$
\begin{align*}
	\lambda_k \mathcal{A}_{\lambda_k}(t,X_k(t,\omega))= J(X_k(t,\omega)-\mathcal{R}_{\lambda_k}(X_k(t,\omega))),
\end{align*}
which implies
$$\|\mathcal{R}_{\lambda_k}(X_k(t,\omega))-X_k(t,\omega)\|_V^{\alpha}=\lambda_k^{\frac{\alpha}{\alpha-1}}\|\mathcal{A}_{\lambda_k}(t,X_k(t,\omega))\|_{V^*}^{\frac{\alpha}{\alpha-1}}.$$
Hence, in terms of Lemma \ref{corogjx1} we can deduce that
\begin{align*}
	\mathbb{E}\int_{0}^{T}\|\mathcal{R}_{\lambda_k}(X_k(s))-X_k(s)\|_V^{\alpha}ds= \lambda_k^{\frac{\alpha}{\alpha-1}}\mathbb{E}\int_{0}^{T}\|\mathcal{A}_{\lambda_k}(s,X_k(s))\|_{V^*}^{\frac{\alpha}{\alpha-1}}ds \xrightarrow{k\to\infty} 0,
\end{align*}
that is,
\begin{align}
	\lim\limits_{k \to  \infty}\|\mathcal{R}_{\lambda_k}(X_k)-X_k\|_{L^{\alpha}([0,T]\times\Omega,V)}=0. \label{Jconv1}
\end{align}
Thus, we have for any $\varphi\in L^{\frac{\alpha}{\alpha-1}}([0,T]\times\Omega,V^*)$,
$$\mathbb{E}\int_{0}^{T}\langle \varphi(s), \mathcal{R}_{\lambda_k}(X_k(s))-X_k(s)\rangle ds  \xrightarrow{k\to\infty} 0.$$ Since $X_k \rightharpoonup X$ in $L^{\alpha}([0,T]\times\Omega,V)$, we can get that for any $\varphi\in L^{\frac{\alpha}{\alpha-1}}([0,T]\times\Omega,V^*)$,
\begin{align*}
	&\mathbb{E}\int_{0}^{T}\langle \varphi(s), \mathcal{R}_{\lambda_k}(X_k(s))-X(s)\rangle ds
\nonumber \\
=&\mathbb{E}\int_{0}^{T}\langle \varphi(s), \mathcal{R}_{\lambda_k}(X_k(s))-X_k(s)\rangle ds+\mathbb{E}\int_{0}^{T}\langle \varphi(s), X_k(s)-X(s)\rangle ds\xrightarrow{k\to\infty}0.
\end{align*}
The claim follows.

\vspace{2mm}
\textbf{Claim 3.}
$\limsup\limits_{k\to \infty}\mathbb{E}\int_{0}^{T} \left<\mathcal{A}_{\lambda_k}(s,X_k(s)), \mathcal{R}_{\lambda_k}(X_k(s))\right> ds\le \mathbb{E}\int_{0}^{T}\left<\eta(s), X(s)\right> ds.$

\vspace{1mm}
By \eqref{em11}, we deduce that for any $t\in [0,T]$,
\begin{align}
	\mathbb{E}\|X_k(t)\|_H^2+2\mathbb{E}\int_{0}^{t}\left<\mathcal{A}_{\lambda_k}(s,X_k(s)), X_k(s)\right>ds+I_1+I_2=\|x\|_H^2\label{Igj1},
\end{align}
where
\begin{align*}
I_1:= &\mathbb{E} \int_{0}^{t} 2 \left<B(s,X_k(s))-B(s,X(s)),X_k(s)-X(s)\right>ds,\\
I_2:= &\mathbb{E}\int_{0}^{t}\Big(2 \left<B(s,X(s)),X_k(s)\right>+2 \left<B(s,X_k(s))-B(s,X(s)),X(s)\right>\bigg.\bigg.\\
&\bigg.\bigg.-\|\sigma(s,X_k(s))\|_{\mathcal{L}_2}^2\Big) ds .
\end{align*}
By \eqref{Bk1}, \eqref{conv11} and \eqref{BClaim41}, there exists a subsequence, still denoted by $k$, such that for all $t\in [0,T]$, we have
\begin{align}
\lim\limits_{k \to  \infty}\mathbb{E}\int_{0}^{t}\left<B(s,X_k(s)),X_k(s)-X(s)\right>ds=0\label{Bcl1}
\end{align}
Using the convergence properties \eqref{conv11}-\eqref{Bjx1} and \eqref{lemma81}, we conclude
\begin{align}
	\lim\limits_{k \to \infty} I_2= &\mathbb{E} \int_{0}^{t}\Big(2 \left<\mathcal{B}(s),X(s)\right>-\|\sigma(s,X(s))\|_{\mathcal{L}_2}^2\Big) ds .\label{I_21}
\end{align}
Substituting \eqref{liminf1} and \eqref{I_21}  into \eqref{Igj1} and recalling (\ref{em21}), we obtain
\begin{align}
&\limsup\limits_{k \to  \infty}2\mathbb{E}\int_{0}^{T}\Big(\left<\mathcal{A}_{\lambda_k}(s,X_k(s)),X_k(s)\right>-\left<\eta(s),X(s)\right>\Big)ds\nonumber \\
\le&\mathbb{E}\int_{0}^{T}\Big(2\left<\mathcal{B}(s),X(s)\right>-\|\sigma(s,X(s))\|_{\mathcal{L}_2}^2\Big)ds-\lim\limits_{k \to \infty} (I_1+ I_2)\nonumber\\
=&-\lim\limits_{k \to \infty} I_1.\label{imp1}
\end{align}   
In view of \eqref{conv11}, \eqref{lemma81} and \eqref{Bcl1}, it follows that
\begin{align*}
	\lim\limits_{k \to  \infty}I_1= 0.
\end{align*}
Thus, we conclude that
\begin{equation}
\limsup\limits_{k\to \infty}\mathbb{E}\int_{0}^{T}\left<\mathcal{A}_{\lambda_k}(s,X_{k}(s)),X_{k}(s)\right> ds\le\mathbb{E}\int_{0}^{T}\left<\eta(s),X(s)\right> ds.\label{limsup1}
\end{equation}
Applying H\"older's inequality, Lemma \ref{corogjx1}, \eqref{Jconv1} and \eqref{limsup1}, we deduce that
\begin{align*}
&\limsup\limits_{k \to  \infty}\mathbb{E}\int_{0}^{T}\left<\mathcal{A}_{\lambda_k}(s,X_k(s)),\mathcal{R}_{\lambda_k}(X_k(s))\right>ds  \\
\le& \limsup\limits_{k \to  \infty}\left(\left(\mathbb{E}\int_{0}^{T}\|\mathcal{A}_{\lambda_k}(s,X_k(s))\|_{V^*}^{\frac{\alpha}{\alpha-1}}ds\right)^{\frac{\alpha-1}{\alpha}}\cdot\|\mathcal{R}_{\lambda_k}(X_k)-X_k\|_{L^{\alpha}([0,T]\times\Omega,V)}\right) \\
&+\limsup\limits_{k\to\infty}\mathbb{E}\int_{0}^{T}\left<\mathcal{A}_{\lambda_k}(s,X_k(s)),X_k(s)\right>ds \\
\le& \mathbb{E}\int_{0}^{T}\left<\eta(s),X(s)\right>ds.
\end{align*}
The claim follows.

Now, we proceed to prove the lemma. Since $\mathcal{A}_{\lambda_k}(\cdot,X_k(\cdot))$ converges weakly to $\eta$ in $L^{\frac{\alpha}{\alpha-1}}([0,T]\times\Omega,V^*)$, by \textbf{Claims 1}-\textbf{3} along with the fact that $\mathcal{A}_{\lambda_k}(\cdot,X_k(\cdot))\in \mathcal{A}(\cdot,\mathcal{R}_{\lambda_k}(X_k(\cdot)))$ (cf.~Proposition \ref{propyosgj} (iv)), Lemma \ref{propmaxi} implies that
$$\eta\in \Psi(X)\quad dt\otimes\mathbb{P}\text{-a.e.},$$
that is, $\eta(t,\omega)\in \mathcal{A}(t,X(t,\omega))$ for a.e.~$(t,\omega)\in [0,T]\times\Omega.$
\end{proof}

\subsection{Proof of Theorem \ref{corollary1}}
It suffices to prove the pathwise uniqueness of solutions to \eqref{general}. Then, Theorem \ref{corollary1} follows directly from the Yamada-Watanabe theorem.

Let $(X_1,\eta_1),(X_2,\eta_2)$ be two solutions of Eq.~\eqref{general} with initial values $X_1(0)=X_2(0)=x$. Set
\begin{align*}
	\varphi(t):= \exp\Bigg\{ -\int_{0}^{t}\big(f(s)+\rho(X_1(s))+\zeta(X_2(s))\big)ds\Bigg\}.
\end{align*}
Applying It\^{o}'s formula and by $(\mathbf{H}_{\mathcal{A}}^1)$ and $(\mathbf{H}_B^4)$, it follows that
\begin{align}
	&\varphi(t)\|X_1(t)-X_2(t)\|_H^2\nonumber \\
	=&\int_{0}^{t}\varphi(s)\bigg\{-2\left<\eta_1(s)-\eta_2(s), X_1(s)-X_2(s)\right>\big.\nonumber\\
	&-2\left<B(s,X_1(s))-B(s,X_2(s)), X_1(s)-X_2(s)\right>+\|\sigma(s,X_1(s))-\sigma(s,X_2(s))\|_{\mathcal{L}_2}^2\nonumber\\
	&\big.-[f(s)+\rho(X_1(s))+\zeta(X_2(s))]\|X_1(s)-X_2(s)\|_H^2\bigg\}ds\nonumber\\
	&+2\int_{0}^{t}\varphi(s)\left<X_1(s)-X_2(s),(\sigma(s,X_1(s))-\sigma(s,X_2(s)))dW(s)\right>_H\nonumber\\
	\le&2\int_{0}^{t}\varphi(s)\left<X_1(s)-X_2(s),(\sigma(s,X_1(s))-\sigma(s,X_2(s)))dW(s)\right>_H.\label{twoest1}
\end{align}
Let $\{\tau_n\}$ be a sequence of stopping times such that the local martingale in \eqref{twoest1} becomes a martingale up to $\tau_n$. Then, by taking the expectation on both sides of \eqref{twoest1}, it implies that
\begin{align}
\mathbb{E}\left[\varphi(t\land\tau_n)\|X_1(t\land\tau_n)-X_2(t\land\tau_n)\|_H^2\right]\le 0.\label{uniqstop1}
\end{align}
Letting $n\to\infty$ and applying Fatou's lemma, we deduce that
\begin{align}
	\mathbb{E}\left[\varphi(t)\|X_1(t)-X_2(t)\|_H^2\right]\le 0.\label{uniq1}
\end{align}
Since
\begin{align*}
	\int_{0}^{T}\Big(f(r)+\rho(X_1(r))+\zeta(X_2(r))\Big)dr < \infty\quad \mathbb{P}\text{-a.s.},
\end{align*}
then \eqref{uniq1} yields the pathwise uniqueness of solutions to Eq.~\eqref{general} by the pathwise continuity of solutions in $H$. \hspace{\fill}$\Box$

\subsection{Proof of Theorem \ref{theoremzy21}}
Let the stopping time $\tau_n^M$ be defined as follows
\begin{align*}
	\tau_n^M:= &T\land \inf\left\{t\ge 0: \|X(t,x_n)\|_H >M\right\}\\
	&\land\inf\left\{t\ge 0: \int_{0}^{t}\|X(s,x_n)\|_V^{\alpha}ds > M\right\}\\
	&\land\inf\left\{t\ge 0: \|X(t,x)\|_H > M\right\}\\
	&\land\inf\left\{t\ge 0: \int_{0}^{t}\|X(s,x)\|_V^{\alpha}ds >M \right\},~M>0.
\end{align*}
First, we note that by the moment estimates \eqref{genest1} for the solutions, it follows that
\begin{align}
\lim\limits_{M \to \infty}\sup\limits_{n\in \mathbb{N}}\mathbb{P}(\tau_n^M <T)=0.\label{supst1}
\end{align}
Let $\varphi_n(t) := \exp\Big\{-\int_{0}^{t}\big(f(r)+\rho(X(r,x_n))+\zeta(X(r,x))\big)dr\Big\}$. Combining \eqref{uniqstop1} and \eqref{uniq1}, we obtain
\begin{align}
\mathbb{E}\Big[\varphi_n(t\land\tau_n^M)\|X(t\land \tau_n^M,x_n)-X(t\land \tau_n^M,x)\|_H^2\Big]\le \|x_n-x\|_H^2.\label{gjin}
\end{align}
In light of Chebyshev's inequality and \eqref{gjin}, there exists a constant $C_M >0$ such that for any $\varepsilon >0$,
\begin{align}
	&\mathbb{P}(\|X(t,x_n)-X(t,x)\|_H \ge \varepsilon)\nonumber \\
	\le &\mathbb{P}(\|X(t,x_n)-X(t,x)\|_H\ge \varepsilon, \tau_n^M \ge T)+\mathbb{P}(\tau_n^M < T)\nonumber\\
	\le & \frac{1}{\varepsilon^2C_M}\mathbb{E}\Big[\varphi_n(t\land \tau_n^M)\|X(t\land \tau_n^M,x_n)-X(t\land \tau_n^M,x)\|_H^2\Big]+\mathbb{P}(\tau_n^M < T)\nonumber \\
	\le & \frac{1}{\varepsilon^2C_M}\|x_n-x\|_H^2+\sup\limits_{n\in \mathbb{N}}\mathbb{P}(\tau_n^M < T).\label{Pconv1}
\end{align}
 Taking the limit as $n\to \infty$ and $M \to \infty$ in \eqref{Pconv1} and using \eqref{supst1}, we conclude that for any $t\in [0,T]$,
\begin{align*}
	\|X(t,x_n)-X(t,x)\|_H\xrightarrow{n \to \infty}0 \quad \text{in probability.}
\end{align*}
Furthermore, by \eqref{genest1} we have for any $p \ge 2$,
\begin{align}
	\sup\limits_{n\in \mathbb{N}}\mathbb{E}\Big[\sup\limits_{t\in[0, T]}\|X(t,x_n)\|_H^p \Big] < \infty.\label{supest1}
\end{align}
Then, the Vitali's theorem implies that
\begin{align*}
	\lim\limits_{n\to \infty}\mathbb{E}\int_{0}^{T}\|X(t,x_n)-X(t,x)\|_H^2dt =0,
\end{align*}
which yields that
\begin{align*}
	\|X(t,x_n)-X(t,x)\|_H \xrightarrow{n \to \infty}  0\quad dt\otimes\mathbb{P} \text{-a.e.}.
\end{align*}
Hence, in view of $(\mathbf{H}_{\sigma})$ and \eqref{supest1}, we have
\begin{align}
	\lim\limits_{n\to \infty}\mathbb{E}\int_{0}^{T}\|\sigma(t,X(t,x_n))-\sigma(t,X(t,x))\|_{\mathcal{L}_2}^2dt=0.\label{twoD1}
\end{align}
Then, using \eqref{twoest1}, B-D-G's inequality and Young's inequality, we obtain
\begin{align}
	&\mathbb{E}\Big[\sup\limits_{t\in[0, T\land\tau_n^M]}\big(\varphi_n(t)\|X(t,x_n)-X(t,x)\|_H^2\big)\Big]\nonumber\\
	\le&\|x_n-x\|_H^2+
2\mathbb{E}\Bigg[\sup\limits_{t\in[0, T\land\tau_n^M]}\Big|\int_{0}^{t}\varphi_n(s)\langle X(s,x_n)-X(s,x), \big(\sigma(s,X(s,x_n))
\nonumber\\
&-\sigma(s,X(s,x))\big)dW(s)\rangle_H\Big|\Bigg]\nonumber\\
	\le&\|x_n-x\|_H^2\nonumber\\
&
+C\mathbb{E}\left(\int_{0}^{T\land\tau_n^M}\varphi_n(t)^2\|X(t,x_n)-X(t,x)\|_H^2\|\sigma(t,X(t,x_n))-\sigma(t,X(t,x))\|_{\mathcal{L}_2}^2 dt\right)^{\frac{1}{2}}\nonumber\\
	\le&\|x_n-x\|_H^2+\frac{1}{2}\mathbb{E}\Big[\sup\limits_{t\in[0, T\land\tau_n^M]}\big(\varphi_n(t)\|X(t,x_n)-X(t,x)\|_H^2 \big)\Big]\nonumber\\
	&+C\mathbb{E}\int_{0}^{T\land\tau_n^M}\varphi_n(t)\|\sigma(t,X(t,x_n))-\sigma(t,X(t,x))\|_{\mathcal{L}_2}^2 dt.\label{twosup1}
\end{align}
Combining \eqref{twoD1} and \eqref{twosup1} implies
\begin{align*}
	\lim\limits_{n\to \infty}\mathbb{E}\Big[\sup\limits_{t\in[0, T\land\tau_n^M]}\big(\varphi_n(t)\|X(t,x_n)-X(t,x)\|_H^2 \big)\Big]=0.
\end{align*}
A similar argument as  in the proof of \eqref{Pconv1} leads to
\begin{align*}
\sup\limits_{t\in[0,T]}\|X(t,x_n)-X(t,x)\|_H \xrightarrow{n \to \infty}0 \quad \text{in probability}.
\end{align*}
Furthermore, from \eqref{supest1} we have
\begin{align*}
	\lim\limits_{n\to \infty}\mathbb{E}\Big[\sup\limits_{t\in[0,T]}\|X(t,x_n)-X(t,x)\|_H^p\Big]=0,
\end{align*}
which completes the proof. \hspace{\fill}$\Box$

\section{Finite time extinction}
Self-organized criticality  is extensively studied in  physics from various perspectives (cf.~\cite{T99}). For example, the authors in \cite{BDR092} have suggested that the continuum limit of the Bak-Tang-Wiesenfeld  sandpile model, as introduced in \cite{BTW88},  with a stochastic force can be  interpreted as a type of multi-valued stochastic porous media equations. They have established that the finite time extinction results for this type of multi-valued SPDEs are particularly relevant to the self-organized critical behavior of the BTW model. Subsequently, further investigations have been conducted on  the finite time extinction of stochastic  (sign) fast-diffusion equations perturbed by linear multiplicative noises (cf.~\cite{BDR092,BR12,BRR15,G15,RW13}).

In this section, we demonstrate that the finite time extinction of solutions holds with probability one for all initial value $x\in H$, for a class of multi-valued stochastic evolution inclusions perturbed by linear multiplicative noise.

\subsection{Main result}
Building on the existence and uniqueness of solutions to Eq.~\eqref{general} (i.e.~Theorem \ref{corollary1}), in this part, we mainly focus on the case of the linear multiplicative noise, i.e.,
\begin{align}\label{de}
\sigma(t,x)v:=\sum_{k=1}^{\infty}h_k(t)x\left\langle v,g_k\right\rangle_U,
\end{align}
where $\{g_k\}_{k\in\mathbb{N}}$ is an orthonormal basis on $U$, $\{h_k(t)\}_{k\in\mathbb{N}}$ is a sequence of real-valued functions. More precisely, we consider the noise of the following form
\begin{align*}
\int_0^t\sigma(s,X(s))dW(s)=\sum_{k=1}^{\infty}\int_0^th_k(s)X(s)d\beta_k(s),
\end{align*}
where $\{\beta_k\}_{k\in\mathbb{N}}$ are independent standard Brownian motions defined on a filtered probability space $(\Omega,\mathcal{F},\{\mathcal{F}_t\}_{t\ge 0},\mathbb{P})$.

We assume that there are constants $\alpha\in (1,2)$ and $\delta>0$ such that the following conditions are satisfied for a.e.~$t\in [0,\infty)$.
\begin{enumerate}

\item [(\textbf{H}$_{\mathcal{A}}^{2*}$)]
 For any $x\in V$ and $v\in \mathcal{A}(t,x)$,
$$
\left<v,x\right>\ge \delta\|x\|_V^{\alpha}.
$$

\item [(\textbf{H}$_B^{1*}$)]
There exists $f\in L^1([0,\infty),[0,\infty))$ such that  for any $x\in V$,
$$
2\left<B(t,x),x\right> \ge -f(t)\|x\|_H^2.
$$

\item [$(\mathbf{H}_{\sigma}^*)$]
The function
$$h(t):=\sum_{k=1}^{\infty}|h_k(t)|^2\in L^1([0,\infty),[0,\infty)),$$
which  satisfies
\begin{align*}
    (\alpha-1)h(t)\ge f(t).
\end{align*}

\end{enumerate}

\begin{rem}
    If the diffusion coefficient in $(\ref{de})$ satisfies $(\mathbf{H}_{\sigma}^*)$, then it is easy to verify that $(\mathbf{H}_{\sigma})$ also holds.
\end{rem}

Let $\tau_e$ be the extinction time as follows
$$
\tau_e:=\inf\big\{t\ge0:\|X(t)\|_H=0\big\},
$$
where $X(t), t\ge 0,$ is the solution of Eq.~\eqref{general} with the initial value $x\in H$.

\vspace{1mm}
The following result shows the finite time extinction of solutions to Eq.~\eqref{general}.
\begin{theorem}\label{fini}
	Suppose that  $(\mathbf{H}_{\mathcal{A}}^1)$, $(\mathbf{H}_{\mathcal{A}}^{2*})$,	$(\mathbf{H}_{\mathcal{A}}^3)$, $(\mathbf{H}_B^{1*})$, $(\mathbf{H}_B^2)$-$(\mathbf{H}_B^4)$, and $(\mathbf{H}_{\sigma}^*)$ hold. For any $x\in H$, we have for any $t\ge\tau_e$,
	\begin{align}\label{finit1}
		\|X(t)\|_H=0,\ \mathbb{P}\text{-a.s.,}
	\end{align}
and there is a constant $c^*>0$ such that for any $T>0$,
\begin{align}\label{finit2}
\mathbb{P}(\tau_e\leq T)\geq 1-\frac{c^*\|x\|_{H}^{2-\alpha}}{T}.
\end{align}
Furthermore, we have
\begin{align}\label{finit3}
\mathbb{E}\tau_e\leq c^*\|x\|_H^{2-\alpha}.
\end{align}
\end{theorem}

\begin{rem}
$(i)$ In contrast to earlier works \cite{BDR091, BDR092,BDR12,BR13,BRR15}, where the extinction in finite time was typically established with positive probability, we provide a quantitative characterization of the probability for the extinction time being less than any given  time and also give an explicit upper bound estimate for the first moment  of the extinction time for a class of  multi-valued stochastic evolution
inclusions, including  multi-valued stochastic porous media equations and  multi-valued stochastic $\Phi$-Laplace equations $($see Section \ref{sec5} for details$)$. In fact, as a consequence of the proof of Theorem \ref{fini}, the constant $c^*$ has the following explicit characterization
$$c^*=1/\Big\{\delta(\frac{c_0}{2})^{\alpha}(1-\frac{\alpha}{2})\Big\},$$
where  the constant $c_0$ is related to the embedding constant $\|\cdot\|_V\ge c_0\|\cdot\|_H$.
Moreover, we also drop the requirement of small initial values, and instead consider any initial value $x\in H$.

$(ii)$ A further open question is to obtain the lower bound estimate of the extinction time for multi-valued stochastic evolution
inclusions, which deserves further investigation in the future work.

\end{rem}

\subsection{Approximating sequences}
In order to prove Theorem \ref{fini}, we first recall the approximating equations \eqref{subequ1} constructed in the proof of Theorem \ref{theoremzy11}, i.e.,
\begin{equation}
	\left\{	
	\begin{aligned}
		&dX_{\lambda}(t)+\left[\mathcal{A}_{\lambda}(t,X_{\lambda}(t))+B(t,X_{\lambda}(t))\right]dt= \sum_{k=1}^{\infty}h_k(t)X_{\lambda}(t)d\beta_k(t),\\
		&X_{\lambda}(0)=x\in H,
	\end{aligned}
    \right.\label{subequ2}
\end{equation}
where $\mathcal{A}_{\lambda}$ is the generalized Yosida approximation of $\mathcal{A}$. Moreover, in view of the proof of Theorem \ref{theoremzy11} we take a subsequence $\{\lambda_n\}_{n=1}^{\infty}$, for which $\lambda_n\to 0$ as $n \to \infty$. Along this subsequence, we denote $X_{\lambda_n}$ by $X_n$.

We first recall the powerful characterization of convergence in probability as given by Gy\"{o}ngy and Krylov \cite{GK96}.
Let $(\mathbb{S},\|\cdot\|_{\mathbb{S}})$ be a separable Banach space and $\{Y_n\}_{n\in\mathbb{N}}$ be a sequence of $\mathbb{S}$-valued random variables on the probability space $(\Omega,\mathcal{F},\mathbb{P})$. Let $\{\mu_{n,m}\}_{n,m\in\mathbb{N}}$ be the collection of joint laws of $\{Y_n\}_{n\in\mathbb{N}}$. We have the following result.
\begin{proposition}\label{lemp}
A sequence of $\mathbb{S}$-valued random variables $\{Y_n\}_{n\in\mathbb{N}}$ converges in probability if and only if for every subsequence of joint probabilities laws $\{\mu_{n,m}\}_{n,m\in\mathbb{N}}$, there exists a further subsequence which converges weakly to a probability measure $\mu$ such that
$$\mu(\{(x,y)\in\mathbb{S}\times \mathbb{S}:x=y\})=1.$$

\end{proposition}

We now establish the following convergence of the sequence $\{X_n(t)\}_{n\in\mathbb{N}}$, possibly along with a subsequence,  to $X(t)$.

\begin{theorem}\label{finiconv}
For any $t\in [0,\infty)$,  we have
$$\mathbb{E}\|X_n(t)-X(t)\|_H^2\to 0,~~\text{as}~~n\to\infty.$$
\end{theorem}
\begin{proof}
\textbf{Step 1 (Convergence in probability).} Let us denote
$$
\tilde{\Upsilon}:=\mathbb{S}\times \mathbb{S} \times C([0,T],U_1),
$$
where $\mathbb{S}:=L^{\alpha}([0,T],H)\cap C([0,T],V^*)$.

Similar to the proof of Lemma \ref{lemmatig1}, it is straightforward that the family $\{\mathcal{L}(X_n,X_m,W)\}_{n,m\in\mathbb{N} }$ is tight in $\tilde{\Upsilon}$. Then there exist a probability space $(\widetilde{\Omega},\widetilde{\mathcal{F}},\widetilde{\mathbb{P}})$ and sequences $\{(\widetilde{X}_n,\widetilde{\widetilde{X}}_m,\widetilde{W}_n)\}$ and $(\widetilde{X},\widetilde{\widetilde{X}},\widetilde{W})$ such that along a subsequence,  we have $(\widetilde{X}_n,\widetilde{\widetilde{X}}_m,\widetilde{W}_n)$ converges in $\tilde{\Upsilon}$ to $(\widetilde{X},\widetilde{\widetilde{X}},\widetilde{W})$ $\widetilde{\mathbb{P}}$-a.s..
Note that in particular, $\mu_{n,m}:=\mathcal{L}(X_n,X_m)$  converges weakly to the measure $\mu:=\mathcal{L}(\widetilde{X},\widetilde{\widetilde{X}})|_{\widetilde{\mathbb{P}}}$.

Exactly as for the proof of Theorem \ref{theoremzy11}, we can infer that
both $(\widetilde{X},\widetilde{W})$ and $(\widetilde{\widetilde{X}},\widetilde{W})$ are weak solutions of Eq.~\eqref{general}.
Then by the uniqueness of solutions, we can deduce that $\widetilde{X}=\widetilde{\widetilde{X}}$ $\widetilde{\mathbb{P}}$-a.s..
In other words,
$$\mu(\{(x,y)\in\mathbb{S}\times \mathbb{S}:x=y\})=1.$$
Consequently, Lemma \ref{lemp} implies that the original sequence $X_n$ defined on the initial probability space $(\Omega,\mathcal{F},\mathbb{P})$ converges in probability in the space $\mathbb{S}$ to the unique solution $X$ of Eq.~\eqref{general}.

\vspace{1mm}
\textbf{Step 2 ($L^2$-Convergence).} In view of the Step 1, along with a subsequence, we have
\begin{align*}
	\lim\limits_{n\to \infty}\|X_n(t)-X(t)\|_H = 0~~dt\otimes \mathbb{P}\text{-a.e.},
\end{align*}
which combining with \eqref{genest1} and \eqref{xgj1} implies
\begin{align}\label{t1}
\lim_{n\to\infty}\mathbb{E}\int_{0}^{t}\Big(h(s)\|X_n(s)-X(s)\|_H^2\Big)ds=0.
\end{align}

Applying It\^{o}'s formula for $\|X_n(t)-X(t)\|_H^2$ and by \eqref{eq1} and $(\mathbf{H}_{\sigma}^*)$, it follows that
	\begin{align*}
		&\|X_n(t)-X(t)\|_H^2\nonumber\\
		=&\int_{0}^{t}\Bigg\{-2\langle \mathcal{A}_{\lambda_n}(s,X_n(s))-\widetilde{\eta}(s),X_n(s)-X(s)\rangle\nonumber\\
&-2\langle B(s,X_n(s))-B(s,X(s)), X_n(s)-X(s)\rangle+h(s)\|X_n(s)-X(s)\|_H^2\Bigg
		\}ds\nonumber\\
		&+2\sum_{k=1}^{\infty}\int_{0}^{t}\Big(h_k(s)\|X_n(s)-X(s)\|_H^2\Big)d\beta_k(s).
	\end{align*}
Then, taking the expectation on both sides, we have
	\begin{align}\label{t}
		&\mathbb{E}\|X_n(t)-X(t)\|_H^2\nonumber\\
		=&\mathbb{E}\int_{0}^{t}\Bigg\{-2\langle \mathcal{A}_{\lambda_n}(s,X_n(s))-\widetilde{\eta}(s),X_n(s)-X(s)\rangle\nonumber\\
&-2\langle B(s,X_n(s))-B(s,X(s)), X_n(s)-X(s)\rangle+h(s)\|X_n(s)-X(s)\|_H^2
		\Bigg\}ds.
	\end{align}
By the monotonicity of $\mathcal{A}$, Lemma \ref{corogjx1} and \eqref{Jconv1}, we deduce that
	\begin{align}\label{t2}
		&\lim_{n\to\infty}\mathbb{E}\int_{0}^{t}\langle \widetilde{\eta}(s)-\mathcal{A}_{\lambda_n}(s,X_n(s)),X_n(s)-X(s)\rangle ds\nonumber\\		\le&\sup_{n\in\mathbb{N}}C\Bigg\{\mathbb{E}\int_{0}^{t}\Big(\|\mathcal{A}_{\lambda_n}(s,X_n(s))\|_{V^*}^{\frac{\alpha}{\alpha-1}}+\|\widetilde{\eta}(s)\|_{V^*}^{\frac{\alpha}{\alpha-1}}\Big)ds\Bigg\}^{\frac{\alpha-1}{\alpha}}
\nonumber\\
&\cdot\lim_{n\to\infty}\left(\mathbb{E}\int_{0}^{t}\|\mathcal{R}_{\lambda_n}(X_n(s))-X_n(s)\|_V^{\alpha}ds\right)^{\frac{1}{\alpha}}\nonumber\\
		&-\lim_{n\to\infty}\mathbb{E}\int_{0}^{t}\langle \mathcal{A}_{\lambda_n}(s,X_n(s))-\eta(s),\mathcal{R}_{\lambda_n}(X_n(s))-X(s)\rangle ds
	\nonumber\\
		\le&~0.
	\end{align}
According to the proof of Lemma \ref{lemmaB1}, it follows that there exists a subsequence such that for a.e.~$(t,\omega)$,
	\begin{align*}
		\lim\limits_{n \to \infty}\langle B(t,X_{n}(t,\omega)), X_{n}(t,\omega)-X(t,\omega)\rangle =0.
	\end{align*}
By applying the dominated convergence theorem and using \eqref{xgj1} and \eqref{Bk1}, we conclude that for any $t\in [0,\infty)$,
\begin{align}\label{t3}
    \lim_{n\to\infty}\mathbb{E}\int_{0}^{t}\langle B(s,X_n(s))-B(s,X(s)), X_n(s)-X(s)\rangle ds=0.
\end{align}
Combining \eqref{t1}-\eqref{t3} and using Gronwall's lemma implies that for any $t\in [0,\infty)$,
\begin{align*}
\lim_{n\to\infty}\mathbb{E}\|X_n(t)-X(t)\|_H^2=0.
\end{align*}
We complete the proof.
\end{proof}


\subsection{Proof of main result}
The following lemma plays an important role in the derivation of \eqref{finit1} and \eqref{finit2}.
\begin{lemma}\label{finigj}
	There exists a constant $\rho>0$ such that for any $0\le r<t$,
	\begin{align}\label{finigjine}
		&\|X(t)\|_H^{2-\alpha}+\delta\rho^{\alpha}(1-\frac{\alpha}{2})\int_r^t \mathbf{1}_{\{\|X(s)\|_H>0\}}ds\nonumber\\
		\le& \|X(r)\|_H^{2-\alpha}+2(1-\frac{\alpha}{2})\sum_{k=1}^{\infty}\int_r^t\Big(h_k(s)\|X(s)\|_H^{2-\alpha}\mathbf{1}_{\{\|X(s)\|_H>0\}}\Big) d\beta_s^k.
	\end{align}
\end{lemma}
\begin{proof}

By It\^o's formula for $\|X_n(t)\|_H^2$, which is defined in (\ref{subequ2}), and due to $(\mathbf{H}_{\sigma}^*)$, we have that for any $0\le r<t$,
\begin{align*}
\|X_n(t)\|_H^2=&\|X_n(r)\|_H^2+\int_r^t\Big( -2\langle \mathcal{A}_{\lambda_n}(s,X_n(s))+B(s, X_n(s)), X_n(s)\rangle
\\
&+h(s)\|X_n(s)\|_H^2\Big)ds
+2\sum_{k=1}^{\infty}\int_{r}^{t}\Big(h_k(s)\|X_n(s)\|_H^2\Big)d\beta_s^k.
\end{align*}
For any $\varepsilon > 0$, we apply It\^o's formula based on the auxiliary function $V^{\varepsilon}(x):=(\varepsilon+x)^{1-\frac{\alpha}{2}}$ to get
\begin{align*}
		&(\varepsilon+\|X_n(t)\|_H^2)^{1-\frac{\alpha}{2}}+\delta2^{-\alpha}(1-\frac{\alpha}{2})\int_{r}^{t}\frac{\|X_n(s)\|_V^{\alpha}}{(\varepsilon+\|X_n(s)\|_H^2)^{\frac{\alpha}{2}}}ds\nonumber\\
		\le &(\varepsilon+\|X_n(r)\|_H^2)^{1-\frac{\alpha}{2}}+(1-\frac{\alpha}{2})\int_{r}^{t}\frac{(f(s)+h(s))\|X_n(s)\|_H^2}{(\varepsilon+\|X_n(s)\|_H^2)^{\frac{\alpha}{2}}}ds\nonumber\\
		&-\frac{\alpha}{2}(1-\frac{\alpha}{2})\int_{r}^{t}\frac{2h(s)\|X_n(s)\|_H^4}{(\varepsilon+\|X_n(s)\|_H^2)^{\frac{\alpha}{2}+1}}ds+{\mathcal{M}}_{r,t}^{\varepsilon,n}\nonumber\\
		=&(\varepsilon+\|X_n(r)\|_H^2)^{1-\frac{\alpha}{2}}+(1-\frac{\alpha}{2})\int_{r}^{t}\frac{(f(s)+h(s))(\varepsilon+\|X_n(s)\|_H^2)\|X_n(s)\|_H^2}{(\varepsilon+\|X_n(s)\|_H^2)^{\frac{\alpha}{2}+1}}ds\nonumber\\
		&-(1-\frac{\alpha}{2})\int_{r}^{t}\frac{\alpha h(s)\|X_n(s)\|_H^4}{(\varepsilon+\|X_n(s)\|_H^2)^{\frac{\alpha}{2}+1}}ds+\mathcal{M}_{r,t}^{\varepsilon,n},
\end{align*}
where
$$
\mathcal{M}_{r,t}^{\varepsilon,n}:=2(1-\frac{\alpha}{2})\sum_{k=1}^{\infty}\int_{r}^{t}\frac{h_k(s)\|X_n(s)\|_H^2}{(\varepsilon+\|X_n(s)\|_H^2)^{\frac{\alpha}{2}}}d\beta_s^k,
$$
and we used assumptions (\textbf{H}$_{\mathcal{A}}^{2*}$), $(\mathbf{H}_B^{1*})$ and Lemma \ref{lemmacoe1} in the first inequality.

Since there exists a constant $\rho>0$ such that $\|x\|_V\ge 2\rho\|x\|_H$, it implies that
	\begin{align}\label{xkine}
		&(\varepsilon+\|X_n(t)\|_H^2)^{1-\frac{\alpha}{2}}+\delta\rho^{\alpha}(1-\frac{\alpha}{2})\int_{r}^{t}\frac{\|X_n(s)\|_H^{\alpha}}{(\varepsilon+\|X_n(s)\|_H^2)^{\frac{\alpha}{2}}}ds\nonumber\\
		\le &(\varepsilon+\|X_n(r)\|_H^2)^{1-\frac{\alpha}{2}}+{\mathcal{M}}_{r,t}^{\varepsilon,n}\nonumber\\
		&+(1-\frac{\alpha}{2})\int_{r}^{t}\frac{(f(s)+h(s))(\varepsilon+\|X_n(s)\|_H^2)\|X_n(s)\|_H^2}{(\varepsilon+\|X_n(s)\|_H^2)^{\frac{\alpha}{2}+1}}ds\nonumber\\
		&-(1-\frac{\alpha}{2})\int_{r}^{t}\frac{\alpha h(s)\|X_n(s)\|_H^4}{(\varepsilon+\|X_n(s)\|_H^2)^{\frac{\alpha}{2}+1}}ds.
	\end{align}
 By Theorem \ref{finiconv} and \eqref{xgj1}, we can take $n\to\infty$ (possibly along with a subsequence) to \eqref{xkine} and apply the dominated convergence theorem. As a result, it follows that
\begin{align*}
		&(\varepsilon+\|X(t)\|_H^2)^{1-\frac{\alpha}{2}}+\delta\rho^{\alpha}(1-\frac{\alpha}{2})\int_{r}^{t}\frac{\|X(s)\|_H^{\alpha}}{(\varepsilon+\|X(s)\|_H^2)^{\frac{\alpha}{2}}}ds\nonumber\\
		\le &(\varepsilon+\|X(r)\|_H^2)^{1-\frac{\alpha}{2}}+2(1-\frac{\alpha}{2})\sum_{k=1}^{\infty}\int_{r}^{t}\frac{h_k(s)\|X(s)\|_H^2}{(\varepsilon+\|X(s)\|_H^2)^{\frac{\alpha}{2}}}d\beta_s^k.\nonumber\\
		&+(1-\frac{\alpha}{2})\int_{r}^{t}\frac{(f(s)+h(s))(\varepsilon+\|X(s)\|_H^2)\|X(s)\|_H^2}{(\varepsilon+\|X(s)\|_H^2)^{\frac{\alpha}{2}+1}}ds\nonumber\\
		&-(1-\frac{\alpha}{2})\int_{r}^{t}\frac{\alpha h(s)\|X(s)\|_H^4}{(\varepsilon+\|X(s)\|_H^2)^{\frac{\alpha}{2}+1}}ds.
\end{align*}
Thus, we deduce that
	\begin{align*}
		&(\varepsilon+\|X(t)\|_H^2)^{1-\frac{\alpha}{2}}+\delta\rho^{\alpha}(1-\frac{\alpha}{2})\int_{r}^{t}\Bigg(\frac{\|X(s)\|_H^{\alpha}}{(\varepsilon+\|X(s)\|_H^2)^{\frac{\alpha}{2}}}\mathbf{1}_{\{\|X(s)\|_H>0\}}\Bigg)ds\nonumber\\
		\le &(\varepsilon+\|X(r)\|_H^2)^{1-\frac{\alpha}{2}}+2(1-\frac{\alpha}{2})\sum_{k=1}^{\infty}\int_{r}^{t}\Bigg(\frac{h_k(s)\|X(s)\|_H^2}{(\varepsilon+\|X(s)\|_H^2)^{\frac{\alpha}{2}}}\mathbf{1}_{\{\|X(s)\|_H>0\}}\Bigg)d\beta_s^k\nonumber\\
		&+(1-\frac{\alpha}{2})\int_{r}^{t}\Bigg(\frac{(f(s)+h(s))(\varepsilon+\|X(s)\|_H^2)\|X(s)\|_H^2}{(\varepsilon+\|X(s)\|_H^2)^{\frac{\alpha}{2}+1}}\mathbf{1}_{\{\|X(s)\|_H>0\}}\Bigg)ds\nonumber\\
		&-(1-\frac{\alpha}{2})\int_{r}^{t}\Bigg(\frac{\alpha h(s)\|X(s)\|_H^4}{(\varepsilon+\|X(s)\|_H^2)^{\frac{\alpha}{2}+1}}\mathbf{1}_{\{\|X(s)\|_H>0\}}\Bigg)ds.
	\end{align*}
Letting $\varepsilon\to 0$, it implies that
	\begin{align*}
		&\|X(t)\|_H^{2-\alpha}+\delta\rho^{\alpha}(1-\frac{\alpha}{2})\int_{r}^{t}\mathbf{1}_{\{\|X(s)\|_H>0\}}ds\nonumber\\
		\le & \|X(r)\|_H^{2-\alpha}+{\mathcal{M}}_{r,t}\nonumber\\
		&+(1-\frac{\alpha}{2})\int_r^t\Bigg(\frac{(f(s)+h(s)-\alpha h(s))\|X(s)\|_H^4}{\|X(s)\|_H^{\alpha+2}}\mathbf{1}_{\{\|X(s)\|_H>0\}}\Bigg)ds,
	\end{align*}
	where
	$$
	\mathcal{M}_{r,t}:=2(1-\frac{\alpha}{2})\sum_{k=1}^{\infty}\int_{r}^{t}\Big(h_k(s)\|X(s)\|_H^{2-\alpha}\mathbf{1}_{\{\|X(s)\|_H>0\}}\Big)d\beta_s^k.
	$$
Finally, according to assumption $(\mathbf{H}_{\sigma}^*)$, we conclude that
	$$
	\|X(t)\|_H^{2-\alpha}+\delta\rho^{\alpha}(1-\frac{\alpha}{2})\int_{r}^{t}\mathbf{1}_{\{\|X(s)\|_H>0\}}ds\le \|X(r)\|_H^{2-\alpha}+\mathcal{M}_{r,t},
	$$
which completes the proof.
\end{proof}

\textbf{Proof of Theorem \ref{fini}.} The proof is divided into the following two steps.

\vspace{1mm}
\textbf{Step 1.} In view of Lemma \ref{finigj}, we have that for any $0\le r<t$,
\begin{align*}
	\mathbb{E}[\|X(t)\|_H^{2-\alpha}|\mathcal{F}_r]\le \|X(r)\|_H^{2-\alpha}+\mathbb{E}[\mathcal{M}_{r,t}|\mathcal{F}_r]=\|X(r)\|_H^{2-\alpha}.
\end{align*}
Therefore, $t\to\|X(t)\|_H^{2-\alpha}$ is a nonnegative $(\mathcal{F}_t)$-supermartingale. This combining with (\ref{genest1}) implies that for every pair of stopping times $\tau_1<\tau_2$,
$$\mathbb{E}\|X(\tau_2)\|_{H}^{2-\alpha}\leq \mathbb{E}\|X(\tau_1)\|_{H}^{2-\alpha}.$$
In particular, for any $t>\tau_{e}$ it clear that
$$\mathbb{E}\|X(t)\|_{H}^{2-\alpha}\leq \mathbb{E}\|X(\tau_{e})\|_{H}^{2-\alpha}=0.$$
It implies that for any $t\ge\tau_e$,
$$
\|X(t)\|_H=0\quad \mathbb{P}\text{-a.s.}.
$$
Thus, (\ref{finit1}) follows.

\vspace{1mm}
\textbf{Step 2.} Letting $r=0$ and taking the expectation on both sides of \eqref{finigjine}, we obtain that for any $T> 0$,
\begin{align*}
	\mathbb{E}\|X(T)\|_H^{2-\alpha}+\delta\rho^{\alpha}(1-\frac{\alpha}{2})\int_{0}^{T}\mathbb{P}(\tau_e>s)ds\le\|x\|_H^{2-\alpha}.
\end{align*}
This implies that
$$
\big(\delta\rho^{\alpha}(1-\frac{\alpha}{2})\big)\mathbb{P}(\tau_e>T)T\le\delta\rho^{\alpha}(1-\frac{\alpha}{2})\int_{0}^{T}\mathbb{P}(\tau_e>s)ds\le\|x\|_H^{2-\alpha}.
$$
Then, we deduce that
$$
\mathbb{P}(\tau_e>T)\le\|x\|_H^{2-\alpha}\Big/\big(\delta\rho^{\alpha}(1-\frac{\alpha}{2})\big)T,
$$
which completes (\ref{finit2}).

Thus, letting $T\to\infty$, it follows that
$$
\mathbb{P}(\tau_e<\infty)=1,
$$
which shows that the multi-valued stochastic system \eqref{general} almost surely extinct in finite time.    Based on this, recalling inequality \eqref{finigjine} again, letting $t=\tau_e,r=0$ and taking expectation on both sides of \eqref{finigjine}, we can deduce that
$$\big(\delta\rho^{\alpha}(1-\frac{\alpha}{2})\big)\mathbb{E}\tau_e\leq \|x\|_H^{2-\alpha},$$
which yields  (\ref{finit3}).

We complete the proof.
\hspace{\fill}$\Box$

\section{Examples/Applications}\label{sec5}
In this section, we illustrate our general results  to some examples of multi-valued stochastic evolution inclusions, e.g.,~multi-valued  stochastic porous media equations, multi-valued stochastic $\Phi$-Laplacian equations and stochastic differential inclusions involving subdifferentials.

\subsection{Multi-valued stochastic porous media equations}\label{mspme}
The porous media equation models the flow and diffusion of fluids, such as liquids or gases, through porous materials. It has applications in various fields, including gas flow, heat transfer, and groundwater flow (cf. \cite{B1904}). In this context, we investigate the well-posedness and the finite time extinction of multi-valued stochastic porous media equations.

Let $\Lambda\subset \mathbb{R}^d$ ($d\ge 3$) be an open and bounded domain with smooth boundary $\partial\Lambda$ and consider $p>\frac{2d}{d+2}$. We consider the multi-valued  stochastic porous media equations as follows
\begin{eqnarray}\label{porousmedia}
\left\{
  \begin{aligned}
 &dX(t) \in \Delta\Psi (X(t))dt+\sigma(t, X(t))dW(t), \\
    & X|_{\partial\Lambda}=0, X(0) = x,
  \end{aligned}
\right.
\end{eqnarray}
where $W(t)$ is a cylindrical Wiener process defined on $(\Omega,\mathcal{F},\{\mathcal{F}_t\}_{t\ge 0},\mathbb{P})$ taking values in $U$.

Let $H_0^1(\Lambda):=H_0^{1,2}(\Lambda)$ and $H^{-1}(\Lambda)$ denote the dual space of $H_0^1(\Lambda)$. If $d\ge 3$, by the classical Sobolev embedding theorem, we have $H_0^{1}(\Lambda)\subset L^{\frac{2d}{d-2}}(\Lambda)$ continuously and densely. Since $\frac{p}{p-1}<\frac{2d}{d-2}$, it follows that $L^p(\Lambda)=(L^{\frac{p}{p-1}}(\Lambda))^*\subset(H_0^{1}(\Lambda))^*=H^{-1}(\Lambda)$ continuously and densely.

By Lemma 4.1.12 in \cite{LR15}, the map $(-\Delta)^{-1}:H^{-1}(\Lambda)\to H_0^1(\Lambda)$ is the Riesz isomorphism on $H^{-1}(\Lambda)$. Thus, we can identify $H^{-1}(\Lambda)$ with its dual $(H^{-1}(\Lambda))^*=H_0^1(\Lambda)$ by the Riesz map $(-\Delta)^{-1}$.

Define $V:=L^p(\Lambda), H:=H^{-1}(\Lambda)$ and $V^*:=(L^p(\Lambda))^*$, we have the Gelfand triple
$$
V\subset H\subset V^*,
$$
and the embedding $V\subset H$ is compact.

\begin{lemma}$($cf. \cite[Lemma 4.1.13]{LR15}$)$\label{ps}
    The map
    $$
    \Delta:H_0^1(\Lambda)\to (L^p(\Lambda))^*
    $$
    extends to a linear isometry
    $$
    \Delta: L^{\frac{p}{p-1}}(\Lambda)\to (L^p(\Lambda))^*=V^*
    $$
    and for all $u\in  L^{\frac{p}{p-1}}(\Lambda), v\in L^p(\Lambda)$,
    \begin{align}\label{pse}
        _{V^*}\langle-\Delta u, v\rangle_{V}=_{L^{\frac{p}{p-1}}}\langle u, v\rangle_{L^p}=\int_{\Lambda}u(\xi)v(\xi)d\xi.
    \end{align}
\end{lemma}
Let $\Psi:\mathbb{R}\to 2^{\mathbb{R}}$ be a function satisfying the following assumptions

\vspace{1mm}
$(\Psi1)$ $\Psi$ is maximal-monotone.

\vspace{1mm}
($\Psi2$) For some $p\in (1,\infty)$, there exist $c_1>0$, $c_2\ge 0$ such that for every $x\in \Psi(s)$,
$$s\cdot x\ge c_1|s|^p-c_2.$$

\vspace{1mm}
$(\Psi3)$ There exists $c>0$ such that for every $x\in \Psi(s)$ and every $s\in \mathbb{R}$,
$$|x|\le c|s|^{p-1}+c.$$

Here are some examples from physics models that  satisfy $(\Psi1)$-$(\Psi3)$.
\begin{example}
The stochastic BTW model and the Zhang model  play significant roles in the study of self-organized criticality. Both can be described as multi-valued stochastic porous media equations involving the following operator
\begin{align*}
    \Psi(s)=\left\{
	\begin{aligned}
		&1+\delta s, \ \text{if~}s>0,\\
		&[0,1], \ \text{if~}s= 0,\\
        &0, \ \text{if~}s<0,
	\end{aligned}
	\right.
\end{align*}
with $\delta=0$ in the BTW model $($cf.~\cite{BTW88}$)$ and $\delta>0$ in the Zhang model $($cf.~\cite{Z89}$)$.

\vspace{1mm}
In this work, for $p>1$ and $\nu >0$, we can consider the function
\begin{equation}\label{psi1}
\Psi(s)=\text{sign}(s)(\nu+|s|^{p-1}),
\end{equation}
which satisfies $(\Psi1)$-$(\Psi3)$. Here, the sign function is defined as follows
\begin{align}\label{sign1}
	\text{sign}(s)=\left\{
	\begin{aligned}
		&\frac{s}{|s|}, \ \text{if~}s\neq 0,\\
		&[-1,1], \ \text{if~}s= 0.
	\end{aligned}
	\right.
\end{align}
For $p=2$, (\ref{psi1}) can be seen as a modification of the Zhang model. Moreover, we also extend the well-posedness result established in \cite{liu2014yosida} for $p\in(1,2]$ to the case  $p>1$.
\end{example}

\begin{example}
    For $p>1$ and $\rho>0$, we consider the function
$$\Psi(s)=\rho\cdot\text{sign}(s)|s|^{p-1}+\tilde{\Psi}(s),$$
where $\tilde{\Psi}$ is a continuous monotonically nondecreasing function  with growth of order $p-1$. Our results extend those in \cite{BDR092} from $p\in (1,2)$ to $p>1$, which can be applied to the  slow diffusion models for $p>2$ and the fast diffusion models for $1<p<2$ $($in particular, the plasma fast diffusion model if $p=\frac{3}{2}$$)$.

\begin{example}
For $\rho>0$, we consider the function
$$\Psi(s)=\rho\cdot\text{sign}(s)+\tilde{\Psi}(s),$$
where $\tilde{\Psi}\in C^1(\mathbb{R}\backslash\{0\})$ and there exist constants $\delta,c>0$ and $p>1$ such that $\tilde{\Psi}'(s)\ge \delta|s|^{p-2}$ and $|\tilde{\Psi}(s)|\le c|s|^{p-1}+c$ for all $s\in\mathbb{R}\backslash\{0\}$. For every $x\in \tilde{\Psi}(0)$, assume that there exists a constant $c>0$ such that $|x|\le c$ and that condition $(\Psi1)$ holds. Then, conditions $(\Psi1)$-$(\Psi3)$ are satisfied. The authors \cite{BDR091} investigated the case where $p=2$ in this example and also pointed out that the finite time extinction of solutions to this example exhibits self-organized criticality behavior.
\end{example}

\end{example}
\begin{theorem}\label{expm}
    Assume that $\Psi$ satisfies $(\Psi1)$-$(\Psi3)$ and that $(\mathbf{H}_{\sigma})$ holds. Then for any initial value $x\in H$, the multi-valued stochastic porous media equation \eqref{porousmedia} has a unique (probabilistically) strong solution. In addition, for any $q\ge 2$, we have the following estimate
    \begin{align*}
	\mathbb{E}\Big[\sup\limits_{t\in [0,T]}\|X(t)\|_H^{q}\Big]+\mathbb{E}\left\{\left(\int_{0}^{T}\|X(s)\|_V^pds\right)^{\frac{q}{2}}\right\}<\infty.
\end{align*}
Furthermore, let $X(t,x)$ be the unique solution of Eq.~\eqref{porousmedia} with the initial value $x$, $\{x_n\}_{n\in\mathbb{N}}$ and $x$ be a sequence with $\|x_n-x\|_H\to 0$. Then for any $q>0$,
\begin{align*}
	\lim\limits_{n\to \infty}\mathbb{E}\Big[\sup_{t\in [0,T]}\|X(t,x_n)-X(t,x)\|_H^{q}\Big]=0.
\end{align*}
\end{theorem}

\begin{proof}
We define the multi-valued porous media operator $\mathcal{A}:V\to 2^{V^*}$ by
$$
\mathcal{A}(u):=-\Delta\Psi(u),\quad u\in L^p(\Lambda).
$$
According to condition $(\Psi3)$, we know that $\Psi(u)\in L^{\frac{p}{p-1}}(\Lambda)$ for any $u\in L^p(\Lambda)$. Therefore, the operator $\mathcal{A}$ is well-defined by Lemma \ref{ps}.

In view of \eqref{pse} and condition $(\Psi2)$, we have that for any $u\in L^p(\Lambda)$ and $v\in\Psi(u)$,
$$
\langle -\Delta v,u\rangle =\int_{\Lambda}v(\xi)u(\xi)d\xi\ge c_1\int_{\Lambda}|u(\xi)|^pd\xi-c_2|\Lambda|=c_1\|u\|_V^p-C.
$$
Therefore, the multi-valued porous media operator $\mathcal{A}$ satisfies $(\mathbf{H}_{\mathcal{A}}^2)$ with $\delta:=c_1$ and $\alpha=p$. In addition, let $u\in L^p(\Lambda)$ and $v\in \Psi(u)$ be such that $-\Delta v=\mathcal{A}^0(u)$. According to Lemma \ref{ps} and condition $(\Psi3)$, we deduce that
$$
\|\mathcal{A}^0(u)\|_{V^*}\le C\|u\|_V^{p-1}+C.
$$
Hence, $(\mathbf{H}_{\mathcal{A}}^3)$ holds for the multi-valued porous media operator $\mathcal{A}$. By the proof in Section 6 of \cite{liu2014yosida}, we know that $-\Delta\Psi(\cdot)$ is maximal-monotone.

Since $\mathcal{A}$ satisfies $(\mathbf{H}_{\mathcal{A}}^1)$-$(\mathbf{H}_{\mathcal{A}}^3)$, $B=0$ and $\sigma$ satisfies $(\mathbf{H}_{\sigma})$, we can directly obtain the desired results by applying Theorem \ref{corollary1} and Theorem \ref{theoremzy21}.
\end{proof}

Next, let us focus on the finite time extinction of the multi-valued stochastic porous media equations perturbed by linear
multiplicative noise, i.e.,
\begin{align}\label{porousmedia2}
	\left\{
	\begin{aligned}
		&dX(t) \in \Delta\Psi (X(t))dt+\sum_{k=1}^{\infty}h_k(t)X(t)d\beta_k(t),\\
		& X|_{\partial\Lambda}=0, X(0) = x,
	\end{aligned}
	\right.
\end{align}
where $\{\beta_k\}$ is a sequence of standard Brownian motions on a filtered probability space $(\Omega,\mathcal{F},\{\mathcal{F}_t\}_{t\ge 0},\mathbb{P})$ and $h(t):=\sum_{k=1}^{\infty}|h_k(t)|^2\in L^1([0,\infty),[0,\infty))$.

\begin{theorem}\label{thft1}
Assume that  $(\Psi1)$-$(\Psi3)$ hold for Eq.~\eqref{porousmedia2} with $p\in(1,2)$ and $c_2=0$ in $(\Psi2)$.
For any  $x\in H$ and  $t\ge\tau_e$, where $\tau_e$ is the extinction time,  we have
	\begin{align*}
		\|X(t)\|_H=0,\ \mathbb{P}\text{-a.s.,}
	\end{align*}
and there is a constant $c^*>0$ such that for any $T>0$,
\begin{align*}
\mathbb{P}(\tau_e\leq T)\geq 1-\frac{c^*\|x\|_{H}^{2-p}}{T}.
\end{align*}
Furthermore, we have
\begin{align*}
\mathbb{E}\tau_e\leq c^*\|x\|_H^{2-p}.
\end{align*}
\end{theorem}

\begin{rem}
We point out that if $d=1,2$, one could apply stronger embedding  from  Theorems 7.10 and 7.15 in \cite{GT83} to obtain the well-posedness  and the finite time extinction  for all $p>1$ and $p\in(1,2)$, respectively.
\end{rem}

\subsection{Multi-valued stochastic $\Phi$-Laplacian equations}\label{exlap}

The $\Phi$-Laplacian equations, which generalize the $p$-Laplacian equations, were studied by Gess and T\"olle \cite{GT14} under the weak coercivity condition for $p\in [1,2)$. In this subsection, we focus on the well-posedness and finite time extinction of multi-valued stochastic $\Phi$-Laplacian equations for $p>1$. These include the (singular/degenerate) $p$-Laplacian equations and the models of non-Newtonian fluids governed by the $\Phi$-Laplacian operator.

Let $\Lambda\subset \mathbb{R}^d$ be an open and bounded domain with smooth boundary $\partial\Lambda$, and either $p>1$ if $d=1,2$ or $p> \frac{2d}{d+2}$ if $d\ge 3$. Considering the following multi-valued stochastic $\Phi$-Laplacian equations
 \begin{align}\label{lap}
	\left\{
	\begin{aligned}
		&dX(t) \in \text{div}\Phi(\nabla X(t))dt+\sigma(t, X(t))dW(t),\\
		& X|_{\partial\Lambda}=0, X(0) = x.
	\end{aligned}
	\right.
\end{align}

Define $V:= W_0^{1,p}(\Lambda), H:=L^2(\Lambda)$ and $V^*=(W_0^{1,p}(\Lambda))^*$. Then, we have the Gelfand triple
\begin{align*}
    V\subset H \subset V^*
\end{align*}
compactly and densely.

\vspace{1mm}
Let $\Phi:\mathbb{R}\to 2^{\mathbb{R}}$ be a function satisfying the following conditions.

\vspace{1mm}
$(\Phi1)$ $\Phi$ is maximal-monotone.

\vspace{1mm}
$(\Phi2)$ For $p\in (1,\infty)$, there exist $c_1>0, c_2\ge 0$ such that for any $x\in \Phi(s)$,
$$s\cdot x\ge c_1|s|^p-c_2.$$

\vspace{1mm}
$(\Phi3)$ There exist $c>0$ such that for any $x\in \Phi(s)$ and $s\in \mathbb{R}$,
$$|x|\le c|s|^{p-1}+c.$$

\begin{example}
$(i)$ $p$-Laplacian equations: For $p>1$, the operator
$$ \Phi(s)=|s|^{p-1}\text{sign}(s)$$
satisfies $(\Phi1)$-$(\Phi3)$ with $c_1=1$ and $c_2=0$,
where the sign function is defined in $(\ref{sign1})$.

\vspace{1mm}
$(ii)$ Non-Newtonian fluid models: For $p>1$, the operator
$$\Phi(s)=(1+s^2)^{(p-2)/2}s$$
satisfies $(\Phi1)$-$(\Phi3)$, which has been considered in \cite{SST23}.
\end{example}

\begin{theorem}\label{expl}
Assume that $(\Phi1)$-$(\Phi3)$ and $(\mathbf{H}_{\sigma})$ holds. Then for any initial value $x\in H$, the multi-valued stochastic $\Phi$-Laplacian equations \eqref{lap} have a unique (probabilistically) strong solution. In addition, for any $q\ge 2$, we have the following estimate
    \begin{align*}
	\mathbb{E}\Big[\sup\limits_{t\in [0,T]}\|X(t)\|_H^{q}\Big]+\mathbb{E}\left\{\left(\int_{0}^{T}\|X(s)\|_V^{p}ds\right)^{\frac{q}{2}}\right\}<\infty.
\end{align*}
Furthermore, let $X(t,x)$ be the unique solution of Eq.~\eqref{lap} with the initial value $x$, $\{x_n\}_{n\in\mathbb{N}}$ and $x$ be a sequence with $\|x_n-x\|_H\to 0$. Then for any $q>0$,
\begin{align*}
	\lim\limits_{n\to \infty}\mathbb{E}\Big[\sup_{t\in [0,T]}\|X(t,x_n)-X(t,x)\|_H^{q}\Big]=0.
\end{align*}
\end{theorem}
\begin{proof}
Define the $\Phi$-Laplace operator $\mathcal{A}: V\to 2^{V^*}$ by
$$
\mathcal{A}(u):= -\text{div}\Phi(\nabla u), \quad u\in W_0^{1,p}(\Lambda),
$$
more precisely, given $u\in W_0^{1,p}(\Lambda)$, for any  $\tilde{v}\in \Phi(\nabla u)$, we define $v:=-\text{div}(\tilde{v})\in \mathcal{A}(u)$ such that
\begin{align}\label{lapdef}
    \langle v, w\rangle:=\int_{\Lambda}\tilde{v}(\xi)\nabla w(\xi)d\xi~~\text{for~all~}w\in W_0^{1,p}(\Lambda).
\end{align}
In view of condition $(\Phi3)$, for any $u\in H_0^{1,p}(\Lambda)$ we note that each $v\in\mathcal{A}(u)$ is a well-defined element on $(H_0^{1,p}(\Lambda))^*$, and
$$
\|v\|_{V^*}\le C \|u\|_V^{p-1}+C.
$$
Thus, $(\mathbf{H}_{\mathcal{A}}^3)$ holds for the multi-valued $\Phi$-Laplace operator $\mathcal{A}$ with $\alpha=p$.

According to condition $(\Phi2)$, we deduce that for any $u\in H_0^{1,p}(\Lambda)$,
$$
\langle -\text{div}\Phi(\nabla u), u \rangle=\int_{\Lambda}\Phi(\nabla u(\xi))\nabla u(\xi) d\xi\ge c_1\int_{\Lambda}|\nabla u(\xi)|^pd\xi-c_2|\Lambda|\ge c_1\|u\|_V^{p}-C,
$$
then the operator $\mathcal{A}$ satisfies $(\mathbf{H}_{\mathcal{A}}^2)$ with $\delta:=c_1$ and $\alpha=p$.

Next, we prove that $\mathcal{A}$ is maximal-monotone. Let $x,y\in V, \tilde{v}\in \Phi(\nabla x), \tilde{u}\in \Phi(\nabla y)$ and $v:=-\text{div}(\tilde{v}),u:=-\text{div}(\tilde{u})$. By condition $(\Phi1)$, we have
$$
\langle v-u, x-y \rangle=\int_{\Lambda}(\tilde{v}(\xi)-\tilde{u}(\xi))(\nabla x(\xi)-\nabla y(\xi))d\xi\ge 0.
$$
This shows that $\mathcal{A}$ is monotone.

In order to get the maximal monotonicity $\mathcal{A}$, according to Theorem \ref{maxiifon1}, it is sufficient to show that if for any $y\in V^*$, there exists $x\in V$ such that
\begin{align}\label{exa2}
  J(x)-\text{div}(\tilde{v})= y,
\end{align}
where $\tilde{v}\in \Phi(\nabla x)$ and $J$ is the duality mapping from $V$ to $V^*$.

We begin by considering the approximating equation
\begin{align}\label{exa2apro}
 J(x)-\text{div}\Phi_{\lambda}(\nabla x)=y,
\end{align}
where $\Phi_{\lambda}:\mathbb{R}\to\mathbb{R}$ is the generalized Yosida approximation of $\Phi$ defined by (\ref{Alamd}).

Let $u_n\to u$ in $V$, then we have $\nabla u_n\to \nabla u$ in $L^p(\Lambda)$. By the demicontinuity of $\Phi_{\lambda}$, we deduce that $\Phi_{\lambda}(\nabla u_n)\rightharpoonup \Phi_{\lambda}(\nabla u)$ in $L^{\frac{p}{p-1}}(\Lambda)$. Therefore, for any $w\in V$, we have
$$
\lim_{n\to\infty}\langle -\text{div}\Phi_{\lambda}(\nabla u_n)+\text{div}\Phi_{\lambda}(\nabla u), w \rangle=\lim_{n\to\infty}{}_{L^{\frac{p}{p-1}}}\langle\Phi_{\lambda}(\nabla u_n)-\Phi_{\lambda}(\nabla u), \nabla w\rangle_{L^{p}}=0.
$$
This shows that $-\text{div}\Phi_{\lambda}(\nabla \cdot)$ is also demicontinuous.

Next, since $-\text{div}\Phi_{\lambda}(\nabla \cdot)$ is monotone and demicontinuous, and $J$ is maximal-monotone, we conclude that the operator $x\mapsto J(x)-\text{div}\Phi_{\lambda}(\nabla x)$ is maximal-monotone. According to condition $(\Phi2)$ and Lemma \ref{lemmacoe1}, there exist constants $c_1>0,c_2\ge 0$ such that for any $0<\lambda<c_1^{-1}$, $x\in\Phi_{\lambda}(s)$ and $s\in \mathbb{R}$, we have
$$
s\cdot x\ge c_1 2^{-p}|s|^p-c_2.
$$
Hence, for any $u\in V$, we obtain
\begin{align}\label{exa1}
   \langle -\text{div}\Phi_{\lambda}(\nabla u), u \rangle=\int_{\Lambda}\Phi_{\lambda}(\nabla u(\xi))\cdot\nabla u(\xi)d\xi\ge c_1 2^{-p} \|u\|_V^p-C.
\end{align}
This shows that $-\text{div}\Phi_{\lambda}(\nabla \cdot)$ is coercive. Then, it is clear that $J-\text{div}\Phi_{\lambda}(\nabla \cdot)$ is also coercive. Finally, by Corollary 2.2 in \cite{barbu2010nonlinear} there exists a solution $x_{\lambda}$ to \eqref{exa2apro}.

By \eqref{exa1}, we can deduce that
\begin{align*}
    c_1 2^{-p} \|x_{\lambda}\|_V^p&\le \langle -\text{div}\Phi_{\lambda}(\nabla x_{\lambda}), x_{\lambda}\rangle +C\\
    &=\langle y-J(x_{\lambda}),x_{\lambda}\rangle+C\\
    &\le c_1 2^{-(p+1)}\|x_{\lambda}\|_V^p+C(\|y\|_{V^*}^{\frac{p}{p-1}}+1).
\end{align*}
Therefore, for any $0<\lambda<c_1^{-1}$, we have
\begin{equation}\label{es5}
\sup_{\lambda}\|x_{\lambda}\|_V<\infty.
 \end{equation}
 Denote the operator $\mathcal{A}_{\lambda}(\cdot): = -\text{div}\Phi_{\lambda}(\nabla \cdot)$, it follows from Proposition \ref{propyosgj} (ii) and condition $(\Phi3)$ that
\begin{align}
    \|\Phi_{\lambda}(\nabla x_{\lambda})\|_{L^{\frac{p}{p-1}}(\Lambda)}\le C(\|x_{\lambda}\|_V^p+1)
\end{align}
and
\begin{align*}
\|\mathcal{A}_{\lambda}(x_{\lambda})\|_{V^*}&=\sup_{\|x\|_V=1}|\langle -\text{div}\Phi_{\lambda}(\nabla x_{\lambda}), x\rangle|\\
&\le\sup_{\|x\|_V=1}\int_{\Lambda}|\Phi_{\lambda}(\nabla x_{\lambda}(\xi))| \cdot |\nabla x(\xi)|d\xi\\
&\le C(\|x_{\lambda}\|_V^{p}+1).
\end{align*}
By the definition of $J$, we can conclude that for any $0<\lambda<c_1^{-1}$, the following estimate holds
\begin{equation}\label{es6}
\sup_{\lambda}\left(\|\Phi_{\lambda}(\nabla x_{\lambda})\|_{L^{\frac{p}{p-1}}(\Lambda)}+\|\mathcal{A}_{\lambda}(x_{\lambda})\|_{V^*}+\|J(x_{\lambda})\|_{V^*}\right)<\infty.
 \end{equation}
Hence, in terms of (\ref{es5})-(\ref{es6}), there exists a subsequence $\lambda_k$ of $\lambda$ $(0<\lambda<c_1^{-1})$, we denote $x_{\lambda_k}$ by $x_k$, such that  as $k\to\infty$,
$$x_k\rightharpoonup x~\text{in}~V,~~ \Phi_{\lambda_k}(\nabla x_k)\rightharpoonup x_{\Phi}~\text{in}~L^{\frac{p}{p-1}}(\Lambda),~~ \mathcal{A}_{\lambda_k}(x_k)\rightharpoonup x_{\mathcal{A}},J(x_k)\rightharpoonup x_J~\text{in}~V^*.$$
By \eqref{exa2apro}, it implies that for any $m,n\in \mathbb{N}$,
$$
\langle J(x_n)-J(x_m), x_n-x_m\rangle+ \langle \mathcal{A}_{\lambda_n}(x_n)-\mathcal{A}_{\lambda_m}(x_m), x_n-x_m\rangle =\langle y-y, x_n-x_m \rangle=0.
$$
Since $J$ is monotone, we have
$$
\limsup_{n,m\to \infty}\langle \mathcal{A}_{\lambda_n}(x_n)-\mathcal{A}_{\lambda_m}(x_m), x_n-x_m\rangle\le 0.
$$
This implies that
\begin{align}\label{Alam}
\limsup_{n,m\to \infty}\int_{\Lambda} (\Phi_{\lambda_n}(\nabla x_n(\xi))-\Phi_{\lambda_m}(\nabla x_m(\xi)))(\nabla x_n(\xi)-\nabla x_m(\xi))d\xi\le 0.
\end{align}
In view of $x_k\rightharpoonup x$ in $V$, we have $\nabla x_k\rightharpoonup \nabla x$ in $L^p(\Lambda)$. Since $\Phi$ is maximal-monotone, we know that $\Phi_{\lambda_k}(\nabla x_k)\in \Phi (\mathcal{R}_{\lambda_k}(\nabla x_k))$ and $\|\mathcal{R}_{\lambda_k}(\nabla x_k)-\nabla x_k\|_{L^p(\Lambda)}\to 0$, as $k\to\infty$. By \eqref{Alam} and Lemma \ref{propmaxi}, we deduce that $x_{\Phi}\in \Phi(\nabla x)$. Then, according to \eqref{lapdef}, we have that for all $w\in V$,
\begin{align*}
    \lim_{k\to\infty}\langle \mathcal{A}_{\lambda_k}(x_k)+\text{div}(x_{\Phi}), w\rangle=\lim_{k\to\infty}{_{L^{\frac{p}{p-1}}}}\langle\Phi_{\lambda_k}(\nabla x_k)-x_{\Phi}, \nabla w\rangle_{L^{p}}=0.
\end{align*}
This implies that $\mathcal{A}_{\lambda_n}(x_k)\rightharpoonup -\text{div}(x_{\Phi})\in \mathcal{A}(x)$.
Consequently,
 $$
 \limsup_{n,m\to \infty}\langle J(x_n)-J(x_m), x_n-x_m\rangle\le 0,
 $$
which implies that $x_J=J(x)$. Thus, we obtain that $x\in V$ is the solution of \eqref{exa2}. We complete the proof of (\ref{exa2}).

Since $\mathcal{A}$ satisfies $(\mathbf{H}_{\mathcal{A}}^1)$-$(\mathbf{H}_{\mathcal{A}}^3)$, $B=0$ and $\sigma$ satisfies $(\mathbf{H}_{\sigma})$, the conclusion follows directly from Theorem \ref{corollary1} and Theorem \ref{theoremzy21}.
\end{proof}
We consider the finite time extinction for the following multi-valued stochastic $\Phi$-Laplacian equations perturbed by linear
multiplicative noise
\begin{align}\label{lap2}
	\left\{
	\begin{aligned}
		&dX(t) \in \text{div}\Phi(\nabla X(t))dt+\sum_{k=1}^{\infty}h_k(t)X(t)d\beta_k(t),\\
		& X|_{\partial\Lambda}=0, X(0) = x,
	\end{aligned}
	\right.
\end{align}
where $h(t):=\sum_{k=1}^{\infty}|h_k(t)|^2\in L^1([0,\infty),[0,\infty))$.
\begin{theorem}
Assume that  $(\Phi1)$-$(\Phi3)$ hold for Eq.~\eqref{lap2} with $p\in(1,2)$ and $c_2=0$ in $(\Phi2)$. For any  $x\in H$ and $t\ge\tau_e$,  we have
	\begin{align*}
		\|X(t)\|_H=0,\ \mathbb{P}\text{-a.s.,}
	\end{align*}
and there is a constant $c^*>0$ such that for any $T>0$,
\begin{align*}
\mathbb{P}(\tau_e\leq T)\geq 1-\frac{c^*\|x\|_{H}^{2-p}}{T}.
\end{align*}
Furthermore, we have
\begin{align*}
\mathbb{E}\tau_e\leq c^*\|x\|_H^{2-p}.
\end{align*}
\end{theorem}

\begin{rem}
To the best of our knowledge, this is the first result to establish the finite time extinction with probability one and derive an upper bound for the moment estimate of the extinction time  to multi-valued stochastic $\Phi$-Laplacian equations $($in particular, the $p$-Laplace equations with all $p\in(1,2)$$)$.

In the earlier work \cite{BR13}, the finite time extinction with positive probability was proven for the stochastic total variation flow $($i.e., the singular $1$-Laplace equations$)$ under certain small initial data conditions. Our result can be viewed as a complementary extension to the case $p \in (1,2)$, where the almost sure extinction and the quantitative upper bound for the moment estimate of the extinction time are obtained.
\end{rem}

\subsection{SEIs with subdifferentials}\label{SEIsub}
Evolution inclusions with subdifferentials were first introduced by Br\'ezis in \cite{Brezis73}, originated from simplified models employed in the description of porous medium combustion, chemical reactor theory and game theory (cf.~\cite{OS23}). These evolution inclusions have contributed to the theory of nonlinear evolution equations (cf.~\cite{QX15}). Further research on evolution inclusions with subdifferentials has been conducted in \cite{AD72,KP88,O82,O84}.  In this part, we intend to present a more general result for the existence of weak solutions to a class of SEIs involving subdifferentials with single-valued pseudo-monotone operators.

Let $K$ be a uniformly convex Banach space and $\varphi: K \to \bar{\mathbb{R}} :=(-\infty, \infty]$ be a lower semicontinuous, proper and convex function on $K$. Let $V \subset K$ be a domain of $\partial\varphi$, where the mapping $\partial\varphi:K \to 2^{K^*}$ defined by
$$
\partial\varphi(v):=\Big\{v^*\in K^* \mid \varphi(v)\le \varphi(w) + \langle v^*, v-w\rangle, \forall w\in K\Big\}
$$
is called the subdifferential of $\varphi$. We denote $\mathcal{D}(\partial\varphi)$ as the set of all $x\in K$ for which $\partial\varphi(x) \neq \emptyset$.

\begin{proposition}$($cf. Theorem 2.8 in \cite{barbu2010nonlinear}$)$
For a lower semicontinuous, proper and convex function $\varphi$, the subdifferential $\partial\varphi: K\to 2^{K^*}$ is maximal-monotone.
\end{proposition}\label{subpro}

Let $\Lambda\subset \mathbb{R}^d$ be a bounded domain with smooth boundary $\partial\Lambda$. We consider the following multi-valued stochastic partial differential equation:
\begin{align}\label{quasispde}
	\left\{
	\begin{aligned}
		\partial_t u(t,x)&\in  \left[\Delta u(t,x)-g(t,x, u(t,x),\nabla
   u(t,x))-\partial\varphi(u(t,x))\right]dt\\
&~~~~~~~~~~~~~~~~~~~~+\sigma(t,u(t,x))dW(t),\\
		u(t,\cdot)|_{\partial\Lambda}&=0,~u(0,x)=u_0(x).
	\end{aligned}
	\right.
\end{align}
We define $H:=L^2(\Lambda)$ and $V:=W_0^{1,2}(\Lambda)$. This forms the Gelfand triple
$$V\subset H\subset V^*.$$
The space $V$ is a uniformly convex Banach space and the embedding $V\subset H$ is compact.

We assume that $\varphi$ is a lower semicontinuous, proper and convex function with $\varphi(0)=0$, and that $\mathcal{D}(\partial\varphi)=V$. Additionally, there exist constants $c_1>0,c_2\ge 0$ such that $$\varphi(u)\ge c_1\|u\|_V^{2}-c_2~~\text{for all}~u\in V,$$
 and a constant $c>0$  such that
 $$|\varphi(u)-\varphi(u-z)|\le c_3\|u\|_V+c_4~~\text{for all}~u,z\in V~\text{with}~\|z\|_V=1.$$

 We further assume that $g$ satisfies the following conditions:

\vspace{1mm}
(i) $g$ satisfies the Carath\'eodory conditions: for a.e.~$(t,x)\in [0,T]\times\Lambda$, $g(t,x,u,z)$ is continuous in $(u,z)\in \mathbb{R}\times\mathbb{R}^d$, and for each  $(u,z)\in\mathbb{R}\times\mathbb{R}^d$, $g(t,x,u,z)$ is measurable with respect to $(t,x)\in [0,T]\times\Lambda$.

   \vspace{1mm}
(ii) There exist a constant $c>0$ and a function $f_1\in L^1([0,T]\times\Lambda,[0,\infty))$ such that for a.e.~$(t,x)\in [0,T]\times\Lambda$ and all $(u,z)\in\mathbb{R}\times\mathbb{R}^d$,
    $$
    g(t,x,u,z)u\ge -c|u|^2-f_1(t,x).
    $$

\vspace{1mm}
(iii) There exist a nonnegative constant $c$ and a function $f_2 \in  L^2([0,T]\times\Lambda,[0,\infty))$ such that for a.e.~$(t,x)\in [0,T]\times\Lambda$ and all $(u,z)\in\mathbb{R}\times\mathbb{R}^d$,
    $$
    |g(t,x,u,z)|\le c|z|+c|u|^{\frac{d+2}{d}}+f_2(t,x).
    $$

\begin{theorem}
 Suppose that $\varphi$ and $g$ satisfy the above assumptions and that $(\mathbf{H}_{\sigma})$ holds. Then for any initial value $u_0\in H$, the multi-valued stochastic evolution inclusions \eqref{quasispde} have a (probabilistically) weak solution. Moreover, for any $p\ge 2$, we have the following estimate
    \begin{align*}
	\mathbb{E}\Big[\sup\limits_{t\in [0,T]}\|X(t)\|_H^{p}\Big]+\mathbb{E}\left\{\left(\int_{0}^{T}\|X(s)\|_V^2ds\right)^{\frac{p}{2}}\right\}<\infty.
\end{align*}
\end{theorem}
\begin{proof}
We define the operator $\mathcal{A}:V\to 2^{V^*}$ by
$$
\mathcal{A}(u):= \partial \varphi(u), \quad u\in W_0^{1,2}(\Lambda),
$$
it follows from Proposition \ref{subpro} that $\mathcal{A}$ satisfies the conditions $(\mathbf{H}_{\mathcal{A}}^1)$.

By the definition and assumptions of $\varphi$, we have that for any $u\in V$ and $u^*\in \mathcal{A}(u)$,
\begin{align*}
    \langle u^*,u \rangle&\ge \varphi(u)-\varphi(0)\ge c_1\|u\|_V^2-c_2,
\end{align*}
and
\begin{align*}
    \|u^*\|_{V^*}\le \sup_{\|z\|_V=1}|\langle u^*, z\rangle|\le\sup_{\|z\|_V=1}|\varphi(u)-\varphi(u-z)|\le c_3\|u\|_V+c_4.
\end{align*}
 Thus, $\mathcal{A}$ satisfies the conditions $(\mathbf{H}_{\mathcal{A}}^2)$ and $(\mathbf{H}_{\mathcal{A}}^3)$.

Define the operator $B: V \to V^*$ by
$$
B(u):= -\Delta u+g(u,\nabla u), \quad u\in W_0^{1,2}(\Lambda).
$$
From Example 4.1 in \cite{rockner2022well}, we know that $B$ is pseudo-monotone and satisfies the growth condition $(\mathbf{H}_B^2)$. We have
$$
 \langle  -\Delta u+g(u,\nabla u), u\rangle\ge -c\|u\|_H^2+\|u\|_V^2-\|f(t,\cdot)\|_{L^1}.
$$
Hence, $(\mathbf{H}_B^1)$ holds for $B$. Then, by Theorem \ref{theoremzy11} the result follows.
\end{proof}
\begin{rem}
Let $h:\mathbb{R}\to (-\infty,+\infty]$ be a lower semicontinuous, proper and convex function, and $\varphi: L^2(\Lambda)\to(-\infty,+\infty]$ be defined by
\begin{align*}
	\varphi(u):=\left\{
	\begin{aligned}
		&\int_{\Lambda}h(u(x))dx, \quad \text{if } h(u)\in L^1(\Lambda),\\
		&+\infty, \quad \text{otherwise}.
	\end{aligned}
	\right.
\end{align*}
By  Proposition 2.8 in \cite{barbu76}, $\varphi$ is a lower semicontinuous, proper and convex function and $\mathcal{D}(\partial\varphi)=L^2(\Lambda)$.
\end{rem}

\section{Appendix}
In this section, we present some results of the multi-valued maximal-monotone operator.
\begin{proposition}\label{app1}
    Let $V$ be a Banach space and the dual space $V^*$ is strictly convex, then the duality mapping $J:V\to V^*$ is single-valued, demicontinuous and odd.
\end{proposition}
\begin{proof}
       We first prove that $J$ is single-valued. Let $v_1,v_2\in J(u)$, we have
      $$
      \left\langle v_1,u\right\rangle=\left\langle v_2,u\right\rangle=\|v_1\|_{V^*}^{\frac{\alpha}{\alpha-1}}=\|v_2\|_{V^*}^{\frac{\alpha}{\alpha-1}}=\|u\|_V^{\alpha}.
      $$
      Hence, according to \eqref{Jpro1}, it follows that
      \begin{align*}
          2\|v_1\|_{V^*}\|u\|_V\le \left\langle v_1+v_2,u\right\rangle\le\|v_1+v_2\|_{V^*}\|u\|_V \ \text{for~all~} u\in V,
      \end{align*}
      which implies that $\|v_1\|_{V^*}\le \frac{1}{2}\|v_1+v_2\|_{V^*}$. Since $\|v_1\|_{V^*}=\|v_2\|_{V^*}$ and $V^*$ is strictly convex, we deduce that $v_1=v_2$.

\vspace{1mm}
      Next, we show that $J$ is demicontinuous. Let $\{u_n\}_{n\in \mathbb{N}}\in V$ be strongly convergent to $u$ and let $u^*$ be any weak-* limit point of $\{J(u_n)\}_{n\in \mathbb{N}}$. According to the definition of $J(u_n)$, we have $$\|u_n\|_V^{\alpha}=\|J(u_n)\|_{V^*}^{\frac{\alpha}{\alpha-1}}.$$
       Thus, by the above conditions and the weak lower semicontinuity of the norm in $V^*$, we conclude that
    \begin{align*}
        \|u^*\|_{V^*}^{\frac{\alpha}{\alpha-1}}\le \liminf_{n\to \infty}\|J(u_n)\|_{V^*}^{\frac{\alpha}{\alpha-1}}&\le\lim_{n\to\infty}\|u_n\|_V^{\alpha}(=\|u\|_V^{\alpha})
        \\
        &=\lim_{n\to\infty}\langle J(u_n),u_n\rangle=\langle u^*,u\rangle\le \|u^*\|_{V^*}\|u\|_V.
    \end{align*}
    Then, it follows that $$\langle u^*,u\rangle=\|u\|_V^{\alpha}=\|u^*\|_{V^*}^{\frac{\alpha}{\alpha-1}}.$$ By the definition of $J(u)$, we deduce that $u^*=J(u)$. Consequently, the sequence $\{J(u_n)\}_{n\in\mathbb{N}}$ converges weak-* to $J(u)$. Therefore, $J:V\to V^*$ is demicontinuous.

    \vspace{1mm}
    According to the definition of $J$, it implies that
    \begin{align*}
        \langle J(-u),-u\rangle=\|J(-u)\|_V^{\frac{\alpha}{\alpha-1}}=\|u\|_V^{\alpha}=\langle -J(u),-u\rangle~~~ \text{for~all}~u\in V.
    \end{align*}
    Since $J$ is single-valued, we conclude
    \begin{align*}
        J(-u)=-J(u),
    \end{align*}
which shows that $J$ is odd.

\end{proof}
\begin{theorem}\label{maxiifon1}	Let $V$ and $V^*$ be reflexive and strictly convex Banach spaces. Let $\mathcal{A}: V\to 2^{V^*}$ be a monotone operator and $J: V\to V^*$ denote the duality mapping of $V$. Then $\mathcal{A}$ is maximal-monotone if and only if for any (or some) $\lambda >0$, the range condition
\begin{align*}
\mathcal{R}(\mathcal{A}+\lambda J)= V^*
\end{align*}
is satisfied.
\end{theorem}
\begin{proof}
\textbf{Step 1:} Suppose that $\mathcal{R}(\mathcal{A}+\lambda J)= V^*$ for some $\lambda>0$. If $\mathcal{A}$ is not maximal-monotone, then there exists $[x_1,y_1]\in V\times V^*$ such that $[x_1,y_1]\notin \mathcal{G}(\mathcal{A})$ and
    \begin{align}\label{es7}
    \langle y-y_1,x-x_1\rangle \ge 0, \quad \forall [x,y]\in \mathcal{G}(\mathcal{A}).
    \end{align}
    By assumption, we can choose $[x_2,y_2]\in \mathcal{G}(\mathcal{A})$ such that
    $$
    \lambda J(x_2)+y_2 = \lambda J(x_1)+y_1.
    $$
    Then in view of (\ref{es7}), we can deduce that
    $$
    \langle J(x_2)-J(x_1), x_2-x_1\rangle \le 0.
    $$
    According to the definition of $J$, we obtain
$$
\|x_1\|_V^{\alpha}+\|x_2\|_V^{\alpha}\le \langle J(x_2), x_1\rangle+\langle J(x_1),x_2\rangle\le \|x_2\|^{\alpha-1}\|x_1\|_V+\|x_1\|^{\alpha-1}\|x_2\|_V,
$$
which implies
\begin{align*}
    (\|x_1\|_V^{\alpha-1}-\|x_2\|_V^{\alpha-1})(\|x_1\|_V-\|x_2\|_V)\le 0.
\end{align*}
In light of $\alpha>1$, we deduce that
$$
\langle J(x_2), x_1\rangle=\langle  J(x_1), x_2\rangle=\|x_1\|_V^{\alpha}=\|x_2\|_V^{\alpha}.
$$
Thus, it follows that  $J(x_1)=J(x_2)$. Since the duality mapping $J^{-1}$ of $V^*$ is single-valued, it implies that $x_1=x_2$. Then it is clear that $[x_1,y_1]=[x_2,y_2]\in \mathcal{G}(\mathcal{A})$, which contradicts  to the hypothesis.  Thus, $\mathcal{A}$ is maximal-monotone,

\textbf{Step 2:} Suppose that $\mathcal{A}$ is maximal-monotone. Let $y$ be an arbitrary element of $V^*$ and let $\lambda>0$. Define
$$
B(u) :=\lambda J(u)-y, \quad \forall u\in V.
$$
Since
$$\lim\limits_{\|u\|_V\to\infty}\frac{\langle B(u),u \rangle}{\|u\|_V}\ge \lim\limits_{\|u\|_V\to\infty}(\lambda \|u\|_V^{\alpha-1}-\|y\|_{V^*})=\infty,
$$
we note that $B$ is coercive. From the definition of $J$, it follows that $B$ is monotone, single-valued and hemicontinuous. Then, according to Theorem 2.1 in \cite{barbu2010nonlinear}, there exists $x\in V$ such that
$$ \langle \lambda J (x) -y + v, u-x\rangle \ge 0, \quad \forall [u,v]\in \mathcal{G}(\mathcal{A}).
$$
In view of the maximal monotonicity of $\mathcal{A}$, we deduce that $[x,-\lambda J(x)+y]\in \mathcal{G}(\mathcal{A})$, which implies that $y\in \lambda J(x)+\mathcal{A}(x)$. We complete the proof.
\end{proof}
\begin{lemma}\label{lemmacoe1}
	Let $\alpha>1$ and  $\mathcal{A}:V \to 2^{V^*}$ be a maximal-monotone operator. Let $\mathcal{A}_{\lambda}$ denote its generalized Yosida approximation $($i.e.~$(\ref{Alamd})$$)$. Suppose that there exist $\delta >0$ and $C\in \mathbb{R}$ such that
	\begin{equation}
		\left<v,x\right>\ge \delta \|x\|_V^{\alpha}+C\quad\text{for~all}~x\in \mathcal{D}(\mathcal{A})~\text{and}\ v\in \mathcal{A}(x).\label{lemma 1}
	\end{equation}
	Then, for any $0 <\lambda< \delta^{-1}$, we have$$
	\left<\mathcal{A}_{\lambda}x,x\right>\ge \delta 2^{-\alpha}\|x\|_V^{\alpha}+C\quad\text{for~all}~x\in V.
	$$
\end{lemma}
\begin{proof}
	For any $x\in V$, by the definition of $\mathcal{A}_{\lambda}$ and $J$, we have
	$$
	\left<\mathcal{A}_\lambda x, x-\mathcal{R}_{\lambda}x \right> =\frac{1}{\lambda} \left<J(x-\mathcal{R}_{\lambda}x), x-\mathcal{R}_{\lambda}x \right>=\frac{1}{\lambda}\|x-\mathcal{R}_{\lambda}x\|_V^{\alpha}.
	$$
Since $\mathcal{A}_{\lambda}x\in \mathcal{A}(\mathcal{R}_{\lambda}x)$,  by \eqref{lemma 1}  it follows that
\begin{align*}
\left<\mathcal{A}_\lambda x, x \right>&=\left<\mathcal{A}_\lambda x, \mathcal{R}_{\lambda}x \right>+\frac{1}{\lambda}\|x-\mathcal{R}_{\lambda}x\|_V^{\alpha}\\
&\ge \delta \|\mathcal{R}_\lambda x\|_V^{\alpha}+\frac{1}{\lambda}\|x-\mathcal{R}_{\lambda}x\|_V^{\alpha}+C\\
&=\delta\left(\|\mathcal{R}_{\lambda}x\|_V^{\alpha}+\|x-\mathcal{R}_{\lambda}x\|_V^{\alpha}\right)+(\frac{1}{\lambda}-\delta)\|x-\mathcal{R}_{\lambda}x\|_V^{\alpha}+C.
\end{align*}
Noting that $0<\lambda<\delta^{-1}$, we conclude
	\begin{align*}
		\left<\mathcal{A}_\lambda x, x \right> \ge \delta 2^{-\alpha}\|x\|_V^{\alpha}+C.
	\end{align*}
We complete the proof.
\end{proof}	

%
%
%

\vspace{3mm}


\noindent\textbf{Data availability} Data sharing is not applicable to this article as no datasets were generated or analysed during the current study.

\noindent\textbf{Statements and Declarations} On behalf of all authors, the corresponding author states that there is no conflict of interest.

\end{document}